\documentclass{amsart}

\usepackage[utf8]{inputenc}
\usepackage{amsmath,amsfonts}
\usepackage{amsthm}
\usepackage{thmtools}
\usepackage{amssymb}
\usepackage{graphicx}
\usepackage[shortlabels]{enumitem}
\usepackage{genyoungtabtikz}
\usepackage{ytableau}
\usepackage[toc,page]{appendix}
\usepackage{cancel}
\usepackage[normalem]{ulem}
\usepackage{multicol}
\usepackage[hypertexnames=false]{hyperref}
\usepackage{cleveref}
\usepackage{pgfmath}
\usepackage{tikz}
\usepackage{listofitems}
\usetikzlibrary{calc, backgrounds}

\usetikzlibrary{matrix,  positioning, patterns, backgrounds}
\tikzset{
    every matrix/.append style={
        matrix of math nodes,
        nodes in empty cells,
        inner sep=3pt,
        outer sep=0pt,
        column sep=-\pgflinewidth,
        row sep=-\pgflinewidth,
        },
    mylargenode/.style={
        text centered,
        text width=30pt,
        text height=20pt,
        text depth=10pt,
    },
    mylongnode/.style={
        text centered,
        text width=20pt,
        text height=40pt,
        text depth=30pt,
    },
    mysmallnode/.style={
        text centered,
        text width=20pt,
        text height=20pt,
        text depth=10pt,
    },
    mylonglargenode/.style={
        text centered,
        text width=30pt,
        text height=40pt,
        text depth=30pt,
    },
}

\allowdisplaybreaks

\def\allowhook{tunnel hook}
\def\allowhookfilling{tunnel hook covering}
\def\THC{tunnel hook covering}

\def\IT{\operatorname{IT}}
\def\THCc{content}
\def\GBPR{GBPR}

\def\T{T}
\def\S{S}

\def\h{\mathfrak{h}}

\def\Z{\mathbb{Z}}
\def\c{\tau}

\def\K{\mathbf{K}}

\def\ri{r}
\def\j{q}
\def\p{p}
\def\lp{\left(}
\def\rp{\right)}

\newtheorem{thm}{Theorem}[section]
\newtheorem{lem}[thm]{Lemma}
\newtheorem{cor}[thm]{Corollary}
\newtheorem{prop}[thm]{Proposition}
\newtheorem*{proprestatecycle}{Proposition \ref{prop:permutation-THC-cycle-decomp}}
\newtheorem*{proprestaterows}{Proposition \ref{prop:THC-perm-comp-row-by-row}}
\newtheorem{proc}[thm]{Procedure}

\theoremstyle{definition}
\newtheorem{defn}[thm]{Definition}
\newtheorem{ex}[thm]{Example}

\newtheorem{alg}[thm]{Algorithm}

\definecolor{lightblue}{rgb}{.85,.9,.96}
\definecolor{lightpurple}{rgb}{.9,.5,.9}
\definecolor{lightred}{rgb}{.9,.7,.7}
\definecolor{grey}{rgb}{.85,.85,.85}
\newcommand{\tcred}[1]{\textcolor{red}{#1}}

\def\K{\tilde{K}}
\def\iK{\tilde{K}^{-1}}

\def\pairskkinv{A_{\alpha,\beta}}
\def\pairskkinvone{A_{\alpha,\alpha}}
\def\pairskkinvsymT{B_{\lambda,\mu}}

\def\pairs{C_{\alpha,\beta}}
\def\pairone{C_{\alpha,\alpha}}

\def\sympairs{D_{\lambda,\mu}}
\def\sympairsbig{E_{\lambda,\mu}}

\DeclareMathOperator{\f1}{\phi}
\DeclareMathOperator{\g1}{\psi}
\DeclareMathOperator{\sh}{sh}
\Ylinecolour{black}
\Ynodecolour{black!50!black}

\usepackage[colorinlistoftodos]{todonotes}

\DeclareMathOperator{\I}{\mathfrak{S}}

\DeclareMathOperator{\Sym}{Sym}
\DeclareMathOperator{\SSYT}{SSYT}

\DeclareMathOperator{\NSym}{NSym}

\DeclareMathOperator{\SHF}{THC}
\DeclareMathOperator{\SRT}{SRT}

\DeclareMathOperator{\fl}{flat}

\DeclareMathOperator{\Ri}{\mathcal{R}}
\renewcommand{\sh}{\operatorname{sh}}

\DeclareMathSymbol{\shortminus}{\mathbin}{AMSa}{"39}

\DeclareMathOperator{\sgn}{sgn}
\DeclareMathOperator{\dec}{dec}
\DeclareMathOperator{\perm}{perm}
\DeclareMathOperator{\permSRT}{perm^{\SRT}}
\DeclareMathOperator{\thc}{THC}
\newcommand{\rh}{\mathfrak{g}}

\newcommand{\circl}[1]{\raisebox{.5pt}{\textcircled{\raisebox{-.9pt} {#1}}}}

\AtBeginDocument{%
   \def\MR#1{}
}

\begin{document}

\title[NSym inverse Kostka identity]{A new proof of an E\u{g}ecio\u{g}lu--Remmel inverse Kostka matrix problem via a Garsia--Milne involution involving Sym and NSym}

\author{Edward E. Allen}
\address{Dept. of Mathematics, Wake Forest University, Winston-Salem, NC 27109}
\curraddr{}
\email{allene@wfu.edu}
\thanks{}

\author{Kyle Celano}
\address{Dept. of Mathematics, Wake Forest University, Winston-Salem, NC 27109}
\curraddr{}
\email{celanok@wfu.edu}
\thanks{}

\author{Sarah K. Mason}
\address{Dept. of Mathematics, Wake Forest University, Winston-Salem, NC 27109}
\curraddr{}
\email{masonsk@wfu.edu}
\thanks{SM was partially supported by an AMS-Simons PUI grant.}

\subjclass[2010]{Primary 05E05; Secondary 05A05, 05A19}

\date{}

\dedicatory{}

\begin{abstract}
E\u{g}ecio\u{g}lu and Remmel provide a combinatorial proof (using special rim hook tableaux) that the product of the Kostka matrix $K$ and its inverse $K^{-1}$ equals the identity matrix $I$.  They then pose the problem of proving the reverse identity $K^{-1}K =I$
combinatorially. 
Sagan and Lee prove a special case of this identity using overlapping special rim hook tableaux.  Loehr and Mendes provide a full proof using bijective matrix algebra that relies on the E\u{g}ecio\u{g}lu--Remmel map.  In this article, we solve the problem in full generality independent of the E\u{g}ecio\u{g}lu--Remmel bijection.  To do this, we start by proving NSym versions of both Kostka matrix identities using sign-reversing involutions involving the tunnel hook coverings recently introduced by the first and third authors.  Then we modify our sign-reversing involutions to reduce to Sym.  Finally, we show that our bijection is different than the Loehr and Mendes result by constructing an injective map between special rim tableaux and the symmetric group $S_n.$
\end{abstract}
\maketitle

\section{Introduction}

The complete homogeneous functions $\{ h_{\lambda} \}_{\lambda \vdash n}$ and the Schur functions $\{s_{\lambda}\}_{\lambda \vdash n}$, where $\lambda \vdash n$ indicates that $\lambda$ is a partition of $n$, are bases for the vector space $\Sym$ of symmetric functions.  
The Kostka matrix $K$ is the transition matrix from the complete homogeneous basis to the Schur basis. For $\lambda,\mu\vdash n$, the entry $K_{\lambda,\mu}$ counts the number of semistandard Young tableaux of shape $\lambda$ and content $\mu$ \cite{Mac95}.  The inverse Kostka matrix $K^{-1}$ is the transition matrix from the Schur basis to the complete homogeneous basis whose coefficients are computed via the Jacobi--Trudi formula $s_{\lambda}=\det(h_{\lambda_i-i+j})_{1 \le i,j \le \ell}$~\cite{Mac95}.  E\u{g}ecio\u{g}lu and Remmel~\cite{EgeRem90} interpret these coefficients as \emph{special rim hook tableaux} and use this construction to prove the identity $K K^{-1}=I$ combinatorially, leaving the combinatorial proof of $K^{-1}K=I$ as an open problem. In the following, we abuse notation by calling an involution ``sign-reversing" if it is sign-reversing on all non-fixed points.

\begin{thm}[\cite{EgeRem90}]\label{ERthm}
There exists a sign-reversing involution between certain pairs of semistandard Young tableaux and special rim hook tableaux which shows combinatorially that $\displaystyle{\sum_{\nu\vdash n} K_{\lambda,\nu} K^{-1}_{\nu,\mu} = \delta_{\lambda,\mu}}$, where $\delta_{\lambda,\mu}$ is the Kronecker delta.
\end{thm}

Our first main result is the generalization of \Cref{ERthm} to the vector space $\NSym$ of \emph{noncommutative symmetric functions}.  Here $\K$ is the transition matrix from \emph{complete homogeneous noncommutative symmetric functions} $\{H_\alpha\}_{\alpha \vDash n}$~\cite{GKLLRT95} of degree $n$, where $\alpha\vDash n$ indicates that $\alpha$ is a composition of $n$, to \emph{immaculate functions} $\{\I_\alpha\}_{\alpha\vDash n}$~\cite{BBSSZ14} also of degree $n$. 
Unlike $\Sym$, whose vector space of polynomials of degree $n$ has dimension equal to the number of partitions of $n$,
the dimension of the vector space of polynomials in $\NSym$ of degree $n$ is given by the number of compositions of $n$.
To maintain this distinction between $\Sym$ and $\NSym$, we reserve the symbols $\lambda$, $\mu$,
and $\nu$ for partitions and
$\alpha,$ $\beta,$ and $\gamma$ for compositions and weak compositions. Additionally, $a=(a_1,a_2, \ldots, a_\ell)$ and $b=(b_1,b_2,\ldots,b_\ell)$ are used to indicate integer sequences.

\begin{thm}\label{Th:One}
The sign-reversing involution $\f1$ defined in Section~\ref{section:k-k-inv} provides a combinatorial proof that $\displaystyle\sum_{\gamma\vDash n} \K_{\alpha,\gamma} \iK_{\gamma,\beta}= \delta_{\alpha, \beta}$ for all $\alpha, \beta \vDash n$.
\end{thm}

We then restrict our proof of \Cref{Th:One} to obtain a new proof of the identity $K K^{-1}=I$ using tunnel hook coverings (see Section 2 for their definition) as introduced by the first and third author in \cite{AllenMason25EJC}. See also \cite[Section 7.1]{khanna_local_2025} for another combinatorial proof of $KK^{-1}=I$ by Khanna and Loehr.

On the other hand, Sagan and Lee~\cite{SagLee03} prove the last column (indexed by $(1^n)$) of the identity $K^{-1}K=I$, i.e., $\sum_{\nu \vdash n} K^{-1}_{\lambda,\nu} K^{\textcolor{white}{-1}}_{\nu,(1^n)} = \delta_{\lambda,(1^n)}$. 
Loehr and Mendes~\cite{LoeMen06} use bijective matrix algebra to convert the E\u{g}ecio\u{g}lu--Remmel map into a proof of the $K^{-1}K=I$ identity.  Our proof of $K^{-1}K=I$ begins with a combinatorial proof of the $\NSym$ analogue which we then algorithmically reduce to $\Sym$.

\begin{thm}\label{Th:Two}
The sign-reversing involution $\g1$ defined in Section~\ref{section:k-inv-k} provides a combinatorial proof that $\displaystyle \sum_{\gamma\vDash n} \iK_{\beta,\gamma} \K_{\gamma,\alpha}= \delta_{\beta,\alpha}$ for all $\alpha,\beta\vDash n.$ 
\end{thm}

The $\NSym$ analogues of semistandard Young tableaux and special rim hook tableaux used in this proof are \emph{immaculate tableaux}~\cite{BBSSZ14} and \emph{tunnel hook coverings}~\cite{AllenMason25EJC}, respectively.
We modify our sign-reversing involution $\g1$ from Theorem \ref{Th:Two} so that it applies to pairs consisting of tunnel hook coverings and semistandard Young tableaux.    This then solves the problem posed by E\u{g}ecio\u{g}lu and Remmel~\cite{EgeRem90} in full generality. 

\begin{thm}\label{thm:KinvK-involution}
 The sign-reversing involution $\rho$ defined in \Cref{section: kinv-k-sym} provides a combinatorial proof that $\displaystyle{\sum_{\nu \vdash n} K^{-1}_{\lambda,\nu} K_{\nu,\mu} = \delta_{\lambda,\mu}}$ for all $\lambda,\mu\vdash n$. 
\end{thm}

Our proof of Theorem \ref{thm:KinvK-involution} is an application of the \emph{Garsia--Milne Involution Principle} that was used to give a bijective proof of a Rogers--Ramanujan identity \cite{GarMil81,GarMil81_2}.  For a recent description and history of this important combinatorial tool, see \cite{SZ25}.

In the proof of Theorem \ref{ERthm} \cite[Theorem 1]{EgeRem90}, E\u{g}ecio\u{g}lu and Remmel show that their combinatorial interpretation in terms of special rim hook tableaux satisfies the recurrence relation  underlying the Jacobi--Trudi identity.
In \cite{AllenMason25EJC}, however, the first and third authors give a combinatorial interpretation
of $\iK$ by showing that its corresponding entries are the same as those implied by the noncommutative Jacobi--Trudi formula for immaculate functions
$\I_\alpha$.
Theorems \ref{Th:One} and \ref{Th:Two} show that the combinatorial interpretation of the entries of $\iK$ implies it is the inverse of $\tilde K$ directly (without reference to the Jacobi--Trudi formula).
Therefore this paper gives a proof of the  tunnel hook covering interpretation
of $\iK$ without relying on the Jacobi--Trudi identity.

In fact, \Cref{Th:One} (or \Cref{Th:Two}) together with the bijection (see \Cref{Th:THC_and_Perms}~\cite[Proposition 31]{AllenMason25EJC}) between tunnel hook coverings and the permutations used in evaluating the determinant of the Jacobi--Trudi matrix provide an alternative method for showing that immaculate functions satisfy the Jacobi--Trudi style identity given in~\cite{BBSSZ14}. Moreover, this argument can be further modified to produce a new proof of the classical Jacobi--Trudi identity using tunnel hook coverings of partition shapes.


This article is organized as follows.
In \Cref{sec:background}, we review the construction of important tools including immaculate functions, immaculate tableaux, and tunnel hook coverings. In \Cref{section:k-k-inv}, we give combinatorial proofs of Theorems \ref{Th:One} and \ref{ERthm} via the construction of two sign-reversing
involutions 
$\f1$ and
$\chi$ given by Algorithms \ref{alg:KKinverse-row-labeling} and \ref{alg:KKinverseSym-row-labeling}, respectively.
In \Cref{section:k-inv-k}, we define a sign-reversing involution $\g1$ via \Cref{alg: involution}  which we then use to prove \Cref{Th:Two}.  In \Cref{section: kinv-k-sym}, we define a sign-reversing involution $\theta$ via \Cref{alg:Bad-Involution} and then 
create a sign-reversing involution $\rho$ via \Cref{alg: sym involution FULL} using both $\psi$ and $\theta$ to prove \Cref{thm:KinvK-involution}.  In \Cref{section:connections}, we show that the bijection constructed in this paper is different from that presented in \cite{LoeMen06}.  To do this, we give a novel method to associate a special rim hook tableaux with the permutation that gives the corresponding term from the Jacobi-Trudi determinant.

\section{Background on immaculate functions and NSym Kostka matrices}\label{sec:background} 
A \textit{composition} of length $\ell$ is a sequence $\alpha=(\alpha_1,\alpha_2,\dots,\alpha_\ell)$ of positive integers. We say that $\alpha$ is a composition of $n$, denoted $\alpha\vDash n$, if $\alpha_1+\cdots+\alpha_\ell=n$. A \textit{partition} $\lambda=(\lambda_1,\lambda_2,\dots,\lambda_\ell)$ is a weakly decreasing composition and we write $\lambda \vdash n$ if $\lambda_1+\cdots+\lambda_\ell=n$. A \textit{weak composition} $\gamma=(\gamma_1,\gamma_2,\dots)$ of $n$ is a sequence of nonnegative integers such that $\gamma_1+\gamma_2+ \cdots=n.$  Since there can only be finitely many nonzero parts, we often ignore the trailing zeros and consider our weak compositions to have finitely many parts.

The Ferrers diagram of a composition $\alpha=(\alpha_1,\alpha_2,\ldots,\alpha_\ell)\vDash n$, written in English notation,
consists of a collection of cells $C_\alpha=\{ (i,j): 1 \le i \le \ell, 1 \le j \le \alpha_i\}$ with $(i,j)$ designating the cell in the $i$\textsuperscript{th} row from the top and $j$\textsuperscript{th} column from the left.  A Ferrers diagram of shape $(2,3,4,2)\vDash 11$ is given in Figure \ref{fig:ferrer}.

\begin{figure}
   \centering
$T$\ \  $=$\ \ \begin{minipage}{.3\textwidth}
    
\begin{tikzpicture}[yscale=-1,scale=.75,xscale=1.5]

\def\L1{{2, 3, 4, 2}}
\pgfmathsetmacro{\len}{dim(\L1)}

\foreach \y in {1,...,\len}{
   \pgfmathsetmacro{\j}{\L1[\y-1]}
   \foreach \i in {1,...,\j}{
      \draw (\i-.5,\y-.5) rectangle (\i+.5,\y+.5);
      \node at (\i,\y) {(\y,\i)};
       }
}
\end{tikzpicture}
\end{minipage}
\medskip

$C_\alpha=\{(1,1),(1,2),(2,1),(2,2),(2,3),(3,1),(3,2),(3,3),(3,4),(4,1),(4,2)\}$
\caption{The above figure is a Ferrers diagram of shape $\alpha=(2,3,4,2)$  and the corresponding
collection of cells $C_\alpha.$}\label{fig:ferrer}
\end{figure}

Permutations are usually written in one-line notation; that is, for $\sigma \in S_n$, if $\sigma(i)=\sigma_i$ then we write $\sigma=[\sigma_1,\sigma_2,\ldots,\sigma_n]$.
Transpositions, on the other hand,
are represented by $s_i=(i,i+1)\in S_n$.
We multiply permutations from right to left so that $\sigma\tau=[\sigma_{\tau_1}, \sigma_{\tau_2},\ldots,\sigma_{\tau_n}]$ where $\tau=[\tau_1,\tau_2,\ldots,\tau_n] \in S_n$.

$NSym$ is the algebra generated by the noncommutative elements
$\{ H_i \}_{i \in \mathbb{Z}_{\geq 0}}$ with no relations.  Elements of the set $\{H_i\}_{i \in \mathbb{Z}_{\geq 0}}$ can be thought of as the noncommutative analogues
of the complete homogeneous functions $\{h_i\}_{i \in \mathbb{Z}_{\geq 0}}$ in $\Sym.$  One analogue of the Schur functions in $\NSym$ is the collection of \emph{immaculate functions}~\cite{BBSSZ14}.  This basis is constructed using \emph{noncommutative Berenstein operators}, which are a noncommutative analogue of the Berenstein creation operators used to construct Schur functions.  Immaculate functions can equivalently be defined in terms of a Jacobi--Trudi-like determinant as $\mathfrak{S}_\alpha=\mathfrak{det}(H_{\alpha_i-i+j})_{1\leq i,j\leq \ell}$ for a composition $\alpha$, where $\mathfrak{det}()$ is the noncommutative Laplace expansion of the determinant from row 1 through row $\ell$~\cite[Theorem 3.27]{BBSSZ14}.
As in the $\Sym$ case, we can define the \emph{$\NSym$ Kostka matrix} $\tilde K_{\alpha,\beta}$ for compositions $\alpha$ and $\beta$ by $H_\beta=\sum_{\alpha} \tilde K_{\alpha,\beta} \mathfrak{S}_\alpha$.

For a composition $\alpha=(\alpha_1,\dots,\alpha_\ell)$, an  \emph{immaculate tableau} $S$ of \emph{shape} $\alpha$ \cite{BBSSZ14} is a function $S: C_\alpha \rightarrow \mathbb{Z}_{\ge 1}$ 
such that
$S(i,j) \le S(i,j+1)$ for
$1 \le j \le \alpha_i-1$ and
$S(i,1)<S(i+1,1)$ for $1 \le i \le \ell -1$. The \emph{content} of $S$ is the weak composition $\beta=(\beta_1,\beta_2,\dots)$ where $\beta_k$ is the number of pairs $(i,j)\in C_\alpha$ such that $S(i,j)=k$. 
A \emph{semistandard Young tableau} $T$
of partition shape $\lambda=(\lambda_1,\dots,\lambda_\ell)\vdash n$ is an immaculate tableau of shape $\lambda$ in which we have
$T(i,j)<T(i+1,j)$ for $1 \le i \le \ell -1$
whenever both $(i,j)$ and $(i+1,j)$ are in $C_\lambda.$  Let $\IT_{\alpha,\beta}$ 
and $\SSYT_{\lambda,\mu}$ denote the collections of immaculate tableaux of shape $\alpha$ and content~$\beta$\footnote{The authors of \cite{BBSSZ14} only define immaculate tableaux with composition content. It is natural (and moreover useful for us) to extend the definition to allow for weak composition content. } and  semistandard Young tableaux of shape $\lambda$ and content $\mu$, respectively.

\begin{thm}[{\cite{BBSSZ14}}]\label{thm:itnsmykostka}
For compositions $\alpha,\beta\vDash n$, we have $\K_{\alpha,\beta}=|\IT_{\alpha,\beta}|$.
\end{thm}

A GBPR diagram \cite{AllenMason25EJC} is a generalization of a Ferrers diagram in which
the cells are colored grey, blue, purple, or red.  Let $a=(a_1,a_2,\ldots,a_\ell)$ be an integer sequence and let $\nu=(\nu_1,\nu_2,\ldots,\nu_\ell)$ be a partition.  A GBPR diagram $D_{a/\nu}$ of shape $a/\nu$ is constructed as follows.

\begin{enumerate}
\item For $1 \le i \le \ell$, place $\nu_i$ grey cells in row $i$. 

\begin{enumerate}
\item If $a_i >0$ and $\nu_i \le a_i,$ place $a_i-\nu_i$ blue cells in row $i$ situated immediately to the right of the grey cells. 

\item If $a_i>0$ and $a_i < \nu_i,$ place $ \nu_i-a_i$ red cells in row $i$ situated immediately to the right of the grey cells.

\item If $a_i \le 0,$ place $|a_i|+\nu_i$ red cells  in row $i$ situated immediately to the right of the grey cells.
\end{enumerate}

\item Any cell that is not colored grey, red, or blue is purple. 
\end{enumerate}

At times it will be convenient to ignore the first $r-1$ rows.  The \emph{partial diagram} $D^{(r-1)}_{\mu/\nu^{(r-1)}}$ is obtained by constructing the \GBPR~diagram 
\[D_{(\mu_{r}, \mu_{r+1} , \hdots , \mu_k)/(\nu^{(r-1)}_{r}, \nu^{(r-1)}_{r+1} , \hdots,  \nu^{(r-1)}_k)}\]
and then shifting the resulting diagram down by $r-1$ rows, so that the first nonempty row is row $r$ of the diagram.

A \emph{boundary cell} is a blue, purple, or red cell
in the GBPR diagram
that is vertically, horizontally, or diagonally adjacent to at least one grey cell. A \emph{terminal cell} is a boundary cell whose immediate neighbor on the left is a grey cell. A \emph{tunnel hook} $\h(j,\c_j)$ starting in row $j$ and ending in terminal cell $\c_j=(p_j,q_j)$
includes all blue and red cells in row $j$, one purple cell in row $j$ if there are no blue or red cells in row $j$, and
all of the boundary cells in rows $j+1 \le i \le p_j.$ The \emph{sign} of $\h(j,\c_j)$ is \begin{equation}\sgn(\h(j,\c_j))=(-1)^{p_j-j}.\end{equation}
A {\emph \allowhookfilling~($\SHF$)} of integer sequence shape $a=(a_1,a_2,\hdots,a_\ell)$ 
is constructed using the following procedure.

\begin{proc}[\cite{AllenMason25EJC}]\label{A:THC} 
Consider a sequence $a=(a_1, a_2, \hdots , a_\ell) \in \mathbb{Z}^\ell$.
Start with the GBPR diagram $D_{a/\nu^{(0)}}$, where $\nu^{(0)}=(0,0,\hdots)$,
and repeat the following steps, once for each value of $r$ from $1$ to $\ell$.
\begin{enumerate}

\item 
Choose a \allowhook~$\h(r,\c_r)$ in $D_{a/\nu^{(r-1)}}$.  Set 
\[
\Delta_r=\Delta(\h(r,\c_r))=|\h(r,\c_r)|-2 \rho_r - \g1_r,
\]
where 
$|\h(r,\c_r)|$ is the total number of cells in
$\h(r,\c_r)$
and $\rho_r$ and $\g1_r$ are the number of red cells and purple cells, respectively, in row $r$
that are included in $\h(r,\c_r).$

\item{\label{St:grey}} For each $1 \le i \le \ell$, let $\eta_i^{(r)}$ be the number of cells in row $i$ of $\h(r,\c_r)$ and let $\nu^{(r)}$ be the sequence defined for $1 \le i \le \ell$ by
$\nu_{i}^{(r)}=\nu_{i}^{(r-1)}+\eta_{i}^{(r)}.$

\item{\label{St:gbpr}} Construct the \GBPR~diagram $D_{a/\nu^{(r)}}$. 

\end{enumerate}
\noindent
Let $T=\bigl(\h(1,\c_1), \h(2,\c_2), \ldots, \h(\ell,\c_\ell)\bigr)$ denote the resulting THC and set
$\Delta(T) = (\Delta_1,\Delta_2, \hdots , \Delta_\ell)$. 
\end{proc}

\noindent
The \emph{\THCc} of $T$ is given by $\fl(\Delta(T))$ where $\fl(a)$ is the sequence obtained by removing the zeros from the integer sequence $a=(a_1,a_2,\ldots,a_\ell)$. For a tunnel hook covering $T=\bigl(\h(1,\c_1), \h(2,\c_2), \ldots, \h(\ell,\c_\ell)\bigr)$, define the \textit{sign} of $T$ to be
\begin{equation}\sgn(T)=\prod_{i=1}^\ell \sgn(\h(i,\c_i)).\end{equation}
For compositions $\alpha$ and $\beta$, let $\thc_{\alpha,\beta}$ be the set of tunnel hook coverings of \THCc{} $\alpha$ and shape $\beta$.

\begin{thm}[\cite{AllenMason25EJC}]\label{thm:thcinvkostka}
For compositions $\alpha,\beta$ of a positive integer $n$, we have
\[
\iK_{\alpha,\beta}=\sum_{T\in \thc_{\alpha,\beta}}\sgn(T).
\]
\end{thm}

 An important ingredient in the proof of the above theorem is the association of a unique permutation to each \THC{}~of a diagram.
 First,  the \emph{$d^{th}$ diagonal} $\mathcal{L}_d$
 of a GBPR diagram of shape $\alpha$
 is the set $\mathcal{L}_d=\{(d+k,1+k)\mid k\in \Z_{\ge 0}\}$. 
      For a \THC{} $T=\bigl(\h(1,\c_1), \h(2,\c_2), \ldots, \h(\ell,\c_\ell)\bigr)$, its \textit{permutation} $\sigma\in S_\ell$, denoted $\perm(T)$, is given by setting $\sigma(i)=j$ if $\c_i\in \mathcal{L}_j$
      for each $i\in [\ell].$ See \Cref{fig:THC permutation} for an example.  Recall that the sign $\sgn(\sigma)$ of a permutation $\sigma$ is given by $\sgn(\sigma)=1$ if $\sigma$ is an even permutation and $\sgn(\sigma)=-1$ if $\sigma $ is an odd permutation.

\begin{thm}[{\cite[Proposition 31]{AllenMason25EJC}}]\label{Th:THC_and_Perms}
For each composition $\beta$, the map $T\mapsto \perm(T)$ is a bijection from tunnel hook coverings of shape $\beta$ to permutations of $\ell(\beta)$ such that $\sgn(T)=\sgn(\perm(T))$.
\end{thm}
We use the permutation of a \THC{} to construct the sign-reversing involutions necessary to prove $\K\iK=I$ and $\iK \K=I$ combinatorially. A
useful thing to note is that if a tunnel hook covering $T$ of shape $\beta$ has $\Delta(T)=(\Delta_1,\Delta_2,\dots,\Delta_\ell)$ and $\perm(T)=\sigma$, then
\begin{equation}\label{delta-and-permutation}
    \Delta_i=\beta_i+\sigma_i-i
\end{equation}
for all $i\in [\ell].$ See the proof of Lemma 37 in \cite{AllenMason25EJC} (which uses the notation $\Delta_r=\beta_r-r+j$ with $j=\sigma_r$).

We need the following two lemmas. Recall that for compositions $\alpha$ and $\beta$ of $n$, we say $\alpha$ \emph{dominates} $\beta$, denoted $\alpha\trianglerighteq \beta$, if $\alpha_1+\cdots+\alpha_i\geq \beta_1+\cdots+\beta_i$ for all $i\in [n]$, where we pad the ends of compositions with zeros as necessary to make the comparisons. Further, we say that $\alpha$ is \textit{lexicographically}  greater than or equal to $\beta$, denoted $\alpha \geq_\ell \beta,$ if for the  smallest $i$ such that
$\alpha_i\ne\beta_i$ we have $\alpha_i>\beta_i.$

\begin{lem}[Dominance Lemma for IT {\cite{BBSSZ14}}]\label{lem:dom-lem-IT}
    For $\alpha,\beta\vDash n$, if $\IT_{\alpha,\beta}\neq \emptyset$, then $\alpha \ge_\ell\beta$. Moreover, if $\alpha=\beta$, then $|\IT_{\alpha,\alpha}|=1.$
\end{lem}
\begin{lem}[Dominance Lemma for THC]\label{lem:dom-lem-thc}
    For $\alpha,\beta\vDash n$, if $\thc_{\alpha,\beta}\neq \emptyset$, then $\alpha \trianglerighteq \beta$ and hence $\alpha \ge_\ell \beta$. Moreover, if $\alpha=\beta$, then $\thc_{\alpha,\alpha}$
    is a single element $\{T\}$ 
    with $perm(T)=id$, where $id$ is the identity permutation of $S_{\ell(\alpha)}.$
\end{lem}

\begin{proof}
Let $T\in \thc_{\alpha,\beta}$ with permutation $\sigma=\perm(T)$. Then for each $i\in [\ell(\beta)]$, we have
\begin{equation}\label{eqnalphabeta}
\sum_{j=1}^i \alpha_j\geq \sum_{j=1}^i \Delta_j(T)=\sum_{j=1}^i \beta_j+\sigma_j-j=\sum_{j=1}^i \beta_j+\sum_{j=1}^i (\sigma_j-j)\geq \sum_{j=1}^i \beta_j,
\end{equation}
where the first inequality is because $\alpha=\fl(\Delta(T))$ and final inequality is because $\sigma$ is a bijection $[\ell(\beta)]\to [\ell(\beta)]$. Hence, $\alpha \trianglerighteq \beta$ and thus $\alpha\ge_\ell \beta$.

If $\alpha=\beta$, then \Cref{eqnalphabeta} implies that $$\sum_{j=1}^i \alpha_j = \sum_{j=1}^i \Delta_j(T).$$  A simple induction argument shows that $\alpha_i =\Delta_i(T)$ for all $i$.  Therefore $\alpha_i=\alpha_i+\sigma_i-i$, implying $\sigma_i=i$ for each $i$.
\end{proof}
A partial converse of the above is also useful.
\begin{lem}\label{lem: perm-id-iff-shape-is-content}
Let $T$ be a tunnel hook covering of shape $\alpha\vDash n$. Then the content of $T$ equals the shape of $T$ if and only if $\perm(T)=id$.
\end{lem}
\begin{proof}
One direction is in the proof of \Cref{lem:dom-lem-thc}. For the other direction if $\perm(T)=id$, then $\Delta_i(T)=\alpha_i+i-i=\alpha_i$ for each $i$. Hence, $\fl(\Delta(T))=\alpha.$
\end{proof}

\begin{figure}
   \centering
$T$\ \  $=$\ \ \begin{minipage}{.3\textwidth}
    
\begin{tikzpicture}[yscale=-1,scale=.5]

\def\L{{8, 7, 7, 4}}
\pgfmathsetmacro{\len}{dim(\L)}

\foreach \y in {1,...,\len}{
   \pgfmathsetmacro{\j}{\L[\y-1]}
   \foreach \i in {1,...,\j}{
      \draw (\i-.5,\y-.5) rectangle (\i+.5,\y+.5);
       }
}

\tikzset{every node/.style={inner sep=-4pt,color=blue}}
\node (1i) at (8,1) {1};
\node (2i) at (7,2) {2};
\node (3i) at (7,3) {3};
\node (4i) at (4,4) {4};

\tikzset{every node/.style={inner sep=-4pt,color=red}}
\node (1t) at (1,1) {1};
\node (2t) at (1,3) {3};
\node (3t) at (1,4) {4};
\node (4t) at (3,4) {};
\node at (3+.25,4-.25) { 2};
\begin{scope}[on background layer]
\tikzset{every path/.style={line width = 7pt,color=black,line cap=round,opacity=.15,rounded corners}}
\draw (1i)--(1t);
\draw (2i)--(1,2)--(2t);
\draw (3i)--(2,3)--(2,4)--(3t);
\draw (4i)--(4t);
\end{scope}
\tikzset{every path/.style={line width=.5pt,color=red}}
 \draw (1-.5,2-.5)--($(4t)+(.5,.5)$);
\end{tikzpicture}\end{minipage} 
\(\displaystyle \hspace*{0.15in} \perm(T)=\binom{\textcolor{blue}{1\ 2\ 3\ 4}}{\textcolor{red}{1\ 3\ 4\ 2}}\)
   \caption{The permutation (in 2-line notation) associated to a tunnel hook covering}
   \label{fig:THC permutation}
\end{figure}

\section{A combinatorial proof of Theorem \ref{Th:One}}\label{section:k-k-inv}

In this section, we provide a combinatorial proof that $\K\iK=I$.  For compositions $\alpha,\beta$ of $n$, \Cref{thm:itnsmykostka} and \Cref{thm:thcinvkostka} imply that
\[(\K\iK)_{\alpha,\beta}=\sum_{\gamma\vDash n}\K_{\alpha,\gamma}\iK_{\gamma,\beta}=\sum_{(S,T)\in \pairskkinv}\sgn(T),\]
where $\pairskkinv$ is the set of pairs $(S,T)$ such that
\begin{itemize}
    \item 
    $\S$ is an immaculate tableau of shape $\alpha$,
    \item 
    $\T$ is a \THC{} of shape $\beta$, and
    \item 
    the content of $\S$ is equal to $\Delta(T).$
\end{itemize}
Note that the content of $S$ is $\Delta(T)$ instead of simply the content of $T$ (which is $\fl(\Delta(T))$). 

To prove the matrix identity $\K \iK=I$ combinatorially, we generalize the approach from the $\Sym$ case~\cite{EgeRem90}.  Specifically, we introduce an involution $\f1_{\alpha,\beta}$ on $\pairskkinv$ with the following properties.
\begin{enumerate}
 \item If $\alpha=\beta$, then $\pairskkinvone$ contains exactly one pair $(S,T)$, and this pair is a fixed point under $\f1_{\alpha,\alpha}$ such that $\sgn(T)=1$.
\item If $\alpha\neq \beta,$ then $\f1_{\alpha,\beta}$ is sign-reversing on $\pairskkinv$.
   
\end{enumerate}

The following lemmas are useful in proving the first property.

\begin{lem}\label{lem:diag-pairs-k-kinv-1}
For any composition $\alpha$, $\pairskkinvone$ contains exactly one pair $(S,T)$ and for this pair $(S,T)$, we have $\perm(T)=id$.
\end{lem}
\begin{proof}

Let $(S,T)\in A_{\alpha,\alpha}$. Then $S$ has content $\Delta(T)$. If $\gamma=\fl(\Delta(T))$, we standardize $S$ to construct an immaculate tableau $S'$ with composition content $\gamma$. Hence, \Cref{lem:dom-lem-IT} implies $\alpha\geq_{\ell} \gamma$. Since $\gamma$ is the content of $T$, \Cref{lem:dom-lem-thc} implies $\gamma\geq_{\ell} \alpha$. Thus, $\alpha=\gamma$, implying $T\in \thc_{\alpha,\alpha}$. By \Cref{lem: perm-id-iff-shape-is-content}, $T$ has the identity as its permutation, so $\Delta(T)$ is equal to its content, namely $\alpha$. Hence, $S\in \IT_{\alpha,\alpha}.$ The claim then follows from the second parts of the dominance lemmas. 
\end{proof}

Therefore we may define our involution $\f1_{\alpha,\alpha}$ to fix the sole pair in $\pairskkinvone$ (Step (2) of Algorithm \ref{alg:KKinverse-row-labeling}) and then describe the involution $\f1_{\alpha,\beta}$ on $\pairskkinv$ for which $\alpha\neq \beta$ (Step (3)).  Let $S^i$ be the multiset of all entries in row $i$ of $S$.

\begin{alg}\label{alg:KKinverse-row-labeling}
Suppose $\alpha$ and $\beta$ are 
compositions of $n$ and let $(S,T)\in \pairskkinv$.  Let $\ell=\ell(\alpha)$, let $\sigma=\perm(T)$. We construct another pair $\f1_{\alpha,\beta}(S,T)=(U,V)\in \pairskkinv$ as follows.

\begin{enumerate}
\item For each row $i$, let $\j_i$ be the element of $S^i$ such that $\sigma(\j_i)\geq \sigma(k)$ for all entries $k\in S^i$. Let $m$ be the smallest positive integer such that $\sigma(\j_m)\neq m$, or set $m=\ell$ if there is none.
\item If $m=\ell$, then $\alpha=\beta$ (see Lemma~\ref{lem: qis-give-alpha-and-beta}), so we set $\f1_{\alpha,\alpha}(S,T)=(S,T)$.
\item Otherwise, $\sigma(\j_m)>m$. Then,
\begin{enumerate}
 \item Let $V$ be the unique \THC{} of shape $\beta$ and permutation $\perm(V)=s_{\sigma(\j_m)-1}\sigma$.
             \item Let $U$ be obtained by turning one $\j_m$ in $S^m$ into $\p=\sigma^{-1}(\sigma(\j_m)-1)$ and then sorting the entries of $S^m$ to be weakly increasing.
        \end{enumerate}
    \end{enumerate}
    Now set $\phi(S,T)=\phi_{\alpha,\beta}(S,T)$ if $(S,T)\in A_{\alpha,\beta}.$
\end{alg}

See \Cref{fig:KKinverse-example} for an example.

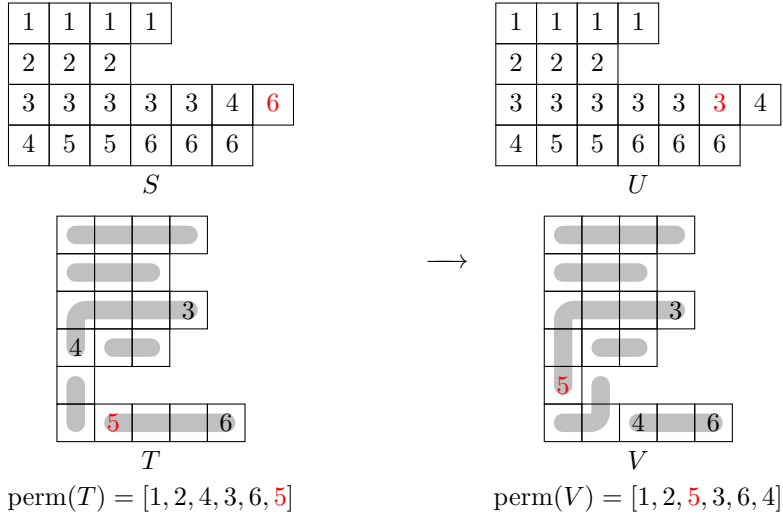
\begin{figure}
    \centering
    $\begin{minipage}{.45\textwidth}\begin{tikzpicture}
            
        \node at (0,.5) {\ytableaushort{1111,222,333334{\tcred{6}},455666,\none\none\none{\none[S]}}};
        \node at (0,-3) {\begin{tikzpicture}[yscale=-1,scale=.5]
\def\L{{4,3,4,3,1,5}}
\pgfmathsetmacro{\len}{dim(\L)}

\foreach \y in {1,...,\len}{
    \pgfmathsetmacro{\j}{\L[\y-1]}
    \foreach \i in {1,...,\j}{
        \draw (\i-.5,\y-.5) rectangle (\i+.5,\y+.5);
        }
}
\node at (4,3) {3};
\node at (3,4) {};
\node at (1,5) {};
\node at (1,4) {4};
\node at (5,6) {$6$};
\node at (2,6) {\tcred{5}};
\node at (3,7) {$T$};
\node at (3,8) {$\perm(T)=[1,2,4,3,6,\tcred{5}]$};
\begin{scope}[on background layer]
\tikzset{every path/.style={line width = 7pt,color=black,line cap=round,opacity=.25,rounded corners}}
\draw (4,1)--(1,1);
\draw (3,2)--(1,2);
\draw (4,3)--(1,3)--(1,4);
\draw (3,4)--(2,4);
\draw (1,5)--(1,6);
\draw (5,6)--(2,6);
\end{scope}

\end{tikzpicture}};\end{tikzpicture} \end{minipage}\xrightarrow[]{\quad} \begin{minipage}{.45\textwidth}\begin{tikzpicture}
    
\node at (0,.5) {$\ytableaushort{1111,222,33333{\tcred{3}}4,455666,\none\none\none{\none[U]}}$};
\node at (0,-3) {\begin{tikzpicture}[yscale=-1,scale=.5]
\def\L{{4,3,4,3,1,5}}
\pgfmathsetmacro{\len}{dim(\L)}

\foreach \y in {1,...,\len}{
    \pgfmathsetmacro{\j}{\L[\y-1]}
    \foreach \i in {1,...,\j}{
        \draw (\i-.5,\y-.5) rectangle (\i+.5,\y+.5);
        }
}
\node at (4,3) {$3$};
\node at (1,5) {\tcred{5}};
\node at (2,4) {};
\node at (1,5) {};
\node at (5,6) {$6$};
\node at (3,6) {4};
\node at (3,7) {$V$};
\node at (3,8) {$\perm(V)=[1,2,\tcred{5},3,6,4]$};
\begin{scope}[on background layer]
\tikzset{every path/.style={line width = 7pt,color=black,line cap=round,opacity=.25,rounded corners}}
\draw (4,1)--(1,1);
\draw (3,2)--(1,2);
\draw (4,3)--(1,3)--(1,5);
\draw (3,4)--(2,4);
\draw (2,5)--(2,6)--(1,6);
\draw (3,6)--(5,6);
\end{scope}

\end{tikzpicture}};\end{tikzpicture}
\end{minipage}
$
    \caption{Using \Cref{alg:KKinverse-row-labeling}, $q_1=1,q_2=2,$ and $q_3=6$.  Thus $m=3$ since $\perm(T)_1=1,\perm(T)_2=2,$ and $\perm(T)_6=5$. Note that row 3 is sorted into weakly increasing order when $q_3=6$ is changed into $p=3$.}
    \label{fig:KKinverse-example}
\end{figure}

\begin{lem}\label{lem: qis-give-alpha-and-beta}
Using the notation in \Cref{alg:KKinverse-row-labeling}, if $\sigma(\j_i)=i$ for all $i\leq M$ for some $M\in [\ell]$, then $\alpha_i=\beta_i$ for all $i\leq M$.
\end{lem}

\begin{proof}
Let $\Delta(T)=(\Delta_1,\dots,\Delta_\ell)$ and induct on $M$. If $M=1$, then $q_1\in S^1$ and $\sigma(q_1)=1$. We must have $\sigma(1)\geq 1$, so 
\[\Delta_1=\beta_1+\sigma_1-1\geq 1.\]
Thus, $S$ has at least one 1 (because $S$ is of content $\Delta(T)$) that therefore must be in row 1. Since $1=\sigma(q_1)\geq \sigma(k)$ for all other $k\in S^1$, we can conclude that $q_1=1$. Thus, $S^1$ has only 1's in it and $\sigma(1)=1$. Hence, $\alpha_1=\Delta_1$ and $\Delta_1=\beta_1$, as desired.
    
Now assume that having $\sigma(q_i)=i$ for all $i\leq M-1$ implies that $\alpha_i=\Delta_i=\beta_i$ for $i\leq M-1$ and further assume that $\sigma(q_{i})=i$ for all $i\leq M$. By the inductive hypothesis, $\alpha_i=\Delta_i=\beta_i$ for all $i\leq M-1$. Since $\Delta_i=\alpha_i$ for $i<M$, the entries in $S^{M}$ must be from the set $\{M,\dots,\ell\}$. Since $\Delta_i=\beta_i$ for $i\leq M-1$, we have $\sigma(i)=i$ for $i\leq M-1$.  Therefore $\sigma(M) \ge M$ and hence $$\Delta_M = \beta_M + \sigma_M-M \ge \beta_M.$$  Therefore $\Delta_M >0$ and hence $S^M$ contains at least one $M$.  

Since $\sigma(q_M)=M$, the only number occurring in $S^{M}$ must be $q_M$ and hence $q_M=M$.  This also implies that $\sigma(M)=M$ and so $\alpha_M = \Delta_M = \beta_M$, as desired.
\end{proof}

For a $\SHF$~$T
$ of shape $\beta=(\beta_1,\dots,\beta_\ell)$, recall from \Cref{delta-and-permutation} that $\Delta(T)=(\Delta_1,\dots,\Delta_\ell)$ satisfies
\begin{equation*}
    \Delta_i=\beta_i+\sigma_i-i
\end{equation*}
for each $i\in [\ell]$, where $\sigma=\perm(T)$. \Cref{thm:changehooklengths}
now describes how the tunnel hooks of a THC $V$ compare to $T$ under the function $\phi_{\alpha,\beta}(S,T)=(U,V)$.

\begin{thm}\label{thm:changehooklengths}
Suppose $\alpha$ and $\beta$ are compositions of $n$ and $\phi_{\alpha,\beta}$ is the function defined in \Cref{alg:KKinverse-row-labeling}. If $(S,T) \in A_{\alpha,\alpha}$ then $\phi_{\alpha,\alpha}(S,T)=(S,T)$.  If $\alpha$ and $\beta$ are distinct compositions, then  $\phi_{\alpha,\beta}$ is a sign-reversing involution on $\pairskkinv$.  Moreover, let $\f1_{\alpha,\beta}(S,T) = (U,V)$.  If $\p$ and $\j_m$ are chosen as in the description of \Cref{alg:KKinverse-row-labeling}, we have
    \[\Delta_r(V)=\begin{cases}
    \Delta_{\p}(T)+1&\textrm{if } r=\p,\\
    \Delta_{\j_m}(T)-1&\textrm{if } r=\j_m,\\
    \Delta_r(T)&\textrm{if } r\not\in \{\p,\j_m\},
    \end{cases}\]  for all $1 \le r \le \ell(\beta)$.
\end{thm}

\begin{proof}
Let $m$ be as in the statement of \Cref{alg:KKinverse-row-labeling}
and let $\sigma=\perm(T)$. We first claim that $m=\ell$ if and only if $\alpha=\beta$. One direction is given by \Cref{lem: qis-give-alpha-and-beta}. The other direction follows from \Cref{lem:diag-pairs-k-kinv-1}. Specifically, in the proof we conclude that $T\in \thc_{\alpha,\alpha}$ and $S\in \IT_{\alpha,\alpha}$. Hence, for all $i\in [\ell]$, $S$ only has $i$'s in row $i$ so $q_i=i$ and $T$ has the identity as its permutation, so $\sigma(q_i)=i$.


We may now assume that $\alpha\neq \beta$ and $m<\ell$. By \Cref{lem: qis-give-alpha-and-beta}, $S^i$ is comprised of only $i$'s for each $i<m$. In particular, the first column of $S$ has the elements $1,2,\dots,m-1$ in order from top to bottom. Hence, rows $m,\dots,\ell$ of $S$ contain only the numbers $m,\dots,\ell$ because $S$ is immaculate. In particular, $m$ only appears in $S^m$ (but there could be other entries in $S^m$). Also, there must be at least one $m$ in $S^m$.  If not, then  $\Delta_m=0$, so at least one tunnel hook originating in an earlier row of $T$ must include a cell in row $m$.  However, this is not possible since $\Delta_i=\beta_i$ for all $1 \le i < m$.

Since $\sigma(\j_m)\neq m$, we must have $\sigma(\j_m)>m$. Thus, $V$ (the tunnel hook covering of shape $\beta$ and permutation $\perm(V)=s_{\sigma(\j_m)-1}\sigma$) is well-defined with $\sgn(V)=\sgn(s_{\sigma(\j_m)-1}\sigma)=-\sgn(\sigma)=-\sgn(T).$ 
    
We claim that $U$ as constructed in Step (3b) is an immaculate tableau. If $\j_m=m$, then $\sigma(\j_m)>m$ implies that $\Delta_m=\beta_m+\sigma_m-m > \beta_m \ge 1$.  Thus there are at least 2 $m$'s in row $m$ of $S$. If $\j_m>m$, then $\j_m$ appears to the right of column $1$ in $S^m$ since there must be at least one $m$ in $S^m$. Hence in either case, there is an occurrence of $\j_m$ in $S^m$ that is not the leftmost element of $S^m$. Thus, the set of entries in the leftmost column of $U$ is identical to the set of entries in the leftmost column of $S$. Since row $m$ of $U$ is increasing and the rest of the rows are identical to those in $S$, $U$ is an immaculate tableau.

For the statement regarding $\Delta(V)$, recall that $\Delta_{k}(V)=\beta_k-k+\perm(V)_k$ for each $k$.  If $\p=\sigma^{-1}(\sigma(\j_m)-1)$, then
\[\Delta_{\j_m}(V)=\beta_{\j_m}-{\j_m}+\perm(V)_{\j_m}=\beta_{\j_m}-{\j_m}+\perm(T)_{\j_m}-1=\Delta_{\j_m}(T)-1,\]
where the second equality is because $\perm(V)=s_{\sigma(\j_m)-1}\perm(T)$. Further, we have
\[\Delta_{\p}(V)=\beta_{\p}-\p+\perm(V)_{\p}=\beta_{\p}-\p+\perm(T)_{\p}+1= \Delta_{\p}(T)+1.\]

Since no other entries in the permutation are altered by applying the transposition $s_{\sigma(\j_m)-1}$, we have $\Delta_k(V)=\Delta_k(T)$ for all $k \not= \p,\j_m$, as desired.

Finally, we must show that $\phi_{\alpha,\beta}$ is an involution.  We need to prove that if $\f1_{\alpha,\beta}(S,T)=(U,V)$ then $\phi_{\alpha,\beta}(U,V)=(S,T)$. Let $\tau=\perm(V)=s_{\sigma(\j_m)-1}\sigma$ and for each $i$, set $q'_i=q_i$ in Step (1) of the algorithm for $(U,V)$ (to distinguish them from the values chosen through $\f1_{\alpha,\beta}(S,T)$). Since $\sigma(\j_m)\geq \sigma(k)$ for all $k\in S^m$ and $p=\sigma^{-1}(\sigma(\j_m)-1)$ is in row $m$ of $U$ by construction, we have $\tau(p)=\sigma(\j_m)\geq \tau(k)$ for any other $k$ appearing in row $m$ of $U$. Thus, $q_m'=p$. Since the other rows are unchanged, we have $q'_i=q_i$ for $i<m$.  Hence, $m$ is still the smallest index such that $\tau(q'_m)\neq m$.  Since $\tau^{-1}(\tau(q'_m)-1)=\sigma^{-1}s_{\sigma(\j_m)-1}(\sigma(\j_m)-1)=q_m$, Step (3b) of \Cref{alg:KKinverse-row-labeling} converts a $p$ in row $m$ of $U$ to $q_m$ and reorders the entries, thus producing the immaculate tableau $S$ as desired.  
Since $q_m'=p=\sigma^{-1} ( \sigma(q_m)-1)$, applying $s_{\tau(q_m')-1}$ to $\tau$ produces the permutation $s_{\tau(p)-1} s_{\sigma(q_m)-1} \sigma = \sigma$. Hence, the tunnel hook covering produced in Step (3a) has shape $\beta$ and permutation $\sigma$, which must be $T$. Therefore, the result of applying $\phi_{\alpha,\beta}$ to $(U,V)$ is $(S,T)$, as desired.
\end{proof}

We use \Cref{alg:KKinverse-row-labeling} to provide a new combinatorial proof of the classical Kostka matrix identity $KK^{-1}=I$. 
 From \cite{AllenMason25EJC}, for each $\lambda,\mu\vdash n$ we have
\begin{equation}\label{eq:kostka-in-terms-thc}K^{-1}_{\lambda,\mu}=\sum_{\substack{\alpha\vDash n\\\dec(\alpha)=\lambda}}\sum_{T\in \thc_{\alpha,\mu}}\sgn(T),\end{equation}
where $\dec(\alpha)$ is the partition obtained by writing the composition $\alpha$ in weakly decreasing order. 

As in~\cite{EgeRem90}, we use a variation on the pairs to make the proofs more straightforward.  In particular, we make use of the fact that $K_{\lambda,\alpha}=K_{\lambda,\mu}$ for any rearrangement $\alpha$ of $\mu$.

Let $\pairskkinvsymT$ be the set of pairs $(S,T)$ where
\begin{itemize}
\item $S$ is a semistandard Young tableau of shape $\lambda$,  
    \item 
    $T$ is a \THC{} of shape $\mu$, and
    \item the content of $S$ is the weak composition  \begin{equation}\label{Eq:contentS}\Delta_{\sigma^{-1}}(T):=(\Delta_{\sigma^{-1}(1)}(T),\Delta_{\sigma^{-1}(2)}(T),\dots,\Delta_{\sigma^{-1}(\ell(\mu))}(T)),\end{equation} where $\sigma=\perm(T)$.
\end{itemize} 
Note that if $(S,T)\in B_{\lambda,\mu}$, then in particular $\Delta(T)$ has no negative entries.
\begin{lem} For partitions $\lambda,\mu\vdash n,$
\[(KK^{-1})_{\lambda,\mu}=\sum_{(S,T)\in \pairskkinvsymT}\sgn(T).\]
\end{lem}

\begin{proof}
Suppose $\lambda,\mu\vdash n$. Then,
\begin{align*}
(KK^{-1})_{\lambda,\mu}&=\sum_{\gamma\vdash n}K_{\lambda,\gamma}K_{\gamma,\mu}^{-1}\\
    &=\sum_{\gamma\vdash n}\lp \sum_{S\in SSYT_{\lambda,\gamma}}1 \rp \lp \sum_{\substack{\alpha\vDash n\\dec(\alpha)=\gamma}}\sum_{T\in THC_{\alpha,\mu}}\sgn(T)\rp \\
    &=
    \sum_{\alpha\vDash n}\lp \sum_{T\in THC_{\alpha,\mu}} \sgn(T) \rp
    \lp \sum_{S\in SSYT_{\lambda,\alpha}}1 \rp\\
    &=
    \sum_{\alpha\vDash n}\lp \sum_{T\in THC_{\alpha,\mu}} \sgn(T)  \sum_{S\in SSYT_{\lambda,\Delta_{\sigma^{-1}}(T)}}1 \rp\\
    &=
    \sum_{(S,T)\in \pairskkinvsymT}\sgn(T)
\end{align*}
 as desired.
\end{proof}

We define a map $\chi_{\lambda,\mu}:\pairskkinvsymT\to \pairskkinvsymT$ which we will prove is a sign-reversing involution whenever $\lambda \not= \mu$ and the identity map otherwise.  Recall that $S^i$ is the multiset of all entries in row $i$ of $S$.

\begin{alg}\label{alg:KKinverseSym-row-labeling}
        Suppose $\lambda$ and $\mu$ are 
    partitions of $n$ and let $(S,T)\in \pairskkinvsymT$. Let $\sigma=\perm(T)$.  We construct another pair $\chi_{\lambda,\mu}(S,T)=(U,V)\in \pairskkinvsymT$ as follows.

\begin{enumerate}
\item For each $i$, let $\j_i$ be the largest element of $S^i$. Let $m$ be the smallest value in such that $\j_m\neq m$ or set $m=\ell(\lambda)$ if there is none.
\item If $m=\ell(\lambda)$, then $\lambda=\mu$ (see  \Cref{lem: qis-give-alpha-and-beta-Sym}), so we set $\chi_{\lambda,\lambda}(S,T)=(S,T)$.
\item Otherwise, $\j_m>m$. 
\begin{enumerate}
 \item Let $V$ be the unique \THC{} of shape $\mu$ and permutation $\perm(V)=s_{\j_m-1}\sigma$.
    \item Replace the leftmost occurrence of $q_m$ in row $m$ of $S$ with $q_m-1$ to create $S'$. Let $U$ to be the semistandard Young tableau obtained by applying the standard Bender--Knuth involution swapping the number of occurrences of $q_m-1$ and $q_m$ in $S'$  (e.g. \cite[Theorem 7.10.2]{Sta99}).
        \end{enumerate}
    \end{enumerate}
    Now set $\chi(S,T)=\chi_{\lambda,\mu}(S,T)$ if $(S,T)\in B_{\lambda,\mu}.$
\end{alg}

See \Cref{fig:KKinverseSYM-example} for an example.

\begin{figure}
    \centering
    $\begin{minipage}{.45\textwidth}\begin{tikzpicture}
            
        \node at (0,.5) {\ytableaushort{111111,22222,4444{\tcred{5}},5566,\none\none\none{\none[S]}}};
        \node at (0,-3) {\begin{tikzpicture}[yscale=-1,scale=.5]
\def\L{{6,5,3,2,2,2}}
\pgfmathsetmacro{\len}{dim(\L)}

\foreach \y in {1,...,\len}{
    \pgfmathsetmacro{\j}{\L[\y-1]}
    \foreach \i in {1,...,\j}{
        \draw (\i-.5,\y-.5) rectangle (\i+.5,\y+.5);
        }
}
\node at (1,5) {\tcred{5}};
\node at (3,8) {$\perm(T)=[1,2,4,\tcred{5},3,6]$};
\node at (3,7) {$T$};
\begin{scope}[on background layer]
\tikzset{every path/.style={line width = 7pt,color=black,line cap=round,opacity=.25,rounded corners}}
\draw (6,1)--(1,1);
\draw (5,2)--(1,2);
\draw (3,3)--(1,3)--(1,4);
\draw (2,4)--(2,5)--(1,5);
\draw (2,6)--(1,6);
\end{scope}

\end{tikzpicture}};\end{tikzpicture} \end{minipage}\xrightarrow[]{\quad} \begin{minipage}{.45\textwidth}\begin{tikzpicture}
            
        \node at (0,.5) {\ytableaushort{111111,22222,4455{\tcred{5}},5566,\none\none{\none[U]}}};
        \node at (0,-3) {\begin{tikzpicture}[yscale=-1,scale=.5]
\def\L{{6,5,3,2,2,2}}
\pgfmathsetmacro{\len}{dim(\L)}

\foreach \y in {1,...,\len}{
    \pgfmathsetmacro{\j}{\L[\y-1]}
    \foreach \i in {1,...,\j}{
        \draw (\i-.5,\y-.5) rectangle (\i+.5,\y+.5);
        }
}
\node at (1,5) {\tcred{5}};
\node at (3,8) {$\perm(T)=[1,2,\tcred{5},4,3,6]$};
\node at (3,7) {$V$};
\begin{scope}[on background layer]
\tikzset{every path/.style={line width = 7pt,color=black,line cap=round,opacity=.25,rounded corners}}
\draw (6,1)--(1,1);
\draw (5,2)--(1,2);
\draw (3,3)--(1,3)--(1,5);
\draw (2,4)--(2,5);
\draw (2,6)--(1,6);
\end{scope}

\end{tikzpicture}};\end{tikzpicture}
\end{minipage}
$
    \caption{Using \Cref{alg:KKinverseSym-row-labeling}, $q_1=1,q_2=2,$ and $q_3=5$.  Thus $m=3$ since $q_3 \not= 3$.  Note we first change the $5$ in row 3 to a 4 and then apply Bender--Knuth, resulting in each 4 changing to 5, except the first 2 which are paired. }
    \label{fig:KKinverseSYM-example}
\end{figure}
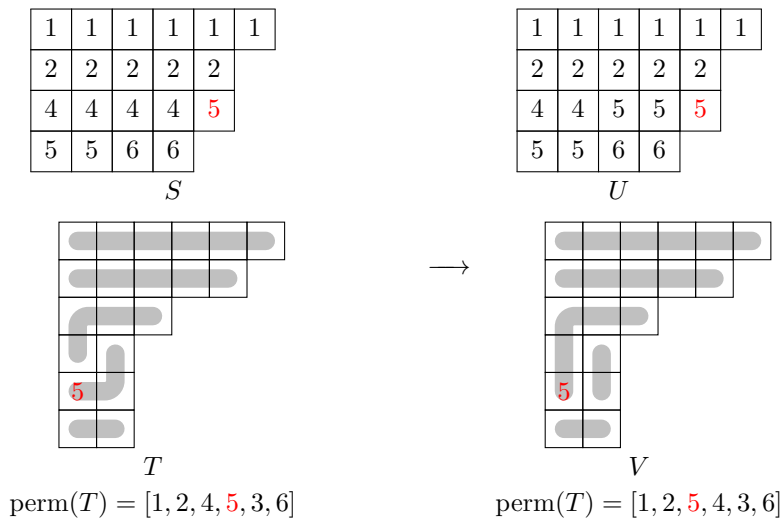

\begin{lem} \label{lem: qis-give-alpha-and-beta-Sym} Using the notation in \Cref{alg:KKinverseSym-row-labeling}, if $\j_i=i$ for all $i\leq \ell(\lambda)$, then $\lambda=\mu$. \end{lem}

\begin{proof}
Recall that both $\lambda$ and $\mu$ are partitions and we are considering $\pairskkinvsymT$
where the content of $S$ is $\Delta_{\sigma^{-1}}(T)$ as in \Cref{Eq:contentS}. Let $\ell=\ell(\lambda)$ and assume $\j_i=i$ for all $i \leq \ell$. Then no entries other than $i$ appear in row $i$ of $S$ for all $i \le \ell$ and the content of $S$ equals the shape of $S$. So $\Delta_{\sigma^{-1}(i)}(T) = \lambda_i$ for all $i \in [\ell]$.  

Since $\lambda$ is a partition, $\lambda_1 \ge \lambda_2 \ge \cdots \ge \lambda_{\ell}$.  This implies \begin{equation}\label{eq:Deltas-decreasing}\Delta_{\sigma^{-1}(1)}(T) \ge \Delta_{\sigma^{-1}(2)}(T) \ge \cdots \ge \Delta_{\sigma^{-1}(\ell)}(T).\end{equation}  
By \Cref{delta-and-permutation}, we have 
\[\Delta_{1}(T)=\mu_1+\sigma_1-1\quad \text{and}\quad \Delta_{\sigma^{-1}(1)}(T)=\mu_{\sigma^{-1}(1)}+\sigma(\sigma^{-1}(1))-1=\mu_{\sigma^{-1}(1)}. \]
If $\sigma_1>1$, then $\Delta_1(T)>\mu_1\geq \mu_{\sigma^{-1}(1)}= \Delta_{\sigma^{-1}(1)}(T)$, which contradicts \Cref{eq:Deltas-decreasing}. Hence, $\sigma_1=1$, and so $\mu_1=\Delta_1(T)=\lambda_1$.

A similar calculation shows that in fact $\sigma^{-1}(i)=i$ for all $1 \le i \le \ell$.  Therefore $\sigma$ is the identity permutation and thus the shape of $T$ is equal to its content.  Therefore $\lambda = \mu$.
\end{proof}


Before we show that $\chi_{\lambda,\mu}$ is a sign-reversing involution for $\lambda \not= \mu$, we must prove that this map is well-defined.  

\begin{lem}
For all partitions $\lambda \not= \mu$, the map $\chi_{\lambda,\mu}$ is a well-defined map from $\pairskkinvsymT$ to $\pairskkinvsymT$.
\end{lem}

\begin{proof}
Consider a pair $(S,T) \in \pairskkinvsymT$ and let $(U,V)=\chi_{\lambda,\mu}(S,T)$. Let  $\alpha$ be the content of $S$, $\perm(T)=\sigma$, and $\ell=\ell(\mu)$.  Since $\sigma$ is a permutation of $[\ell]$, we must have $q_m \le \ell$.  Therefore $q_m-1 < \ell(\sigma)$.  So $s_{q_m-1} \sigma$ is of the same length as $\sigma$, meaning $V$ is a well-defined tunnel hook covering by \Cref{Th:THC_and_Perms}.

By construction, $U$ has the same shape as $S$, so $U$ is of partition shape.  We must prove that converting the leftmost occurrence of $q_m$ in row $m$ of $S$ into a $q_m-1$ produces a semistandard Young tableau $S'$; if so, then $U$ is a semistandard Young tableau since applying the Bender--Knuth involution to a semistandard Young tableau produces a semistandard Young tableau.  The rows of $S'$ remain weakly increasing (converting the leftmost occurrence of $q_m$ to $q_m-1$ produces a weakly increasing sequence), so we must show that the columns of $S'$ are strictly increasing.   By construction of $m$, we know the entry immediately above $q_m$ must be $m-1$, and since $q_m>m$ we must have $q_m-1 > m-1$.  Therefore, the columns of $S'$ are also strictly increasing and so $S'$ is a semistandard Young tableau.

Next we must show that the content of $U$ is equal to \[(\Delta_{\sigma^{-1}(1)}(V), \Delta_{\sigma^{-1}(2)}(V),, \hdots , \Delta_{\sigma^{-1}(\ell)}(V)).\]
To see this, first note that $\perm(V) = s_{q_m-1} \sigma$.  Recall (see \Cref{delta-and-permutation}) that $\Delta_i=\mu_i+\sigma_i-i$.  Let $\sigma^{-1}(q_m-1)=x$ and $\sigma^{-1}(q_m)=y$.  Then $\Delta_{\sigma^{-1}(q_m)}(T) = \mu_y + q_m-y$ and $\Delta_{\sigma^{-1}(q_m-1)}(T) = \mu_x + q_m-1-x$.  If $\alpha=(\alpha_1, \hdots , \alpha_{\ell})$ is the content of $S$, we have \[\alpha_{q_m-1}=\mu_x+q_m-1-x \; \textrm{ and } \; \alpha_{q_m}=\mu_y+q_m-y.\]  Replacing one $q_m$ with $q_m-1$ and applying the Bender--Knuth involution to $S'$ results in a semistandard Young tableau with $q_m-1$ appearing $\alpha_{q_m}-1$ times, $q_m$ appearing $\alpha_{q_m-1}+1$ times, and $i$ appearing $\alpha_i$ times for all other $i$.  Thus, the content of $U$ is $\alpha'=(\alpha_1',\dots,\alpha_\ell')$, where 
\[\alpha_i'=\begin{cases}
\mu_y+q_m-y-1& \textrm{if } i=q_m-1,\\
 \mu_x + q_m-x& \textrm{if } i=q_m,\\
 \alpha_i&\textrm{otherwise.}
\end{cases}\]

On the other hand, when we apply $s_{q_m-1}$ to $\sigma$ (to obtain $\perm(V)$), we have $(s_{q_m-1}\sigma)^{-1}(q_m)=x$ and $(s_{q_m-1}\sigma)^{-1}(q_m-1)=y.$  Then, we have
\[\Delta_{(s_{q_m-1}\sigma)^{-1}(i)}(V)=\begin{cases}
\mu_y+q_m-y-1&\textrm{if } i=q_m-1,\\
 \mu_x + q_m-x&\textrm{if } i=q_m,\\
 \alpha_i&\textrm{otherwise.}
\end{cases}\]

Since the content is unchanged for values other than $q_m$ and $q_{m}-1$ and the permutation is only modified by the transposition $s_{q_m-1}$, this implies that the content of $U$ is indeed equal to $(\Delta_{(s_{q_m-1}\sigma)^{-1}(1)}(V), \hdots , \Delta_{(s_{q_m-1}\sigma)^{-1}(\ell)}(V))$ and the pair $(U,V)$ is an element of $\pairskkinvsymT$.
\end{proof}

We are now ready to prove that $\chi_{\lambda,\mu}$ is a sign-reversing involution.

\begin{thm}
For all $\lambda,\mu\vdash n$ such that $\lambda\neq \mu$, $\chi_{\lambda,\mu}$ is a sign-reversing involution.
\end{thm}

\begin{proof}
The fact that $\chi_{\lambda,\mu}$ is sign-reversing is immediate from construction.  To see that it is indeed an involution, let $\chi_{\lambda,\mu}(S,T) = (U,V)$ and suppose we apply $\chi_{\lambda,\mu}$ to $(U,V)$.  Note that for $1 \le i \le m-1$, $i$ is the only entry appearing in row $i$.  Let the first $j+1$ entries in row $m$ of $S$ be $a_1 \le a_2 \le \cdots \le a_j < q_m$.  Step (3b) first converts the entry in column $j+1$ and row $m$ of $S$ into $q_m-1$.  Since the columns of $S$ are strictly increasing, $S(m+1,j+1) > q_m$.  Furthermore, $S(m+1,k) > q_m$ for $k > j+1$ because the rows of $S$ are weakly increasing.  Therefore none of the entries to the right of $a_j$ are paired during the Bender--Knuth portion of Step (3b) applied to $S'$.  Since the largest entry in $S^m$ is $q_m$ and there is at least one $q_m-1$ in $S'$, the largest entry in $U^m$ must be $q_m$.  The entries in rows $1$ through $m-1$ of $U$ are the same as those in $S$.  Therefore, row $m$ is the row selected in Step (1) of Algorithm~\ref{alg:KKinverseSym-row-labeling} applied to $(U,V)$ and the permutation of the resulting tunnel hook is $s_{q_m-1} s_{q_m-1} \sigma = \sigma$.  Hence, the resulting tunnel hook covering is $T$ by \Cref{Th:THC_and_Perms}.

To see that the resulting semistandard Young tableau is $S$, note that since the Bender--Knuth involution sends segments of paired entries of $S$ to segments of paired entries of $U$ and segments of unpaired entries of $S$ to segments of unpaired entries of $U$, we only need to check the unpaired entries equal to $q_m-1$ and $q_m$ in row $m$.  In $S$, let $x$ be the number of unpaired occurrences of $q_m-1$ in row $m$ and let $y$ be the number of unpaired occurrences of $q_m$ in row $m$.  (Note that $x$ might be zero but $y \ge 1$.)  The first part of Step (3b) applied to $S$ produces a tableau $S'$ with $x+1$ unpaired occurrences of $q_m-1$ and $y-1$ unpaired occurrences of $q_m$ since the affected entry $q_m$ is unpaired in $S$ and its replacement $q_m-1$ is unpaired in $S'$.  Applying the Bender--Knuth involution results in $U^m$ containing $y-1$ unpaired occurrences of $q_m-1$ and $x+1$ unpaired occurrences of $q_m$. On the other hand, applying the algorithm $\chi_{\lambda,\mu}$ to $U$ first produces a tableau $U'$ with $y$ unpaired occurrences of $q_m-1$ and $x$ unpaired occurrences of $q_m$ in row $m$.  Then applying the Bender--Knuth algorithm results in a tableau whose $m^{th}$ row contains $x$ unpaired copies of $q_m-1$ and $y$ unpaired copies of $q_m$, as in $S$.  The paired entries remain paired as in $S$ and thus the resulting semistandard Young tableau is $S$.  Therefore the map is an involution. \qedhere

\end{proof}

\section{A combinatorial proof of Theorem~\ref{Th:Two}}
\label{section:k-inv-k}

The goal of this section is to prove that $(\iK \K)_{\alpha,\beta}=\delta_{\alpha,\beta}$ for all compositions $\alpha,\beta$ of $n$. The combinatorial interpretations of these matrices imply that
\begin{equation}\label{Eq:KinvK}(\iK \K)_{\alpha,\beta}=\sum_{\gamma\vDash n}\iK_{\alpha,\gamma}\K_{\gamma,\beta}=\sum_{(T,S)\in \pairs}\sgn(T),
\end{equation}
where $\pairs$ is the set of pairs $(T,S)$ such that
\begin{itemize}
   \item 
    $\T$ is a 
    THC with content $\alpha$,
   \item 
    $\S$ is an immaculate tableau of content $\beta$, and
   \item 
$\sh(\S)=\sh(T)$.  
\end{itemize}
We prove that $(\iK \K)_{\alpha,\beta}=\delta_{\alpha,\beta}$ for all compositions $\alpha,\beta$ of $n$ by defining an involution $\g1_{\alpha,\beta}$ on $\pairs$ with the following properties.
\begin{enumerate}
\item If $\alpha=\beta$, then $C_{\alpha,\alpha}$ contains exactly one pair $(T,S)$ and this pair is a fixed point under $\g1_{\alpha,\alpha}$ such that $\sgn(T)=1$.
\item 
If $\alpha\neq \beta,$ then $\g1_{\alpha,\beta}$ is a sign-reversing involution.
\end{enumerate}
We start with the first point. The proof of  \Cref{lem:samecomp} is the same as that
given in 
\Cref{lem:diag-pairs-k-kinv-1} \textit{mutatis mutandis}.

\begin{lem}\label{lem:samecomp}
For any composition $\alpha$, $\pairone$ contains exactly one pair $(T,S)$ and for this pair, we have $\perm(T)=id$.
\end{lem}

\begin{alg}\label{alg: involution} 
Suppose $\alpha$ and $\beta$ are compositions of $n$. Let $(T,S)\in \pairs$ and set $\sigma=\perm(T)$. Assume that $T$ and $S$ have common shape $\gamma=(\gamma_1,\gamma_2,\dots,\gamma_\ell)$ and let $M$ be the maximum entry in $S$. We construct another pair $\g1_{\alpha,\beta}(T,S)=(V,U)\in \pairs$ as follows.

\begin{enumerate}
\item Let $k$ be the smallest nonnegative integer less than $\ell$ such that the following conditions are true for all $i\in \{k+1,\dots,\ell\}$. 

\begin{enumerate}
\item $\sigma_i=i.$ 
\item Row $i$ of $S$ contains no entries other than $M-\ell+i$.
\item $M-\ell+i$ only appears in row $i$ of $S$.
\end{enumerate}

\noindent If no such $k$ exists, set $k = \ell$.

\item If $k=0$, then $\alpha=\beta$ and so we set $\g1_{\alpha,\alpha}(T,S)=(T,S)$.

\item Otherwise $k\in [\ell]$.  Set $$\Ri = \{ \sigma_j \; | \; M-\ell+k \textrm{ appears in row } j \; \mathrm{ of } \; S \}.$$ Select $r$ such that $\sigma_r$ is minimal in $\Ri$.

\item If $\sigma_r=k$, first delete the rightmost entry in row $r$ of $S$
and then add a row (row $p=\sigma^{-1}(k+1)=k+1$) consisting of a single $M-\ell+k$ between rows $k$ and $k+1$ of $S$ to create $U$.  Embed $\sigma$ into $S_{\ell+1}$ in the obvious way and then let $V$ be the THC of shape $\sh(U)$ with $\perm(V)\in S_{\ell+1}$ defined by $\perm(V)=s_k \sigma$.

\item If $\sigma_r \neq k$, delete the last entry in row $\ri$ of $S$ and set $\j=\sigma_r$.  Then append a cell with value $M-\ell+k$ to the end of row $\p=\sigma^{-1}(\j+1)$ to create $U$.  If row $r$ of $U$ is empty (which happens when $r=k$ and $\gamma_k=1$), shift the rows $r+1, \hdots , \ell$ up to fill the rows $r, \hdots , \ell-1$.  Finally let $V$ be the THC of shape $\sh(U)$ with $\perm(V)=s_{\j}\sigma$.

\end{enumerate}
 Now set $\psi(T,S)=\psi_{\alpha,\beta}(T,S)$ if $(T,S)\in C_{\alpha,\beta}.$
 \end{alg}
See Examples \ref{ex:kinvk} and \ref{ex:kinvk2}.  We first show that $\g1_{\alpha,\beta}$ is a well-defined function.
 
\begin{lem}{\label{lem:welldefinedsec4}}
Let $\alpha$ and $\beta$ be compositions of $n$. If $(T,S)\in C_{\alpha,\beta},$ then $\g1_{\alpha,\beta}(T,S)\in C_{\alpha,\beta}$.
\end{lem}

\begin{proof}
Assume first that $\alpha = \beta$.  Lemma~\ref{lem:samecomp} states that there is only one pair $(T,S) \in C_{\alpha,\alpha}$ and for this pair, $\perm(T)=id$.  Therefore Condition (1a) ($\sigma_i=i$) is satisfied by every integer $1 \le i \le \ell$ and thus $\alpha=\gamma$.  Hence $\gamma = \alpha = \beta$.  Since the content of $S$ equals the shape of $S$, Conditions (b) and (c) of Step (1) are also satisfied for all $1 \le i \le \ell$.  Therefore $k=0$ and hence $\g1_{\alpha,\beta}$ fixes $(T,S)$.

Next assume that $\alpha \not= \beta$. In both cases (Steps (4) and (5)), we delete the last entry in row $r$.  Since $M-\ell+k$ appears in row $r$ by assumption and $M-\ell+k$ is the largest entry in rows $1$ through $k$, the last entry in row $r$ of $S$ must be $M-\ell+k$.  Therefore, in all cases we delete one occurrence of $M-\ell+k$ and append it to row $p$.  Thus, the content of $U$ is the same as the content of $S$.  We need to prove that $U$ is indeed an immaculate tableau.   The rows remain weakly increasing (since row $p$ is among the first $k$ rows), so we just need to show that the leftmost column is strictly increasing from top to bottom.

Suppose Step (4) is applied.  Since $\sigma_r=k$ and $r$ is chosen so that $\sigma_r$ is minimal among elements of $\mathcal{R}$ (whose values must be less than or equal to $k$), the only row containing $M-\ell+k$ is row $r$.  If $M-\ell+k$ were the leftmost entry in row $k$, then $r=k$ and row $k$ could not contain any entries other than $M-\ell+k$ (since the row entries weakly increase from left to right and no entry larger than $M-\ell+k$ is contained in row $k$).  But then $r=k$ and row $k$ would satisfy Conditions (a), (b), and (c) in Step (1), contradicting our choice of $k$ in Step (1).  Therefore, the leftmost entry in row $k$ must be strictly less than $M-\ell+k$.  The entry in the leftmost column of row $k+1$ of $S$ must equal $M-\ell+k+1$ by Condition (1b), so the leftmost column remains strictly increasing after the insertion of $M-\ell+k$ between rows $k$ and $k+1$.  Hence in this case, the resulting diagram $U$ is indeed an immaculate tableau.

Now suppose Step (5) is applied. Then, the leftmost column of $S$ is only altered if row $r$ is deleted.  In this case, the leftmost column remains strictly increasing since deleting an entry from a strictly increasing sequence results in a strictly increasing sequence.  Therefore, $U$ is an immaculate tableau.

The tunnel hook covering $V$ has the same shape as $U$.  We need to prove that $V$ is a well-defined tunnel hook covering whose content is equal to the content of $T$.  From this, it follows that $(V,U)\in C_{\alpha,\beta}.$

Suppose we are in Step (4).  A new row is added to $S$ and therefore $V$ has $\ell+1$ rows.  Since $\sigma_r=k \le \ell$, we have $\sigma_r < k +1\leq \ell+1$ and hence $V$ is a well-defined tunnel hook covering by \Cref{Th:THC_and_Perms}.

Suppose we are in Step (5). Then, $\sigma_r<k \le \ell$.  If $\gamma_k \not=1$, the number of rows of $V$ equals $\ell$ and hence $s_q \sigma$ is a permutation of $[\ell]$ and $V$ is a well-defined tunnel hook covering. Now, suppose $\gamma_r=1$. By construction of $\mathcal{R}$,  there is at least one $M-\ell+k$ in row $r$ of $S$ and since $\gamma_r=1$, this is the only entry in row~$r$. Note that this is in the leftmost column of $S$. Observe that $M-\ell+k$ is the largest value appearing in rows $1,\dots,k$ because all entries greater than $M-\ell+k$ are in rows $k+1,\dots,\ell$ by Condition (1c). Hence, we must actually have $r=k$ because $S$ is immaculate and therefore its leftmost column must be strictly increasing.

Recall $\Delta(T)=(\Delta_1,\Delta_2,\dots,\Delta_\ell)$ is given by $\Delta_i=\gamma_i +\sigma_i-i$ for each $i\in [\ell]$ (\Cref{delta-and-permutation}). We must have $\sigma_r=\sigma_k \ge k-1$.  To see this, assume that $\sigma_k < k-1$.  Since $\gamma_k=\gamma_r=1$, we see that $\Delta_k = 1 + \sigma_k-k < 1+k-1-k=0$, which is a contradiction since $\fl(\Delta)=\alpha$, which is a sequence of \emph{positive} integers and therefore $\Delta$ cannot include any negative integers. Hence, $\sigma_k\geq k-1$ and thus $\sigma_k=k-1$ since $\sigma_k=\sigma_r<k$ by assumption.

From here, note that if $k<\ell$, then $\sigma_\ell=\ell$ and $q=\sigma_r=k-1<\ell-1$, so $s_q\sigma$ is a permutation of $[\ell-1]$. If $k=\ell$, then $\sigma_{\ell}=\ell-1$, so $s_q\sigma(\ell)=\ell$. Thus, $s_q\sigma$ is a permutation of $[\ell-1]$. Therefore by \Cref{Th:THC_and_Perms}, $V$ is a well-defined tunnel hook covering.

Now that we know $V$ is a well-defined tunnel hook covering, we must prove that its content equals the content of $T$.  Let $\gamma'$ be the shape of $V$ and $\tau=s_q\sigma$.  Then $\Delta(V)$ is computed by \[\Delta_i(V)=\begin{cases}
        \gamma'_{r}+\tau_{r}-r=(\gamma_{r}-1)+(\sigma_r+1)-r&\textrm{if }i=r,\\
        \gamma'_p+\tau_p-p=(\gamma_p+1)+(\sigma_p-1)-p&\textrm{if }i=p,\\
        \gamma'_i+\tau_i-i=\gamma_i+\sigma_i-i&\text{otherwise.}
    \end{cases}\]

Note that in all of these cases, $\Delta_i(V) = \Delta_i(T)$.  Therefore the content of $V$ is equal to the content of $T$, as desired.
\end{proof}
   
Now, we can show that $\g1_{\alpha,\beta}$ is a sign-reversing involution on $C_{\alpha,\beta}$ when $\alpha\neq \beta$.
\begin{thm}
Let $\alpha$ and $\beta$ be compositions of $n$ and let $(T,S)\in \pairs$.
    \begin{enumerate}
        \item If $\alpha=\beta$, then $\g1_{\alpha,\alpha}(T,S)=(T,S)$.
        \item If $\alpha\neq \beta$, then $\g1_{\alpha,\beta}(T,S)=(V,U)$ satisfies
        \begin{enumerate}
            \item $\sgn(V)=-\sgn(T)$,
            \item $\g1_{\alpha,\beta}(V,U)=(T,S)$.
        \end{enumerate}
    \end{enumerate}
\end{thm}
\begin{proof} 
Note that in proving \Cref{lem:welldefinedsec4}, we also proved Part (1). Let $\alpha\neq \beta$. Since $\g1_{\alpha,\beta}$ is clearly sign-reversing, we just need to verify Part (2b). Suppose $\g1_{\alpha,\beta}(T,S)=(V,U)$ and let $k,\mathcal{R},q,$ and $p$ be as defined in \Cref{alg: involution} for computing $\g1_{\alpha,\beta}(T,S)$. Also, let $\tau=\perm(V)$ and let $k',\mathcal{R}',q',$ and $p'$ be defined as in \Cref{alg: involution} for computing $\g1_{\alpha,\beta}(V,U)$.

Suppose $\sigma_r=k$. Then applying $\g1_{\alpha,\beta}$ to $(T,S)$ adds a new row (with one box containing the entry $M-\ell+k$) between rows $k$ and $k+1$.  This means that in $U$, row $k+1$ contains one entry ($M - \ell+k$) and $\perm(V)_{k+1}=k$.  Since rows $k+1$ and greater all satisfy Conditions (a), (b), and (c) of Step (1) and $\perm(V)_{k+1}\neq k+1$, the selected index in Step (1)  when we apply \Cref{alg: involution} to $(V,U)$ is $k'=k+1$. Since $\sigma_r=k$ is the minimal entry of $\mathcal{R}$, $\perm(V)_{k+1}=(s_k\sigma)_{k+1}=k$ is the minimal entry in $\mathcal{R}'$. Thus, we select $r'=k+1$. Since $\perm(V)_{r'}=k<k+1=k'$ we are in Step (5) of the algorithm. Row $k+1$ is deleted (and the rows are shifted up) and the entry $M-\ell+k$ is placed in row $\tau^{-1}(k+1)=k$, which is exactly the row it came from during $\g1_{\alpha,\beta}(T,S)$.  The transposition $s_k$ is applied to $\perm(V)$, resulting in the original permutation $\sigma$.  Therefore $\g1_{\alpha,\beta}(\g1_{\alpha,\beta}(T,S))=(T,S)$ in this case.

Now suppose $\sigma_r \neq k$. In this situation, an entry ($M-\ell+k$) is deleted from row $r$ and added to row $p=\sigma^{-1}(\sigma_r+1)$.  Thus, $\sigma(p)=\sigma_r+1$, so when $s_q$ is applied to $\sigma$ to produce the permutation $\tau = s_q \sigma$, we have $\tau(p)=\sigma_r$ and $\tau(r)=\sigma_r+1$.  Since there is at least one $M-\ell+k$ in row $p$ of $U$, that must have the smallest $\tau$ value in $\mathcal{R}'$.  Therefore this value becomes $r'$.  This means one entry ($M-\ell+k$) is removed from row $p$ of $U$ and placed into row $r$ (since $\tau'(r)=\sigma_r+1$).  The permutation for the resulting tunnel hook covering is $s_q s_q \sigma = \sigma$.  Therefore $\g1_{\alpha,\beta}(\g1_{\alpha,\beta}(T,S))=(T,S)$ in this case, so the map is an involution.
\end{proof}

\begin{figure}[h]
    \centering
\begin{minipage}{.3\textwidth}\begin{tikzpicture}[yscale=-1,scale=.55]

\def\L{{4,3,4}}
\pgfmathsetmacro{\len}{dim(\L)}

\foreach \y in {1,...,\len}{
    \pgfmathsetmacro{\j}{\L[\y-1]}
    \foreach \i in {1,...,\j}{
        \draw (\i-.5,\y-.5) rectangle (\i+.5,\y+.5);
        }
}
\setsepchar{\\/,/ }
\readlist\LS{1,1,2,6\\2,3,5\\4,4,6,6}
\pgfmathsetmacro{\lens}{\listlen\LS[]}
\foreach \y in {1,...,\lens}{
    \pgfmathsetmacro{\j}{\L[\y-1]}
    \foreach \i in {1,...,\j}{
        \node at (\i,\y) {\LS[\y,\i]};
        }
}
\node at (-1,2) {$(S,T)\ =\ $};

\tikzset{every node/.style={inner sep=-4pt}}
\begin{scope}[on background layer]
\tikzset{every path/.style={line width = 7pt,color=black,line cap=round,opacity=.15,rounded corners}}
\draw (4,1)--(1,1)--(1,3);
\draw (3,2)--(2,2)--(2,3);
\draw (4,3)--(3,3);
\end{scope}

\end{tikzpicture}
\end{minipage}
$\xrightarrow[]{}$
\begin{minipage}{.3\textwidth}
\begin{tikzpicture}[yscale=-1,scale=.55]

\def\L{{4,4,3}}
\pgfmathsetmacro{\len}{dim(\L)}

\foreach \y in {1,...,\len}{
    \pgfmathsetmacro{\j}{\L[\y-1]}
    \foreach \i in {1,...,\j}{
        \draw (\i-.5,\y-.5) rectangle (\i+.5,\y+.5);
        }
}
\setsepchar{\\/,/ }
\readlist\LS{1,1,2,6\\2,3,5,6\\4,4,6}
\pgfmathsetmacro{\lens}{\listlen\LS[]}
\foreach \y in {1,...,\lens}{
    \pgfmathsetmacro{\j}{\L[\y-1]}
    \foreach \i in {1,...,\j}{
        \node at (\i,\y) {\LS[\y,\i]};
        }
}
\node at (-1,2) {$(U,V)\ =\ $};

\tikzset{every node/.style={inner sep=-4pt}}
\begin{scope}[on background layer]
\tikzset{every path/.style={line width = 7pt,color=black,line cap=round,opacity=.15,rounded corners}}
\draw (4,1)--(1,1)--(1,3);
\draw (4,2)--(2,2);
\draw (3,3)--(2,3);
\end{scope}

\end{tikzpicture}
\end{minipage}
\caption{An example of \Cref{alg: involution}}
\label{fig:THC-alg4-ex-general}
\end{figure}

\begin{ex}\label{ex:kinvk} We apply \Cref{alg: involution} to the pair $(S,T)\in \pairs$ as in \Cref{fig:THC-alg4-ex-general} with $\alpha=(2,2,1,2,1,3)$ and $\beta=(6,3,2)$.  Here we superimpose the two objects, both of shape $\gamma=(4,3,4)$ (and so $\ell=3$). The permutation of $T$ is $\perm(T)=\sigma=[3,2,1]$ and note that  the maximum entry of $S$ is $M=6$.

In Step (1), we find $k=3$, so $M-\ell+k=6$. Rows $1$ and $3$ contain $6$. Since $\sigma_1=3>1=\sigma_3$, we have $r=3$ in Step (3) of Algorithm~\ref{alg: involution}.

Since $\sigma_r=1\neq k$, we move from Step (3) directly to Step (5). Since $\sigma_3=1=q$, we have $\perm(V)=[3,1,2]$. We move the right-most cell in row $r=3$ to row $p=\sigma^{-1}(1+1)=2$ to create $U$ of shape $(4,4,3)$, and then define $V$ to be the unique \THC{} with shape $(4,4,3)$ and permutation $s_{1}\sigma=[3,1,2].$

\end{ex}

\begin{figure}[h]
    \centering
\begin{minipage}{.3\textwidth}\begin{tikzpicture}[yscale=-1,scale=.55]

\def\L{{4,3,3,2,3}}
\pgfmathsetmacro{\len}{dim(\L)}

\draw[white] (.5,5.5)-- (1.5,6.5);
\foreach \y in {1,...,\len}{
    \pgfmathsetmacro{\j}{\L[\y-1]}
    \foreach \i in {1,...,\j}{
        \draw (\i-.5,\y-.5) rectangle (\i+.5,\y+.5);
        }
}
\setsepchar{\\/,/ }
\readlist\LS{1,1,2,2\\2,3,5\\4,6,6\\7,7\\8,8,8}
\pgfmathsetmacro{\lens}{\listlen\LS[]}
\foreach \y in {1,...,\lens}{
    \pgfmathsetmacro{\j}{\L[\y-1]}
    \foreach \i in {1,...,\j}{
        \node at (\i,\y) {\LS[\y,\i]};
        }
}
\node at (-1,3.5) {$(S,T)\ =\ $};

\tikzset{every node/.style={inner sep=-4pt}}
\begin{scope}[on background layer]
\tikzset{every path/.style={line width = 7pt,color=black,line cap=round,opacity=.15,rounded corners}}
\draw (1,2)--(1,1)--(4,1);
\draw (2,2)--(3,2);
\draw (1,3)--(3,3);
\draw (1,4)--(2,4);
\draw (1,5)--(3,5);
\end{scope}

\end{tikzpicture}
\end{minipage}
$\xrightarrow[]{}$
\begin{minipage}{.3\textwidth}\begin{tikzpicture}[yscale=-1,scale=.55]

\def\L{{4,3,2,1,2,3}}
\pgfmathsetmacro{\len}{dim(\L)}

\foreach \y in {1,...,\len}{
    \pgfmathsetmacro{\j}{\L[\y-1]}
    \foreach \i in {1,...,\j}{
        \draw (\i-.5,\y-.5) rectangle (\i+.5,\y+.5);
        }
}
\setsepchar{\\/,/ }
\readlist\LS{1,1,2,2\\2,3,5\\4,6\\6\\7,7\\8,8,8}
\pgfmathsetmacro{\lens}{\listlen\LS[]}
\foreach \y in {1,...,\lens}{
    \pgfmathsetmacro{\j}{\L[\y-1]}
    \foreach \i in {1,...,\j}{
        \node at (\i,\y) {\LS[\y,\i]};
        }
}
\node at (-1,3.5) {$(U,V)\ =\ $};

\tikzset{every node/.style={inner sep=-4pt}}
\begin{scope}[on background layer]
\tikzset{every path/.style={line width = 7pt,color=black,line cap=round,opacity=.15,rounded corners}}
\draw (1,2)--(1,1)--(4,1);
\draw (2,2)--(3,2);
\draw (1,4)--(1,3)--(2,3);
\draw (1,5)--(2,5);
\draw (1,6)--(3,6);
\end{scope}

\end{tikzpicture}
\end{minipage}
\caption{An example of \Cref{alg: involution}}
\label{fig:THC-alg4-ex-general2}
\end{figure}
\begin{ex}\label{ex:kinvk2}  For $\alpha=(2,3,1,1,1,2,2,3)$ and $\beta=(5,2,3,2,3)$, consider the pair $(S,T)\in \pairs$ in \Cref{fig:THC-alg4-ex-general2} with $S$ and $T$ of shape $\gamma=(4,3,3,2,3)$. The permutation of $T$ is $\perm(T)=\sigma=[2,1,3,4,5]$ and note that $M=8$.

In Step (1), we set $k=3$, so $M-\ell+k=6$. Only Row $3$ contains $6$. Hence, we set $r=3$. Since $\sigma_r=3=k$, we move from Step (3) to Step (4) in the algorithm. Hence, we add a row below row $k$ consisting of a single $6$ to create $U$. $V$ has permutation $[2,1,4,3,5,6].$ Note that the content of $V$ is still $\beta=(5,2,3,2,3).$

\end{ex}

\section{A combinatorial proof of Theorem \ref{thm:KinvK-involution}}
\label{section: kinv-k-sym}

In this section, we provide a full combinatorial proof that $(K^{-1} K)_{\lambda,\mu}=\delta_{\lambda,\mu}$
(Theorem \ref{thm:KinvK-involution}) by utilizing Algorithm \ref{alg: involution} which showed $(\iK \K)_{\alpha,\beta}=\delta_{\alpha,\beta}$
(Theorem \ref{Th:Two}).  
Recall that for a composition $\alpha$, $\dec(\alpha)$ is the partition obtained by writing $\alpha$ in weakly decreasing order. For $\lambda,\mu\vdash n$, let $\sympairs$ be the set of pairs $(T,S)$ such that 
\begin{itemize}
\item 
$T$ is a \THC{}  with \THCc{} $\alpha$ where $\dec(\alpha)=\lambda$ and
\item 
$S$ is a semistandard Young tableau of shape $\sh(T)$ and content $\mu$.
\end{itemize}
Note that $S$ is partition shaped. Writing $K^{-1}_{\lambda,\mu}$ as a signed sum of tunnel hook coverings as in \Cref{eq:kostka-in-terms-thc} produces the equation
\begin{equation}\label{eq: KinvK-sym-setup}
(K^{-1}K)_{\lambda,\mu}=\sum_{(T,S)\in \sympairs}\sgn(T).\end{equation}
We prove \Cref{thm:KinvK-involution} by exhibiting a sign-reversing involution on $\sympairs$.

Define $\sympairsbig$ to be the set of pairs $(T,S)$ as above by replacing the condition that ``$S$ is a semistandard Young tableau'' with ``$S$ is an immaculate tableau.''  Note that the shapes of $T$ and $S$ do not need to be partitions.

Assume that $\lambda\neq \mu$. For a pair $(T,S) \in \sympairsbig$, we say that $c=(i,j)$ is a \emph{bad cell} of $S$ if $i\geq 2$ and either $S(i-1,j)$ is not defined or $S(i-1,j)\geq S(i,j)$.  We have the following two observations.
\begin{enumerate}
    \item If $(T,S)\in \sympairs$, then applying $\g1$ (\Cref{alg: involution}) to $(T,S)$ produces an element of $\sympairsbig.$
    \item Let $(T,S)\in \sympairsbig$. Then $S$ has a bad cell if and only if $(T,S)\not\in \sympairs.$
\end{enumerate}

In order to define our involution on $\sympairs$, it is useful to first give an involution $\theta_{\lambda,\mu}$  on $\sympairsbig\setminus \sympairs.$  The map $\theta_{\lambda,\mu}$ is the same style of involution as those appearing in a variety of theorems including a Berg--Bergeron--Saliola--Serrano--Zabrocki result in computing certain Littlewood--Richardson coefficients for immaculate functions \cite{BBSSZ17}, Gasharov and Shareshian--Wachs results in a Schur positivity proof for certain chromatic symmetric functions \cite{gasharov_1996,shareshian_chromatic_2016}, and a Gessel--Vienott result on determinants~\cite{GesVie89}.

\begin{alg}\label{alg:Bad-Involution}
For $\lambda,\mu\vdash n$ such that $\lambda\neq \mu$, define a map
\[\theta_{\lambda,\mu }:\sympairsbig\setminus \sympairs\to \sympairsbig\setminus \sympairs\]
algorithmically as follows. 
\begin{enumerate}
\item Let $(V,U)\in \sympairsbig\setminus \sympairs$ where $V$ has \THCc~$\beta$ with $\dec(\beta)=\lambda$. 
\item Let $i$ be the leftmost column of $U$ containing a bad cell and let $t$ be the largest value such that row $t$ contains a bad cell in column $i$.  Swap the cells of $U$ in the sets
\[
P=\{(t,j)\in U \mid j>i\}
\quad 
\text{and}\quad 
Q=\{(t-1,j)\in U \mid j\geq i\}
\]
to create $U'$, meaning we set $U'(t,j)=U(t-1,j-1)$ for $j>i$ and $U'(t-1,j)=U(t,j+1)$ for $j\geq i$ (when defined), and set $U'(x,y)=U(x,y)$ otherwise.
\item Let $V'$ be the THC of shape $\sh(U')$ such that $\perm(V')=\perm(V) s_{t-1}.$ 
\item Set $\theta_{\lambda,\mu}(V,U)=(V',U')$.
\end{enumerate}
Now set $\theta(V,U)=\theta_{\lambda,\mu}(V,U)$ if $(V,U)\in E_{\lambda,\mu}\setminus D_{\lambda,\mu}.$
\end{alg}

\begin{ex}
Consider the immaculate tableau $U$ in \Cref{fig:bad-involution-example}, which has content \[\mu=(4,4,4,4,4,2,2,2,2).\] The cell $(4,3)$ (circled) is a bad cell. Column $i=3$ is the leftmost column containing a bad cell (columns 4, 5, 6, and 7 also contain bad cells) and $t=4$ is the largest value such that row $t$ contains a bad cell in column 3 (cell $(3,3)$ is also bad). The map $\theta_{\lambda,\mu}$ then swaps cells of $P$ and $Q$ (highlighted) to obtain the immaculate tableau $U'$. Note that $(4,3)$ is still a bad cell of $U'$, that column $i=3$ is the leftmost column containing a bad cell, and that $t=4$ is still the largest value such that row $t$ contains a bad cell in column $i=3$. Hence, $\theta_{\lambda,\mu}$ would turn $U'$ back into $U$.
\begin{figure}
    \centering
    \begin{tikzpicture}
\node at (0,0) {\begin{tikzpicture}[yscale=-1,scale=.55]
\node at (3,6) {$U$};
\node at (4.5,2) {$P$};
\node at (5.5,5) {$Q$};
\def\L{{9,3,5,7,4}}
\pgfmathsetmacro{\len}{dim(\L)}

\foreach \y in {1,...,\len}{
    \pgfmathsetmacro{\j}{\L[\y-1]}
    \foreach \i in {1,...,\j}{
        \draw (\i-.5,\y-.5) rectangle (\i+.5,\y+.5);
        }
}
\setsepchar{\\/,/ }
\readlist\LS{1,1,1,1,2,2,3,3,4\\2,2,5\\3,3,5,5,5\\4,4,\circl{4},6,6,7,7\\8,8,9,9}
\pgfmathsetmacro{\lens}{\listlen\LS[]}
\foreach \y in {1,...,\lens}{
    \pgfmathsetmacro{\j}{\L[\y-1]}
    \foreach \i in {1,...,\j}{
        \node at (\i,\y) {\LS[\y,\i]};
        }

}

\tikzset{every node/.style={inner sep=-4pt}}
\begin{scope}[on background layer]
\tikzset{every path/.style={line width = 7pt,color=black,line cap=round,opacity=.15,rounded corners}}
\draw (5,3)--(3,3);
\draw (4,4)--(7,4);
\end{scope}

\end{tikzpicture}};
\node at (3,0) {$\longrightarrow$};

\node at (6.5,0){\begin{tikzpicture}[yscale=-1,scale=.55]
\node at (3,6) {$U'$};
\node at (5.5,5) {$P$};
\node at (4.5,2) {$Q$};
\def\L{{9,3,6,6,4}}
\pgfmathsetmacro{\len}{dim(\L)}

\foreach \y in {1,...,\len}{
    \pgfmathsetmacro{\j}{\L[\y-1]}
    \foreach \i in {1,...,\j}{
        \draw (\i-.5,\y-.5) rectangle (\i+.5,\y+.5);
        }
}
\setsepchar{\\/,/ }
\readlist\LS{1,1,1,1,2,2,3,3,4\\2,2,5\\3,3,6,6,7,7\\4,4,\circl{4},5,5,5\\8,8,9,9}
\pgfmathsetmacro{\lens}{\listlen\LS[]}
\foreach \y in {1,...,\lens}{
    \pgfmathsetmacro{\j}{\L[\y-1]}
    \foreach \i in {1,...,\j}{
        \node at (\i,\y) {\LS[\y,\i]};
        }

}

\tikzset{every node/.style={inner sep=-4pt}}
\begin{scope}[on background layer]
\tikzset{every path/.style={line width = 7pt,color=black,line cap=round,opacity=.15,rounded corners}}
\draw (6,3)--(3,3);
\draw (4,4)--(6,4);
\end{scope}

\end{tikzpicture}};
\end{tikzpicture}
    \caption{An example of \Cref{alg:Bad-Involution}}
    \label{fig:bad-involution-example}
\end{figure}

\end{ex}

\begin{lem}
For all $\lambda,\mu\vdash n$ with $\lambda\neq \mu$, $\theta_{\lambda,\mu}$ is a sign-reversing involution on $\sympairsbig\setminus\sympairs$.
\end{lem}

\begin{proof}
The map is clearly sign-reversing, so it just remains to show that it is a well-defined involution.  The content of $U'$ is equal to that of $U$ and $V'$ is a THC with $\sh(V')=\sh(U')$ by construction. Therefore proving that $\theta_{\lambda,\mu}$ is well-defined reduces to proving that $U'$ is an immaculate tableau of composition shape which is not a semistandard Young tableau and that the content of $V'$ is a rearrangement of $\lambda$.
  
To see that $U'$ is an immaculate tableau of composition shape, first note that since $U$ is immaculate, the first column of $U$ is increasing so a bad cell is not in the first column of $U$.  Hence $P$ and $Q$ contain no cells in the first column, so the first column of $U'$ is increasing.  We must prove that the rows of $U'$ weakly increase from left to right.  

Suppose $(t,i)$ is the chosen bad cell.  By minimality of $i$, $U(t-1,i-1)$ is defined. By maximality of $t$ and the fact that $U$ is an immaculate tableau, $U(t-1,i-1)<U(t,i-1) \le U(t,i)$.  To check if row $t-1$ of $U'$ is weakly increasing, we may assume that $U(t,i+1)$ is defined and show that $U(t-1,i-1) \le U(t,i+1)$.  If $U(t,i+1)$ is defined then $U(t,i) \le U(t,i+1)$, which implies $U(t-1,i-1) \le U(t,i+1)$.  Therefore row $t-1$ of $U'$ is weakly increasing.  

If $U(t-1,i)$ is undefined then row $t$ of $U'$ is weakly increasing.  Otherwise, since $(t,i)$ is a bad cell, we have $U(t-1,i) \ge U(t,i)$ and row $t$ of $U'$ is weakly increasing.  All other rows of $U'$ are the same as those in $U$, so $U'$ is an immaculate tableau. Moreover, $(t,i)$ is still a bad cell of $U'$ since $U'(t,i)=U(t,i) \le U(t,i+1)=U'(t-1,i)$ (assuming $U(t,i+1)$ and thus $U'(t-1,i)$ are defined). Hence, $U'$ is an immaculate tableau but not a semistandard Young tableau.

Next observe that if $U$ has shape $\alpha=(\alpha_1, \alpha_2, \hdots , \alpha_{\ell}) \vDash n$, then $U'$ has shape \[\alpha'=(\alpha_1,\alpha_2,\dots,\alpha_{t-2},\alpha_{t}-1,\alpha_{t-1}+1,\alpha_{t+1},\dots,\alpha_{\ell}).\] We must have $\alpha_t \ge 2$ (since the bad cell cannot be in the leftmost column), and therefore $\alpha'$ must be a composition.  

Next, to see that the content of $V'$ is a rearrangement of $\lambda$, let $\tau=\perm(V)$ and $\tau'=\tau s_{t-1}$. Hence, we have 
\begin{align*}
\Delta_k(V')&=\alpha'_{k}-k+\tau'_k=\alpha_k-k+\tau_k=\Delta_k(V)\quad (\text{for $k\not\in \{t-1,t\}$}),\\
\Delta_{t-1}(V')&=\alpha'_{t-1}-(t-1)+\tau'_{t-1}=\alpha_{t}-1-(t-1)+\tau_t = \alpha_t-t+\tau_t=\Delta_{t}(V),\\
\Delta_t(V')&=\alpha'_{t}-t+\tau'_t=\alpha_{t-1}+1-t+\tau_{t-1}=\Delta_{t-1}(V).
\end{align*}
Therefore, $$\Delta(V')=(\Delta_1(V), \Delta_2(V), \hdots, \Delta_{t-2}(V), \Delta_t(V), \Delta_{t-1}(V) , \Delta_{t+1}(V), \hdots,\Delta_{\ell}(V))$$ is the \THCc{} of $V'$.  Since $\dec(\Delta(V'))=\dec(\Delta(V))=\lambda$, the map $\theta_{\lambda,\mu}$ is well-defined.

Finally, to see that $\theta_{\lambda,\mu}$ is an involution, first recall that $(t,i)$ is still a bad cell of $U'$. Moreover, there are no bad cells to the left of $(t,i)$ in $U'$ since they would have been bad cells in $U$. By maximality, there are no bad cells in column $i$ of $U'$ that are in a row below row $t+1$ (as they would have also been bad cells of $U$).  Thus, $(t,i)$ remains the ``chosen'' bad cell in $U'$, so $U'$ becomes $U$ again under the algorithm. Since the bad cell is still in row $t$, the tunnel hook covering constructed has shape $\sh(U)$ and permutation $\perm(V')s_{t-1}=\perm(V)$, and thus must be $V$.
\end{proof}

\begin{alg}\label{alg: sym involution FULL}
Let $\lambda,\mu\vdash n$. We define a map $\rho_{\lambda,\mu}:\sympairs\to \sympairs$ as follows. Let $(T, S)\in \sympairs$.  Recall the map $\psi$ from \Cref{alg: involution}.
\begin{enumerate}
\item If $\lambda=\mu$ then $\rho_{\lambda,\lambda}(T,S)=(T,S)$.
\item If $\lambda \not= \mu$, set $(T',S')=(T,S)$.
\begin{enumerate}
\item Let $(V,U)=\psi(T',S')$. If $(V,U) \in \sympairs$, then set $\rho_{\lambda,\mu}(T,S)=(V,U)$.
\item Otherwise, set $(T',S')=\theta(V,U)$ and then go back to Step (2a).
\end{enumerate}
\end{enumerate}
Now set $\rho(T,S)=\rho_{\lambda,\mu}(T,S)$ if $(T,S)\in \sympairs$.
\end{alg}

\begin{thm}\label{thm:sym-involution}
For distinct partitions $\lambda$ and $\mu$ of $n$, $\rho_{\lambda,\mu}$ is a sign-reversing involution on $\sympairs.$
\end{thm}

\begin{proof}
Turn $\sympairsbig$ into a multigraph by connecting $(V,U)$ to $(V',U')$ if they are related by either
\begin{enumerate}
    \item $(V',U')=\g1(V,U)$, or
    \item  $(V,U)\in \sympairsbig\setminus \sympairs$ and $\theta(V,U)=(V',U')$.
\end{enumerate}
Since both $\psi$ and $\theta$ are involutions, every vertex of $\sympairsbig$ has degree $1$ or $2$.  Hence $\sympairsbig$ is a disjoint union of paths and cycles (which are all finite since $\sympairsbig$ is finite). See \Cref{fig:graphic-involution}.

Every element of $\sympairs$ must have degree one while every element of $\sympairsbig\setminus \sympairs$ must have degree two since only $\psi$ can produce an element of $\sympairs$.  Therefore every connected component containing an element of $\sympairs$ must be a path with an odd number of edges whose endpoints are in $\sympairs$.  Thus, $\rho_{\lambda,\mu}$ is an involution, and sign-reversing since every path has an odd number of edges.
\end{proof}

\begin{figure}
      \begin{tikzpicture}[scale=.48]
          \draw (.5,.5) rectangle (24.5,2.5);
          \node at (22,1) {$D_{\lambda,\mu}$};
          \draw (0,0) rectangle (25,7.5);
          \node at (22.5,7) {$E_{\lambda,\mu}$};
          \pgfmathsetmacro{\grdx}{14}
          \pgfmathsetmacro{\grdy}{4}
          \foreach \x in {1,...,\grdx}{
          \foreach \y in {1,...,\grdy}{
          \node (\x\y) at (1.5*\x,1.5*\y+.3) {$\bullet$};
          }
         
          }
          \tikzset{every path/.style={line width=2pt}}
        \draw[red] (11)--(12);
        \draw[blue] (12)--(13);
        \draw[red] (13)--(14);
        \draw[blue] (14)--(24);
        \draw[red] (24)--(23);
        \draw[blue] (23)--(33);
        \draw[red] (33)--(34);
        \draw[blue] (34)--(44);
        \draw[red] (44)--(43);
        \draw[blue] (43)--(42);
        \draw[red] (42)--(41);
        \draw[red] (51)--(52);
        \draw[blue] (52)--(53);
        \draw[red] (53)--(54);
        \draw[blue] (54)--(64);
        \draw[red] (64)--(74);
        \draw[blue] (74)--(84);
        \draw[red] (84)--(83);
        \draw[blue] (83)--(82);
        \draw[red] (82)--(81);

          \draw[red] (73)--(63);
        \draw[blue] (63)--(62);
        \draw[red] (62)--(72);
        \draw[blue] (72)--(73);
        \draw[red] (21)--(22);
        \draw[blue] (22)--(32);
        \draw[red] (32)--(31);
        \draw[red] (111)--(112);
        \draw[blue] (112)--(113);
        \draw[red] (113)--(103);
        \draw[blue] (103)--(104);
        \draw[red] (104)--(114);
        \draw[blue] (114)--(124);
        \draw[red] (124)--(134);
        \draw[blue] (134)--(144);
        \draw[red] (144)--(143);
        \draw[blue] (143)--(133);
        \draw[red] (133)--(123);
        \draw[blue] (123)--(122);
        \draw[red] (122)--(132);
        \draw[blue] (132)--(142);
        \draw[red] (142)--(141);
        
        \draw[blue] (94)to[bend left=60](93);
        \draw[red] (93)--(94);
        \draw[blue] (92)to[bend left=60](102);
        \draw[red] (92)--(102);
        \draw[red] (61)--(71);
        \draw[red] (91)--(101);
        \draw[red] (121)--(131);
      \end{tikzpicture}
      \caption{This is a sketch of the graph described in the proof of \Cref{thm:sym-involution}. Each dot is an element of $E_{\lambda,\mu}$. Red edges are the involution $\g1$ and blue edges are the involution $\theta$.}\label{fig:graphic-involution}
      \end{figure}
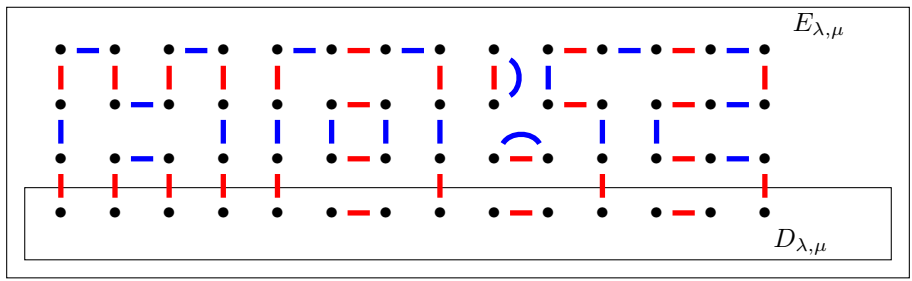

\begin{ex}
\label{Ex:ex53}

Since we will be cycling through collections of pairs,
we will use $(T^{(i+1)},S^{(i+1)})$ to denote the result of either $\psi$ or $\theta$ applied to $(T^{(i)},S^{(i)})$.  We will embed all of our permutations into $S_5$.  Consider the following pair: 
\[\begin{tikzpicture}[yscale=-1,scale=.55]

\def\L{{2,2,1}}
\pgfmathsetmacro{\len}{dim(\L)}

\foreach \y in {1,...,\len}{
    \pgfmathsetmacro{\j}{\L[\y-1]}
    \foreach \i in {1,...,\j}{
        \draw (\i-.5,\y-.5) rectangle (\i+.5,\y+.5);
        }
}
\setsepchar{\\/,/ }
\readlist\LS{1,2\\3,4\\5}
\pgfmathsetmacro{\lens}{\listlen\LS[]}
\foreach \y in {1,...,\lens}{
    \pgfmathsetmacro{\j}{\L[\y-1]}
    \foreach \i in {1,...,\j}{
        \node at (\i,\y) {\LS[\y,\i]};
        }
}
\node at (-1.5,2) {$(T^{(1)},S^{(1)})\ =\ $};

\tikzset{every node/.style={inner sep=-4pt}}
\begin{scope}[on background layer]
\tikzset{every path/.style={line width = 7pt,color=black,line cap=round,opacity=.15,rounded corners}}
\draw (1,1)--(2,1);
\draw (1,3)--(1,2)--(2,2);
\draw (2-.25,3)--(2+.25,3);
\end{scope}

\end{tikzpicture}\]

Begin with \Cref{alg: sym involution FULL} Step (2a), which applies \Cref{alg: involution}.  Note that $\sigma^{(1)}=[1,3,2,4,5]$ and $k^{(1)}=3$, $r^{(1)}=3$, $q^{(1)}=\sigma^{(1)}_3=2\neq k^{(1)},$ so Step (5)  of \Cref{alg: involution} applies. In particular, $\perm(T^{(2)})=s_{q^{(1)}}\sigma^{(1)}
=s_2\sigma^{(1)}=[1,2,3,4,5]$ and
$p^{(1)}=(\sigma^{(1)})^{-1}(q^{(1)}+1)=2$. Hence we move the $5$ to the end of row $p^{(1)}=2$ to obtain:

\[\begin{tikzpicture}[yscale=-1,scale=.55]

\def\L{{2,3}}
\pgfmathsetmacro{\len}{dim(\L)}

\foreach \y in {1,...,\len}{
    \pgfmathsetmacro{\j}{\L[\y-1]}
    \foreach \i in {1,...,\j}{
        \draw (\i-.5,\y-.5) rectangle (\i+.5,\y+.5);
        }
}
\setsepchar{\\/,/ }
\readlist\LS{1,2\\3,4,5}
\pgfmathsetmacro{\lens}{\listlen\LS[]}
\foreach \y in {1,...,\lens}{
    \pgfmathsetmacro{\j}{\L[\y-1]}
    \foreach \i in {1,...,\j}{
        \node at (\i,\y) {\LS[\y,\i]};
        }
}
\node at (-1.5,1.5) {$(T^{(2)},S^{(2)})\ =\ $};

\tikzset{every node/.style={inner sep=-4pt}}
\begin{scope}[on background layer]
\tikzset{every path/.style={line width = 7pt,color=black,line cap=round,opacity=.15,rounded corners}}
\draw (1,1)--(2,1);
\draw (1,2)--(3,2);
\end{scope}

\end{tikzpicture}\]



Note that 
$\sh(S^{(2)})=(2,3)$ is not a partition.  So, we need to apply Step (2b) of \Cref{alg: sym involution FULL}, which in turn applies \Cref{alg:Bad-Involution}.
Note that cell $(2,3)$ is the only bad cell, so $t^{(2)}=2$ and $P^{(2)}=Q^{(2)}=\emptyset$. Thus, we have: 

\[\begin{tikzpicture}[yscale=-1,scale=.55]

\def\L{{2,3}}
\pgfmathsetmacro{\len}{dim(\L)}

\foreach \y in {1,...,\len}{
    \pgfmathsetmacro{\j}{\L[\y-1]}
    \foreach \i in {1,...,\j}{
        \draw (\i-.5,\y-.5) rectangle (\i+.5,\y+.5);
        }
}
\setsepchar{\\/,/ }
\readlist\LS{1,2\\3,4,5}
\pgfmathsetmacro{\lens}{\listlen\LS[]}
\foreach \y in {1,...,\lens}{
    \pgfmathsetmacro{\j}{\L[\y-1]}
    \foreach \i in {1,...,\j}{
        \node at (\i,\y) {\LS[\y,\i]};
        }
}
\node at (-1.5,1.5) {$(T^{(3)},S^{(3)})\ =\ $};

\tikzset{every node/.style={inner sep=-4pt}}
\begin{scope}[on background layer]
\tikzset{every path/.style={line width = 7pt,color=black,line cap=round,opacity=.15,rounded corners}}
\draw (1,2)--(1,1)--(2,1);
\draw (2,2)--(3,2);
\end{scope}

\end{tikzpicture}\]

Next we apply Step (2a) of \Cref{alg: sym involution FULL} to the pair
$(T^{(3)},S^{(3)})$, which applies \Cref{alg: involution} with $k^{(3)}=2$ and $r^{(3)}=2$. Then, $\sigma^{(3)}_{r^{(3)}}=1\neq k^{(3)}$, so we apply 
Step (5) of \Cref{alg: involution}. Since $q^{(3)}=\sigma^{(3)}_{r^{(3)}}=1$ and $p^{(3)}=(\sigma^{(3)}_{r^{(3)}})^{-1}(q^{(3)}+1)=1$, we move $5$ to the end of row $p^{(3)}=1$ and set $\perm(T^{(4)})=s_{q^{(3)}}\sigma^{(3)}=[1,2,3,4,5]$ to get:
\begin{center}
\begin{tikzpicture}[yscale=-1,scale=.6]
\def\L{{3,2}}
\pgfmathsetmacro{\len}{dim(\L)}
\foreach \y in {1,...,\len}{
    \pgfmathsetmacro{\j}{\L[\y-1]}
    \foreach \i in {1,...,\j}{
        \draw (\i-.5,\y-.5) rectangle (\i+.5,\y+.5);
        }
}

\node at (-1.5,1.5) {$(T^{(4)},S^{(4)})=$};
\tikzset{every node/.style={inner sep=-4pt,color=black}}
\node (1t) at (1,1) {1};
\node (1t) at (2,1) {2};
\node (3t) at (3,1) {5};
\node (2t) at (1,2) {3};
\node (2t) at (2,2) {4};
\tikzset{every node/.style={inner sep=-4pt}}
\begin{scope}[on background layer]
\tikzset{every path/.style={line width = 7pt,color=black,line cap=round,opacity=.15,rounded corners}}
\draw (1,1)--(3,1);
\draw (1,2)--(2,2);
\end{scope}
\end{tikzpicture}
\end{center}
The table below records the common shape $\alpha^{(i)}$ (padded with 0's) of $(T^{(i)},S^{(i)})$ and the permutation $\sigma^{(i)}=\perm(T^{(i)})$ (extended to a permutation of $[5]$) as we progress through the algorithm.
\[
\begin{array}{c|c|c}(T^{(i)},S^{(i)})&\alpha^{(i)}&\sigma^{(i)}  \\\hline
     (T^{(1)},S^{(1)})&(2,2,1,0,0)&[1,3,2,4,5]\\
     (T^{(2)},S^{(2)})&(2,3,0,0,0)&[1,2,3,4,5]\\
     (T^{(3)},S^{(3)})&(2,3,0,0,0)&[2,1,3,4,5]\\
     (T^{(4)},S^{(4)})&(3,2,0,0,0)&[1,2,3,4,5]\\
\end{array}\]
The corresponding terms of the determinants of the Jacobi--Trudi matrices that correspond to the pairs $(\alpha^{(i)},\sigma^{(i)})$ 
are seen in Figure
\ref{fig:Sym}.

\begin{figure}
\begin{tikzpicture}
\node at (0,0) {
\hskip -1.875truein
\begin{tikzpicture}[   
    every matrix/.append style={
        left delimiter={[},
        right delimiter={]},
        }
 ]
  \matrix{
    \node[draw=red] {h_2}; & \node{h_3}; & \node {h_4}; & \node {h_5};& \node {h_6}; \\
     \node {h_1}; & \node{h_2}; & \node[draw=red] {h_3}; & \node {h_4};& \node {h_5}; \\    
     \node {h_{-1}}; & \node[draw=red]{h_0}; & \node {h_1}; & \node {h_2};& \node {h_3}; \\    
      \node{h_{-3}}; & \node {h_{-2}}; & \node {h_{-1}};& \node[draw=red] {h_0}; & \node {h_{1}};\\    
     \node {h_{-4}}; & \node{h_{-3}}; & \node {h_{-2}}; & \node {h_{-1}};& \node[draw=red] {h_0}; \\ 
};
    
     \draw[->] (2.3,0)--(3.35,0);
    \hskip 2.25truein
\matrix{
    \node[draw=red] {h_2}; & \node{h_3}; & \node {h_4}; & \node {h_5};& \node {h_6}; \\
     \node {h_2}; & \node[draw=red]{h_3}; & \node {h_4}; & \node {h_5};& \node {h_6}; \\    
     \node {h_{-2}}; & \node{h_{-1}}; & \node[draw=red] {h_0}; & \node {h_1};& \node {h_2}; \\    
      \node{h_{-3}}; & \node {h_{-2}}; & \node {h_{-1}};& \node[draw=red] {h_0}; & \node {h_{1}};\\    
     \node {h_{-4}}; & \node{h_{-3}}; & \node {h_{-2}}; & \node {h_{-1}};& \node[draw=red] {h_0}; \\ 
};\end{tikzpicture}};
\node at (0,-3) {
\hskip -1.875truein
\begin{tikzpicture}[   
    every matrix/.append style={
        left delimiter={[},
        right delimiter={]},
        }
 ]
\matrix{
    \node {h_2}; & \node[draw=red]{h_3}; & \node {h_4}; & \node {h_5};& \node {h_6}; \\
     \node[draw=red] {h_2}; & \node{h_3}; & \node {h_4}; & \node {h_5};& \node {h_6}; \\    
     \node {h_{-2}}; & \node{h_{-1}}; & \node[draw=red] {h_0}; & \node {h_1};& \node {h_2}; \\    
      \node{h_{-3}}; & \node {h_{-2}}; & \node {h_{-1}};& \node[draw=red] {h_0}; & \node {h_{1}};\\    
     \node {h_{-4}}; & \node{h_{-3}}; & \node {h_{-2}}; & \node {h_{-1}};& \node[draw=red] {h_0}; \\ 
};
\draw[->] (2.3,0)--(3.35,0);
\hskip 2.25truein
\matrix{
    \node[draw=red] {h_3}; & \node{h_4}; & \node {h_5}; & \node {h_6};& \node {h_7}; \\
     \node {h_1}; & \node[draw=red]{h_2}; & \node {h_3}; & \node {h_4};& \node {h_5}; \\    
     \node {h_{-2}}; & \node{h_{-1}}; & \node[draw=red] {h_0}; & \node {h_1};& \node {h_2}; \\    
      \node{h_{-3}}; & \node {h_{-2}}; & \node {h_{-1}};& \node[draw=red] {h_0}; & \node {h_{1}};\\    
     \node {h_{-4}}; & \node{h_{-3}}; & \node {h_{-2}}; & \node {h_{-1}};& \node[draw=red] {h_0}; \\ 
};\end{tikzpicture}};
\draw[->] (.45,-.25)--(-.55,-2.6);
\end{tikzpicture}
   
    \caption{An example of the product of terms in the Jacobi--Trudi determinant as we go from $(T^{(i)},S^{(i)})$
    to $(T^{(i+1)},S^{(i+1)})$ for $i \in \{1,2,3\}$}
    \label{fig:Sym}
\end{figure}
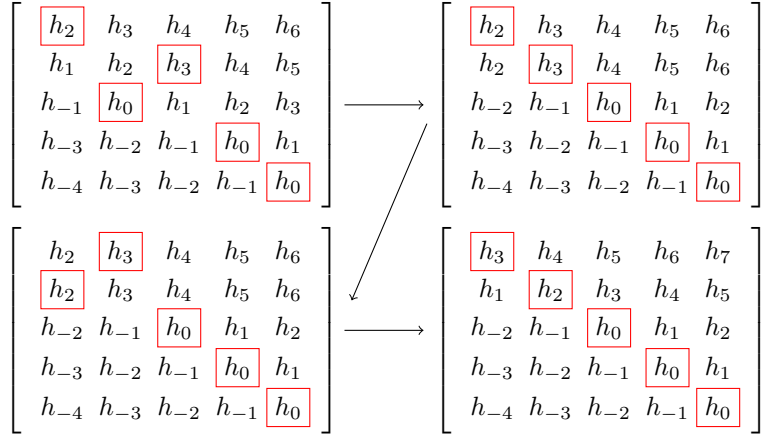

Note that the term of the determinant corresponding to the pair $(\alpha^{(1)},\sigma^{(1)})$
is $-h_2h_3h_0h_0h_0$
while the term in the corresponding
determinant corresponding to the pair
$(\alpha^{(4)},\sigma^{(4)})$
  is $h_3h_2h_0h_0h_0.$ 
  Since we are in $Sym$,
  these two terms cancel each other out.
\end{ex}

\begin{ex}
    \Cref{fig:sym-involu-full-ex} depicts \Cref{alg: sym involution FULL} applied to the pair $(T^{(1)},S^{(1)})\in \sympairs$ where $\lambda=(5,2,1)$ and $\mu=(2,2,2,2)$.  The circled entry indicates the selected bad cell for $\theta$.
\end{ex}

   \begin{figure}
    \centering
$\begin{tikzpicture}[yscale=-1,scale=.55]

\def\L{{4,2,2}}
\pgfmathsetmacro{\len}{dim(\L)}

\foreach \y in {1,...,\len}{
    \pgfmathsetmacro{\j}{\L[\y-1]}
    \foreach \i in {1,...,\j}{
        \draw (\i-.5,\y-.5) rectangle (\i+.5,\y+.5);
        }
}
\setsepchar{\\/,/ }
\readlist\LS{1,1,4,4\\2,2\\3,3}
\pgfmathsetmacro{\lens}{\listlen\LS[]}
\foreach \y in {1,...,\lens}{
    \pgfmathsetmacro{\j}{\L[\y-1]}
    \foreach \i in {1,...,\j}{
        \node at (\i,\y) {\LS[\y,\i]};
        }
}

\node at (2,4) {$(T^{(1)},S^{(1)})$};

\tikzset{every node/.style={inner sep=-4pt}}
\begin{scope}[on background layer]
\tikzset{every path/.style={line width = 7pt,color=black,line cap=round,opacity=.15,rounded corners}}
\draw (4,1)--(1,1)--(1,2);
\draw (2-.25,2)--(2+.25,2);
\draw (2,3)--(1,3);
\end{scope}

\end{tikzpicture}\ \begin{minipage}{.05\textwidth}\vspace{-2.8cm}$\xrightarrow{\psi}$
\end{minipage}\ \begin{tikzpicture}[yscale=-1,scale=.55]

\def\L{{3,2,3}}
\pgfmathsetmacro{\len}{dim(\L)}

\foreach \y in {1,...,\len}{
    \pgfmathsetmacro{\j}{\L[\y-1]}
    \foreach \i in {1,...,\j}{
        \draw (\i-.5,\y-.5) rectangle (\i+.5,\y+.5);
        }
}
\setsepchar{\\/,/ }
\readlist\LS{1,1,4\\2,2\\3,3,\circl{4}}
\pgfmathsetmacro{\lens}{\listlen\LS[]}
\foreach \y in {1,...,\lens}{
    \pgfmathsetmacro{\j}{\L[\y-1]}
    \foreach \i in {1,...,\j}{
        \node at (\i,\y) {\LS[\y,\i]};
        }
}

\node at (2,4) {$(T^{(2)},S^{(2)})$};

\tikzset{every node/.style={inner sep=-4pt}}
\begin{scope}[on background layer]
\tikzset{every path/.style={line width = 7pt,color=black,line cap=round,opacity=.15,rounded corners}}
\draw (3,1)--(1,1)--(1,3);
\draw (2-.25,2)--(2+.25,2);
\draw (3,3)--(2,3);
\end{scope}

\end{tikzpicture}\ \begin{minipage}{.05\textwidth}\vspace{-2.8cm}$\xrightarrow{\theta}$ \end{minipage}\ \begin{tikzpicture}[yscale=-1,scale=.55]

\def\L{{3,2,3}}
\pgfmathsetmacro{\len}{dim(\L)}

\foreach \y in {1,...,\len}{
    \pgfmathsetmacro{\j}{\L[\y-1]}
    \foreach \i in {1,...,\j}{
        \draw (\i-.5,\y-.5) rectangle (\i+.5,\y+.5);
        }
}
\setsepchar{\\/,/ }
\readlist\LS{1,1,4\\2,2\\3,3,4}
\pgfmathsetmacro{\lens}{\listlen\LS[]}
\foreach \y in {1,...,\lens}{
    \pgfmathsetmacro{\j}{\L[\y-1]}
    \foreach \i in {1,...,\j}{
        \node at (\i,\y) {\LS[\y,\i]};
        }
}

\node at (2,4) {$(T^{(3)},S^{(3)})$};

\tikzset{every node/.style={inner sep=-4pt}}
\begin{scope}[on background layer]
\tikzset{every path/.style={line width = 7pt,color=black,line cap=round,opacity=.15,rounded corners}}
\draw (3,1)--(1,1)--(1,3);
\draw (2,2)--(2,3);
\draw (3-.25,3)--(3+.25,3);
\end{scope}

\end{tikzpicture}
\ \begin{minipage}{.05\textwidth}\vspace{-2.8cm}$\xrightarrow{\psi}$ \end{minipage}\ \begin{tikzpicture}[yscale=-1,scale=.55]

\def\L{{3,3,2}}
\pgfmathsetmacro{\len}{dim(\L)}

\foreach \y in {1,...,\len}{
    \pgfmathsetmacro{\j}{\L[\y-1]}
    \foreach \i in {1,...,\j}{
        \draw (\i-.5,\y-.5) rectangle (\i+.5,\y+.5);
        }
}
\setsepchar{\\/,/ }
\readlist\LS{1,1,4\\2,2,\circl{4}\\3,3}
\pgfmathsetmacro{\lens}{\listlen\LS[]}
\foreach \y in {1,...,\lens}{
    \pgfmathsetmacro{\j}{\L[\y-1]}
    \foreach \i in {1,...,\j}{
        \node at (\i,\y) {\LS[\y,\i]};
        }
}

\node at (2,4) {$(T^{(4)},S^{(4)})$};

\tikzset{every node/.style={inner sep=-4pt}}
\begin{scope}[on background layer]
\tikzset{every path/.style={line width = 7pt,color=black,line cap=round,opacity=.15,rounded corners}}
\draw (3,1)--(1,1)--(1,3);
\draw (3,2)--(2,2);
\draw (2-.25,3)--(2+.25,3);
\end{scope}

\end{tikzpicture}
\ \begin{minipage}{.05\textwidth}\vspace{-2.8cm}$\xrightarrow{\theta}$ \end{minipage}\ 
\begin{tikzpicture}[yscale=-1,scale=.55]

\def\L{{2,4,2}}
\pgfmathsetmacro{\len}{dim(\L)}

\foreach \y in {1,...,\len}{
    \pgfmathsetmacro{\j}{\L[\y-1]}
    \foreach \i in {1,...,\j}{
        \draw (\i-.5,\y-.5) rectangle (\i+.5,\y+.5);
        }
}
\setsepchar{\\/,/ }
\readlist\LS{1,1\\2,2,4,4\\3,3}
\pgfmathsetmacro{\lens}{\listlen\LS[]}
\foreach \y in {1,...,\lens}{
    \pgfmathsetmacro{\j}{\L[\y-1]}
    \foreach \i in {1,...,\j}{
        \node at (\i,\y) {\LS[\y,\i]};
        }
}

\node at (2,4) {$(T^{(5)},S^{(5)})$};

\tikzset{every node/.style={inner sep=-4pt}}
\begin{scope}[on background layer]
\tikzset{every path/.style={line width = 7pt,color=black,line cap=round,opacity=.15,rounded corners}}
\draw (2,1)--(1,1);
\draw (4,2)--(1,2)--(1,3);
\draw (2-.25,3)--(2+.25,3);
\end{scope}

\end{tikzpicture}
$ 
\medskip

$
\ \begin{minipage}{.05\textwidth}\vspace{-3.4cm}$\xrightarrow{\psi}$ \end{minipage}\ \begin{tikzpicture}[yscale=-1,scale=.55]

\def\L{{2,3,2,1}}
\pgfmathsetmacro{\len}{dim(\L)}

\foreach \y in {1,...,\len}{
    \pgfmathsetmacro{\j}{\L[\y-1]}
    \foreach \i in {1,...,\j}{
        \draw (\i-.5,\y-.5) rectangle (\i+.5,\y+.5);
        }
}
\setsepchar{\\/,/ }
\readlist\LS{1,1\\2,2,\circl{4}\\3,3\\4}
\pgfmathsetmacro{\lens}{\listlen\LS[]}
\foreach \y in {1,...,\lens}{
    \pgfmathsetmacro{\j}{\L[\y-1]}
    \foreach \i in {1,...,\j}{
        \node at (\i,\y) {\LS[\y,\i]};
        }
}

\node at (2,5) {$(T^{(6)},S^{(6)})$};

\tikzset{every node/.style={inner sep=-4pt}}
\begin{scope}[on background layer]
\tikzset{every path/.style={line width = 7pt,color=black,line cap=round,opacity=.15,rounded corners}}
\draw (2,1)--(1,1);
\draw (3,2)--(1,2)--(1,4);
\draw (2-.25,3)--(2+.25,3);
\draw (2-.25,4)--(2+.25,4);
\end{scope}

\end{tikzpicture}\ \begin{minipage}{.05\textwidth}\vspace{-3.4cm}$\xrightarrow{\theta}$ \end{minipage}\ \begin{tikzpicture}[yscale=-1,scale=.55]

\def\L{{2,3,2,1}}
\pgfmathsetmacro{\len}{dim(\L)}

\foreach \y in {1,...,\len}{
    \pgfmathsetmacro{\j}{\L[\y-1]}
    \foreach \i in {1,...,\j}{
        \draw (\i-.5,\y-.5) rectangle (\i+.5,\y+.5);
        }
}
\setsepchar{\\/,/ }
\readlist\LS{1,1\\2,2,4\\3,3\\4}
\pgfmathsetmacro{\lens}{\listlen\LS[]}
\foreach \y in {1,...,\lens}{
    \pgfmathsetmacro{\j}{\L[\y-1]}
    \foreach \i in {1,...,\j}{
        \node at (\i,\y) {\LS[\y,\i]};
        }
}

\node at (2,5) {$(T^{(7)},S^{(7)})$};

\tikzset{every node/.style={inner sep=-4pt}}
\begin{scope}[on background layer]
\tikzset{every path/.style={line width = 7pt,color=black,line cap=round,opacity=.15,rounded corners}}
\draw (2,1)--(1,1)--(1,4);
\draw (3,2)--(2,2);
\draw (2-.25,3)--(2+.25,3);
\draw (2-.25,4)--(2+.25,4);
\end{scope}

\end{tikzpicture}\ \begin{minipage}{.05\textwidth}\vspace{-3.4cm}$\xrightarrow{\psi}$ \end{minipage}\ \begin{tikzpicture}[yscale=-1,scale=.55]

\def\L{{2,2,3,1}}
\pgfmathsetmacro{\len}{dim(\L)}

\foreach \y in {1,...,\len}{
    \pgfmathsetmacro{\j}{\L[\y-1]}
    \foreach \i in {1,...,\j}{
        \draw (\i-.5,\y-.5) rectangle (\i+.5,\y+.5);
        }
}
\setsepchar{\\/,/ }
\readlist\LS{1,1\\2,2\\3,3,\circl{4}\\4}
\pgfmathsetmacro{\lens}{\listlen\LS[]}
\foreach \y in {1,...,\lens}{
    \pgfmathsetmacro{\j}{\L[\y-1]}
    \foreach \i in {1,...,\j}{
        \node at (\i,\y) {\LS[\y,\i]};
        }
}

\node at (2,5) {$(T^{(8)},S^{(8)})$};

\tikzset{every node/.style={inner sep=-4pt}}
\begin{scope}[on background layer]
\tikzset{every path/.style={line width = 7pt,color=black,line cap=round,opacity=.15,rounded corners}}
\draw (2,1)--(1,1)--(1,4);
\draw (2,2)--(2,3);
\draw (3-.25,3)--(3+.25,3);
\draw (2-.25,4)--(2+.25,4);
\end{scope}

\end{tikzpicture}
\ \begin{minipage}{.05\textwidth}\vspace{-3.4cm}$\xrightarrow{\theta}$ \end{minipage}\ \begin{tikzpicture}[yscale=-1,scale=.55]

\def\L{{2,2,3,1}}
\pgfmathsetmacro{\len}{dim(\L)}

\foreach \y in {1,...,\len}{
    \pgfmathsetmacro{\j}{\L[\y-1]}
    \foreach \i in {1,...,\j}{
        \draw (\i-.5,\y-.5) rectangle (\i+.5,\y+.5);
        }
}
\setsepchar{\\/,/ }
\readlist\LS{1,1\\2,2\\3,3,4\\4}
\pgfmathsetmacro{\lens}{\listlen\LS[]}
\foreach \y in {1,...,\lens}{
    \pgfmathsetmacro{\j}{\L[\y-1]}
    \foreach \i in {1,...,\j}{
        \node at (\i,\y) {\LS[\y,\i]};
        }
}

\node at (2,5) {$(T^{(9)},S^{(9)})$};

\tikzset{every node/.style={inner sep=-4pt}}
\begin{scope}[on background layer]
\tikzset{every path/.style={line width = 7pt,color=black,line cap=round,opacity=.15,rounded corners}}
\draw (2,1)--(1,1)--(1,4);
\draw (2-.25,2)--(2+.25,2);
\draw (3,3)--(2,3);
\draw (2-.25,4)--(2+.25,4);
\end{scope}

\end{tikzpicture}
\ \begin{minipage}{.05\textwidth}\vspace{-3.4cm}$\xrightarrow{\psi}$ \end{minipage}\ \begin{tikzpicture}[yscale=-1,scale=.55]

\def\L{{2,2,2,2}}
\pgfmathsetmacro{\len}{dim(\L)}

\foreach \y in {1,...,\len}{
    \pgfmathsetmacro{\j}{\L[\y-1]}
    \foreach \i in {1,...,\j}{
        \draw (\i-.5,\y-.5) rectangle (\i+.5,\y+.5);
        }
}
\setsepchar{\\/,/ }
\readlist\LS{1,1\\2,2\\3,3\\4,4}
\pgfmathsetmacro{\lens}{\listlen\LS[]}
\foreach \y in {1,...,\lens}{
    \pgfmathsetmacro{\j}{\L[\y-1]}
    \foreach \i in {1,...,\j}{
        \node at (\i,\y) {\LS[\y,\i]};
        }
}

\node at (2,5) {$(T^{(10)},S^{(10)})$};

\tikzset{every node/.style={inner sep=-4pt}}
\begin{scope}[on background layer]
\tikzset{every path/.style={line width = 7pt,color=black,line cap=round,opacity=.15,rounded corners}}
\draw (2,1)--(1,1)--(1,4);
\draw (2-.25,2)--(2+.25,2);
\draw (2,3)--(2,4);
\draw (3-.25,4)--(3+.25,4);
\end{scope}

\end{tikzpicture}
$
    \caption{An example of {\Cref{alg: sym involution FULL}}}
    \label{fig:sym-involu-full-ex}
\end{figure}

\section{Connections to special rim hook tableaux}\label{section:connections}

In this section, we provide a bijection between special rim hook tableaux and tunnel hook coverings with nonnegative content. We do so by assigning a unique permutation to each special rim hook tableau in a manner similar to \Cref{Th:THC_and_Perms}.  This allows us to make comparisons between our algorithm and those of Loehr and Mendes~\cite{LoeMen06} and Sagan and Lee~\cite{SagLee03}.

\subsection{The permutation associated to a special rim hook tableau}

Following \cite{EgeRem90}, we define a rim hook $\rh$ to be a connected skew diagram (the set theoretic difference of two Ferrers diagrams of partition shapes) which does not contain
a $2 \times 2$ square of cells. A \emph{special rim hook tableau} of shape $\lambda$ is a set theoretic partition of the cells of a Ferrers diagram of shape $\lambda$ into rim hooks such that each rim hook contains a cell from the leftmost column of $\lambda.$ 

The \textit{initial cell} of a rim hook $\rh$ is the northeastern-most cell contained in $\rh$.  The \emph{terminal cell} of a rim hook $\rh$ is the southwestern-most cell contained in $\rh$. Recall that the $d$\textsuperscript{th} diagonal is the set of all cells of the form $\mathcal{L}_d=\{(d+k,1+k)\mid k\in \Z\}$. Note that $(x,y)$ and $(a,b)$ are on the same diagonal if and only if $x-y=a-b$. In particular, $(x,y)\in \mathcal{L}_d$ if and only if $d=x-y+1$.  This definition of diagonals allows them to be indexed by integers (either positive, zero, or negative).  However, while initial cells may occur in diagonals indexed by any integer, terminal cells always occur in diagonals indexed by positive integers since terminal cells always appear in the leftmost column. 

\begin{lem}{\label{lem:uniqueDiagonals}}
Each diagonal of a special rim hook tableau contains a maximum of one initial cell.
\end{lem}

\begin{proof}
Assume that diagonal $d$ contains $k$ cells, at least one of which is an initial cell.  Let $\rh_1, \rh_2, \hdots , \rh_k$ be the $k$ distinct special rim hooks passing through this diagonal, ordered so that $\rh_i$ passes through cell $(d+i-1,i)$.  Let $m$ be the maximum value such that $\rh_m$ has initial cell $(d+m-1,m)$ in diagonal $d$.  Then $(d+m-2,m)$ is in the diagram and cannot be in $\rh_m$.  Therefore $(d+m-2,m)$ must be in $\rh_{m-1}$.
This means $(d+m-2,m-1)$ is not the initial cell of $\rh_{m-1}$.  Continuing inductively, we see that none of the special rim hooks $\rh_1, \rh_2, \hdots , \rh_{m-1}$ can have its initial cell in diagonal $d$.
\end{proof}

\Cref{lem:uniqueDiagonals} allows us to define a function $\sigma$ associated to a special rim hook tableau based on the 
individual special rim hooks.  Intuitively, each special rim hook defines one value $\sigma_i$ and the remainder of the function will be determined solely from the shape of the special rim hook tableau.

\begin{defn}\label{def:findperm}
For a special rim hook tableau $R$ of shape $\lambda$, its \textit{permutation} $\sigma$, denoted $\permSRT(R)$, is defined by specifying $\sigma_i$ for each $i \in [\ell(\lambda)]$ as follows:
\begin{enumerate}
\item If the diagonal $i-\lambda_i+1$ does not contain any initial cells, set $\sigma_i=i-\lambda_i$.
\item Otherwise, let $\rh$ be the special rim hook whose initial cell is in diagonal $i-\lambda_i+1$.  Let $r$ be the row containing the terminal cell of $\rh$ and set $\sigma_i=r$.
\end{enumerate}
\end{defn}

See   \Cref{fig:SRHT-def} for an example.  It is not immediately obvious that the object $\permSRT(R)$ is in fact a permutation in $S_{\ell(\lambda)}$, so we prove this before showing that the map $R \mapsto \permSRT(R)$ is an injection.

\begin{figure}
    \centering
   \begin{center}
    \begin{tikzpicture}[yscale=-1,scale=.6]

\def\L{{8, 7, 7, 4,4,4, 2, 2, 2}}
\pgfmathsetmacro{\len}{dim(\L)}

\foreach \y in {1,...,\len}{
    \pgfmathsetmacro{\j}{\L[\y-1]}
    \foreach \i in {1,...,\j}{
        \draw (\i-.5,\y-.5) rectangle (\i+.5,\y+.5);
        }
}
\node at (12,5) { \text{$\displaystyle\sigma=\binom{\textcolor{blue}{1\ 2\ 3\ 4\ 5\ }6\ 7\ \textcolor{blue}{8\ } 9}{\textcolor{red}{1\ 6\ 4\ 3\ 9\ } 2\  5\ \textcolor{red}{8\ } 7}$}};

\tikzset{every node/.style={inner sep=-4pt,color=blue}}
\node (1i) at (8,1) {};
\node (2i) at (7,2) {};
\node (3i) at (6,2) {};
\node (4i) at (2,2) {};
\node (5i) at (4,5) {};
\node (8i) at (1,7) {};
\node[label={[label distance=.1cm]45:1}] at (8,1) {};
\node[label={[label distance=.1cm]45:2}] at (7,2) {};
\node[label={[label distance=.1cm]45:3}] at (7,3) {};
\node[label={[label distance=.1cm]45:4}] at (4,4) {};
\node[label={[label distance=.1cm]45:5}] at (4,5) {};
\node[label={[label distance=.1cm]45:6}] at (4,6) {};
\node[label={[label distance=.1cm]45:7}] at (2,7) {};
\node[label={[label distance=.1cm]45:8}] at (2,8) {};
\node[label={[label distance=.1cm]45:9}] at (2,9) {};

\tikzset{every node/.style={inner sep=-4pt,color=red}}
\node (1t) at (1,1) {1};
\node (2t) at (1,6) {6};
\node (3t) at (1,4) {4};
\node (4t) at (1,3) {3};
\node (5t) at (1,9) {9};
\node (8t) at (1,8) {8};
\begin{scope}[on background layer]
\tikzset{every path/.style={line width = 7pt,color=black,line cap=round,opacity=.15,rounded corners}}
\draw (1i)--(1t);
\draw (2i)--(7,3)--(4,3)--(4,4)--(3,4)--(3,5)--(1,5)--(2t);
\draw (3i)--(3,2)--(3,3)--(2,3)--(2,4)--(3t);
\draw (4i)--(1,2)--(4t);
\draw (5i)--(4,6)--(2,6)--(2,9)--(5t);
\draw (8i)--(8t);
\end{scope}
\tikzset{every path/.style={line width=.5pt,color=black}}
\draw (8+.5,1+.5)-- (8-.5,1-.5);
\draw (7+.5,2+.5)--(7-1.5,2-1.5);
\draw (7+.5,3+.5)--(7-2.5,3-2.5);
\draw (4+.5,4+.5)--(4-3.5,4-3.5);
\draw (4+.5,5+.5)--(4-3.5,5-3.5);
\draw (4+.5,6+.5)--(4-3.5,6-3.5);
\draw (2+.5,7+.5)--(2-1.5,7-1.5);
\draw (2+.5,8+.5)--(2-1.5,8-1.5);
\draw (2+.5,9+.5)--(2-1.5,9-1.5);
\tikzset{every node/.style={inner sep=-4pt,color=black}}
\node at (8+1.2,1+.75) {$\mathcal{L}_{-6}$};
\node at (7+1.2,2+.75) {$\mathcal{L}_{-4}$};
\node at (7+1.2,3+.75) {$\mathcal{L}_{-3}$};
\node at (4+1,4+.75) {$\mathcal{L}_{1}$};
\node at (4+1,5+.75) {$\mathcal{L}_{2}$};
\node at (4+1,6+.75) {$\mathcal{L}_{3}$};
\node at (2+1,7+.75) {$\mathcal{L}_{6}$};
\node at (2+1,8+.75) {$\mathcal{L}_{7}$};
\node at (2+1,9+.75) {$\mathcal{L}_{8}$};

\end{tikzpicture}
\end{center}
    \caption{The diagonals $\mathcal{L}_{i-\lambda_i+1}$ for $i\in [\ell(\lambda)]$ considered in \Cref{def:findperm} are drawn in black for this special rim hook tableau of shape $(8,7,7,4,4,4,2,2,2)$.  The integers in $\sigma$ are colored blue and red if Case (2) is applied and black if Case (1) is applied.  The blue 5 (top) and the red 9 (bottom) indicate that the special rim hook starting in diagonal $\mathcal{L}_{5-\lambda_5+1}$ ends in row 9 (Case 2).  The black 7 (top) and black 5 (bottom) indicate that there are no initial cells in diagonal $\mathcal{L}_{7-\lambda_7+1}$ and so $\sigma_7=7-\lambda_7=5$ (Case 1).}
    \label{fig:SRHT-def}
\end{figure}
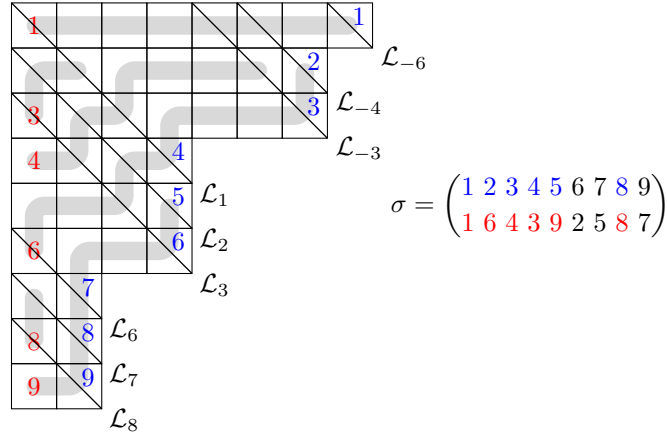

\begin{lem}
Suppose $R$ is a special rim hook tableau of shape $\lambda=(\lambda_1,\lambda_2,\dots,\lambda_\ell)$. Then $\sigma=\permSRT(R)$ is a permutation of $[\ell]$.
\end{lem}
\begin{proof}
Using the notation in \Cref{def:findperm}, we first check that $\sigma_i$ is in $[\ell]$ for every $1 \le i \le \ell$.  
It is clear that the values produced in Case (2) of \Cref{def:findperm} lie in $[\ell]$, so we must prove that $0 < i-\lambda_i \le \ell$ for all $i$ such that the diagonal $\mathcal{L}_{i-\lambda_i+1}$ contains no initial cells.

Fix an arbitrary $i \in [\ell]$ such that the diagonal $\mathcal{L}_{i-\lambda_i+1}$ does not contain any initial cells.  Since $i \le \ell$, it is clear that $i-\lambda_i \le \ell$.  Therefore we must show $0<i-\lambda_i$.  Assume not; that is, assume $i-\lambda_i \le 0$.  Then $i \le \lambda_i$, so there are $i$ cells in diagonal $\mathcal{L}_{i-\lambda_i+1}$ of the diagram.  Since special rim hooks move down and to the left, no two entries on the same diagonal can be covered by the same special rim hook.  Therefore there must be exactly $i$ special rim hooks passing through the diagonal $\mathcal{L}_{i-\lambda_i+1}$.  Let $\rh_1, \rh_2, \hdots , \rh_{i}$ be the special rim hooks passing through this diagonal.  Since $(1,\lambda_i-i+1)$ is not an initial cell, then $\rh_1$ must contain the cell $(1,\lambda_i-i+2)$.  But then since $(2,\lambda_i-i+2)$ cannot be an initial cell, $\rh_2$ includes the cell $(2, \lambda_i-i+3)$.  Continuing inductively, we arrive at the following contradiction.  We see that if none of $\rh_1,\rh_2, \hdots , \rh_{i-1}$ has initial cell in diagonal $\mathcal{L}_{i-\lambda_i+1}$, then $\rh_i$ must have its initial cell in this diagonal for otherwise $\rh_i$ would need to include $(i,\lambda_i-i+(i+1))=(i,\lambda_i+1)$, which is not in the diagram.  Hence by contradiction we must have $i-\lambda_i>0$ for all $i$ such that the diagonal $\mathcal{L}_{i-\lambda_i+1}$ contains no initial cells.

Now we must check the injectivity of $\sigma$.  Since $\lambda$ is a partition, $i - \lambda_i = i'-\lambda_{i'}$ if and only if $i=i'$.  Since no two special rim hooks can have terminal cells in the same row, it is clear that if $i$ and $j$ are distinct values such that diagonals $i-\lambda_i+1$ and $j-\lambda_j+1$ do contain initial cells, then $\sigma_i \not= \sigma_j$. 

Suppose there exists $i$ such that the diagonal $\mathcal{L}_{i-\lambda_i+1}$ contains no initial cells.  Note that we just proved $i - \lambda_i > 0$.  We will show that  row $i-\lambda_i$ cannot contain a terminal cell, which completes the proof that $\sigma_j \not= \sigma_i$ for any $j \not= i$ such that $1 \le j \le \ell$.  
The special rim hook tableau definition implies that there are $\lambda_i$  distinct special rim hooks passing through diagonal $\mathcal{L}_{i-\lambda_i+1}$.  Call these $\rh_1, \rh_2, \hdots , \rh_{\lambda_i}$ where $\rh_{\lambda_i}$ is the special rim hook passing through the cell $(i,\lambda_i)$.  Since $(i,\lambda_i)$ is the last cell in row $i$ and cannot be the initial cell for $\rh_{\lambda_i}$, the cell $(i-1,\lambda_i)$ must be contained in $\rh_{\lambda_i}$.  But then since $(i-1,\lambda_{i}-1) \in \mathcal{L}_{i-\lambda_i+1}$ cannot be the initial cell for $\rh_{\lambda_i-1}$, the cell $(i-2,\lambda_i-1)$ must be contained in $\rh_{\lambda_i-1}$.  Continuing inductively, 
we see that the cell $(i-\lambda_i,1)$ must be contained in $\rh_1$, which must terminate at or below row $i-\lambda_i+1$ since $(i-\lambda_i+1,1) \in \rh_1$.  Therefore there cannot be a terminal cell in row $i-\lambda_i$, as desired.

Therefore $\permSRT(R)$ is indeed a permutation of $[\ell]$ since $\sigma=\permSRT(R)$ is an injection whose image is contained in $[\ell]$.
\end{proof}

We are now ready to prove that distinct special rim hook tableaux map to distinct permutations.  Recall that the \textit{sign} of a special rim hook tableau $R$, denoted $\sgn(R)$, is $(-1)^k$ if $k$ is the number of rows crossed by special rim hooks in $R$. 
\begin{thm}\label{thm:SRT-and-Permutations}
Let $\lambda=(\lambda_1,\dots,\lambda_\ell)\vdash n$. The map $R\mapsto \permSRT(R)$ is an injection from the set of special rim hook tableaux of shape $\lambda$ into $S_\ell$ such that $\sgn(R)=\sgn(\permSRT(R)).$ 
\end{thm}
\begin{proof}
We prove injectivity by induction on $\ell$. Suppose two special rim hook tableaux $R$ and $T$ of shape $\lambda$ have permutations $\sigma$ and $\tau$, respectively, with $\sigma=\tau$. If $\ell=1$, then clearly $R=T$. Therefore assume $\ell>1$.  By construction, the special rim hook tableaux $R$ and $T$ have rim hooks $\rh$ and $\rh'$, respectively, each terminating in cell $(\ell,1)$.  Since $\sigma^{-1}(\ell)=\tau^{-1}(\ell)$, the initial cells of these special rim hooks must lie in the same diagonal.  However, only one cell from each diagonal can be in the outer rim of the shape and the cells in these rim hooks all lie on the outer rim, so they must have the same initial cell.  Therefore $\rh=\rh'$.

If we remove $\rh$ from $R$ and $T$, we obtain two special rim hook tableaux $R'$ and $T'$, respectively, on a shape $\lambda'$ such that $\ell(\lambda')<\ell(\lambda)$ and $\permSRT(R')=\permSRT(T')\in S_{\ell(\lambda')}$. 
In particular, suppose $r=\sigma^{-1}(\ell)$.  Then we claim that
\begin{equation}\label{eq:perm-induction}\permSRT(R')_j=\begin{cases}
        \sigma_j&\text{if $j<r$,}\\
        \sigma_{j+1}&\text{if $j\geq r$}
    \end{cases}\quad \text{and}\quad \permSRT(T')_j=\begin{cases}
        \tau_j&\text{if $j<r$,}\\
        \tau_{j+1}&\text{if $j\geq r$}.
    \end{cases}\end{equation}
To see this, first note that the only cells impacted by the removal of $\rh$ and $\rh'$ from $R$ and $T$, respectively,
are those in rows $r$ through $\ell$. 
For $j \ge r$, if $\mathfrak{h}$ is a special rim hook with initial cell in the diagonal $d=(j+1)-\lambda_{j+1}+1$, then $\mathfrak{h}$ still has its initial cell in this diagonal but the cell $(j+1,\lambda_{j+1})$ is no longer contained in the diagram. Therefore this special rim hook is now used to determine the value of $\permSRT(R')_j$.  Since its terminal cell is still in the same row, we have $\permSRT(R')_j = \sigma_{j+1}$.  If there was no initial cell in diagonal $d=(j+1)-\lambda_{j+1}+1$, then this diagonal still contains no initial cell in $R'$.  Since the leftmost cell in row $j$ of $\mathfrak{g}$ is immediately above the rightmost cell in row $j+1$ of $\rh$, $\lambda_j'=\lambda_{j+1}-1$.  Therefore $j-\lambda'_j = j- (\lambda_{j+1}-1) = j+1 - \lambda_{j+1}$.  So $\permSRT(R')_j=j-\lambda'_j=(j+1)-\lambda_{j+1} = \sigma_{j+1}$.  Since $T'$ is constructed in the same way as $R'$, the computation of $\permSRT(T')$ is the same.  Therefore $\permSRT(R') = \permSRT(T')$.  By induction, $R'=T'$, so $R=T$.

We likewise prove the statement about signs inductively. For $R$, $R'$, $\rh$, and $r$ as in the above paragraph, if $(r,\lambda_r)$ is the initial cell of $\rh$, then $\sigma_r=\ell$, so
    \[\sgn(R)=(-1)^{\ell-r}\sgn(R')\]
    since $\rh$ crosses $\ell-r$ rows. On the other hand, \begin{equation}\label{eq:permSRT-recurrence}\permSRT(R)=\permSRT(R')(\ell,\ell-1,\dots,r+1,r)\end{equation}     by \Cref{eq:perm-induction}, so 
    \[\sgn(\permSRT(R))=(-1)^{\ell-r}\sgn(\permSRT(R'))\]
as desired.
\end{proof}

This induction proof implies a more direct formula for the permutation associated to a special rim hook tableau.
\begin{prop}\label{prop:permutation-SRHT-cycle-decomp}
   Suppose a special rim hook tableau $R$ has rim hooks $\rh_1, \rh_2, \dots,\rh_k$ with initial cells in rows $x_1, x_2, \dots,x_k$ and terminal cells in rows $j_1, j_2, \dots,j_k$, respectively. Assume that we have ordered the rim hooks so that $j_1<j_2<\cdots<j_k.$ Then,
\[\permSRT(R)=(j_1,j_1-1,\dots,x_1+1,x_1)\cdots (j_k,j_k-1,\dots,x_k+1,x_k).\] 
\end{prop}
\begin{proof}
    Apply \Cref{eq:permSRT-recurrence} repeatedly.
\end{proof}


\begin{figure}[h]
\centering

\begin{tikzpicture}[yscale=-1,scale=.6]

\def\L{{8, 7, 7, 4,4,4, 2, 2, 2}}
\pgfmathsetmacro{\len}{dim(\L)}

\foreach \y in {1,...,\len}{
    \pgfmathsetmacro{\j}{\L[\y-1]}
    \foreach \i in {1,...,\j}{
        \draw (\i-.5,\y-.5) rectangle (\i+.5,\y+.5);
        }
}

\tikzset{every node/.style={inner sep=-4pt,color=blue}}
\node (1i) at (8,1) {1};
\node (2i) at (7,2) {2};
\node[label={[label distance=.1cm]45:3}] (3i) at (6,2) {};
\node[label={[label distance=.1cm]45:4}] (4i) at (2,2) {};
\node (5i) at (4,5) {5};
\node[label={[label distance=.1cm]45:8}] (8i) at (1,7) {};
\tikzset{every node/.style={inner sep=-4pt,color=red}}
\node (1t) at (1,1) {1};
\node (2t) at (1,6) {6};
\node (3t) at (1,4) {4};
\node (4t) at (1,3) {3};
\node (5t) at (1,9) {9};
\node (8t) at (1,8) {8};
\begin{scope}[on background layer]
\tikzset{every path/.style={line width = 7pt,color=black,line cap=round,opacity=.15,rounded corners}}
\draw (1i)--(1t);
\draw (2i)--(7,3)--(4,3)--(4,4)--(3,4)--(3,5)--(1,5)--(2t);
\draw (3i)--(3,2)--(3,3)--(2,3)--(2,4)--(3t);
\draw (4i)--(1,2)--(4t);
\draw (5i)--(4,6)--(2,6)--(2,9)--(5t);
\draw (8i)--(8t);
\end{scope}
\tikzset{every path/.style={line width=.5pt,color=blue}}
\draw (7+.5,3+.5)--($(3i)-(.5,.5)$);
\draw (4+.5,4+.5)--($(4i)-(.5,.5)$);
\draw (2+.5,8+.5)--($(8i)-(.5,.5)$);

\end{tikzpicture}
\hspace{1cm}
\begin{tikzpicture}[yscale=-1,scale=.6]

\def\L{{8, 7, 7, 4,4,4, 2, 2, 2}}
\pgfmathsetmacro{\len}{dim(\L)}

\foreach \y in {1,...,\len}{
    \pgfmathsetmacro{\j}{\L[\y-1]}
    \foreach \i in {1,...,\j}{
        \draw (\i-.5,\y-.5) rectangle (\i+.5,\y+.5);
        }
}

\tikzset{every node/.style={inner sep=-4pt,color=blue}}
\node (1i) at (8,1) {1};
\node (2i) at (7,2) {2};
\node (3i) at (7,3) {3};
\node (4i) at (4,4) {4};
\node (5i) at (4,5) {5};
\node (8i) at (2,8) {8};

\tikzset{every node/.style={inner sep=-4pt,color=red}}
\node (1t) at (1,1) {1};
\node (2t) at (1,6) {6};
\node[label={[label distance=.1cm]45:4}] (3t) at (2,5) { };
\node[label={[label distance=.1cm]45:3}] (4t) at (3,5) { };
\node (5t) at (1,9) {9};
\node[label={[label distance=.1cm]45:8}] (8t) at (2,9) { };
\begin{scope}[on background layer]
\tikzset{every path/.style={line width = 7pt,color=black,line cap=round,opacity=.15,rounded corners}}
\draw (1i)--(1t);
\draw (2i)--(1,2)--(2t);
\draw (3i)--(2,3)--(3t);
\draw (4i)--(3,4)--(4t);
\draw (5i)--(4,6)--(2,6)--(2,7)--(1,7)--(5t);
\draw (8i)--(8t);
\end{scope}
\tikzset{every path/.style={line width=.5pt,color=red}}
\draw (1-.5,3-.5)--($(4t)+(.5,.5)$);
\draw (1-.5,4-.5)--($(3t)+(.5,.5)$);
\draw (1-.5,8-.5)--($(8t)+(.5,.5)$);

\end{tikzpicture}
\[\binom{\textcolor{blue}{1\ 2\ 3\ 4\ 5\ }6\ 7\ \textcolor{blue}{8\ } 9}{\textcolor{red}{1\ 6\ 4\ 3\ 9\ } 2\  5\ \textcolor{red}{8\ } 7}\]
    \caption{The left diagram is a special rim hook tableau with additional markings to help compute the permutation (in two-line notation). The right diagram is the tunnel hook covering of this shape with the same permutation. Note that we have not drawn the tunnel hooks with $\Delta_i=0.$}\label{fig:SRT permutation}
\end{figure}
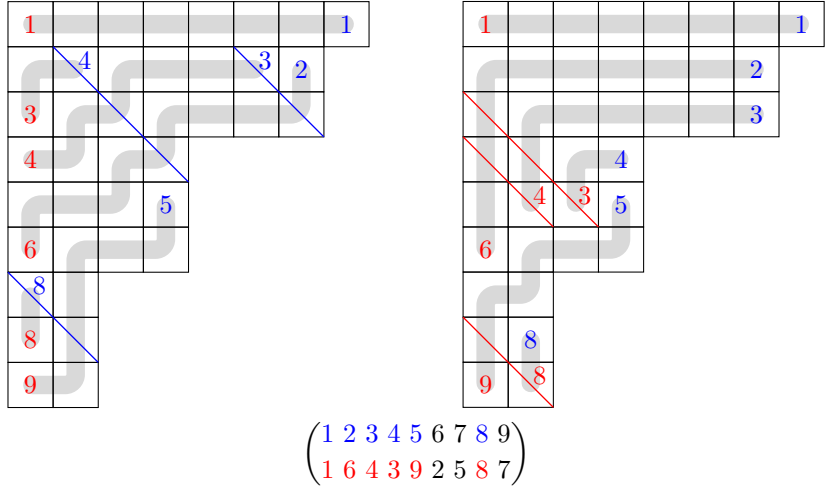

\begin{ex}
For the special rim hook tableau $R$ in \Cref{fig:SRT permutation} (left), the rim hooks $\rh_1,\dots,\rh_6$ have terminal cells in rows $\textcolor{red}{1,3,4,6,8,9}$ and initial cells in rows $1,2,2,2,7,5$ respectively.  Hence, we have 
\[\permSRT(R)=(\textcolor{red}{1})(\textcolor{red}{3},2)(\textcolor{red}{4},3,2)(\textcolor{red}{6},5,4,3,2)(\textcolor{red}{8},7)(\textcolor{red}{9},8,7,6,5).\]
\end{ex}

For tunnel hook coverings there is a similar construction.  Recall that the tunnel hook $\h(i,\tau_i)$ has initial cell in row $i$ and terminal cell $\tau_i$.

\begin{prop}\label{prop:permutation-THC-cycle-decomp}
Suppose a tunnel hook covering $T$ has tunnel hooks $\h(1,\tau_1), \allowbreak \h(2,\tau_2), \dots, \h(\ell,\tau_\ell)$ with terminal cells in rows $j_1,j_2,\cdots,j_\ell$, respectively. Then,
\[\perm(T)=(j_1,j_1-1,\dots,2,1)(j_2,j_2-1,\dots,3,2)\cdots (j_{\ell-1},\ell-1)(\ell).\] 
\end{prop}

See Appendix~\ref{appendix:proofs-of-perm-comps} for the proof.

\begin{ex}
For the tunnel hook covering in \Cref{fig:SRT permutation} (right), the tunnel hooks \[\mathfrak{h}(\textcolor{blue}{1},\tau_1),\ \mathfrak{h}(\textcolor{blue}{2},\tau_2),\ \mathfrak{h}(\textcolor{blue}{3},\tau_3),\ \mathfrak{h}(\textcolor{blue}{4},\tau_4),\ \mathfrak{h}(\textcolor{blue}{5},\tau_5),\ \mathfrak{h}(\textcolor{blue}{6},\tau_6),\ \mathfrak{h}(\textcolor{blue}{7},\tau_7),\ \mathfrak{h}(\textcolor{blue}{8},\tau_8),\ \mathfrak{h}(\textcolor{blue}{9},\tau_9)\] have terminal cells in rows $1,6,5,5,9,6,7,9,9$, respectively. Hence, we have 
\[\perm(T)=(\textcolor{blue}{1})(6,5,4,3,\textcolor{blue}{2})(5,4,\textcolor{blue}{3})(5,\textcolor{blue}{4})(9,8,7,6,\textcolor{blue}{5})(\textcolor{blue}{6})(\textcolor{blue}{7})(9,\textcolor{blue}{8})(\textcolor{blue}{9}).\]
\end{ex}

We sketch an alternate and exceedingly simple way to construct the permutation cycles appearing in Propositions \ref{prop:permutation-SRHT-cycle-decomp} and \ref{prop:permutation-THC-cycle-decomp}. First, for a tunnel hook covering $T$, label the cells covered by $T$ with their row numbers (including the cells outside the diagram). For example: 
    \[\begin{tikzpicture}[yscale=-1,scale=.55]

\def\L{{2,3,2,1}}
\pgfmathsetmacro{\len}{dim(\L)}

\foreach \y in {1,...,\len}{
    \pgfmathsetmacro{\j}{\L[\y-1]}
    \foreach \i in {1,...,\j}{
        \draw (\i-.5,\y-.5) rectangle (\i+.5,\y+.5);
        }
}
\node at (2,4) {4};
\node at (3,3) {3};
\setsepchar{\\/,/ }
\readlist\LS{1,1\\2,2,2\\3,3\\4}
\pgfmathsetmacro{\lens}{\listlen\LS[]}
\foreach \y in {1,...,\lens}{
    \pgfmathsetmacro{\j}{\L[\y-1]}
    \foreach \i in {1,...,\j}{
        \node at (\i,\y) {\LS[\y,\i]};
        }
}

\tikzset{every node/.style={inner sep=-4pt}}
\begin{scope}[on background layer]
\tikzset{every path/.style={line width = 7pt,color=black,line cap=round,opacity=.15,rounded corners}}
\draw (2,1)--(1,1)--(1,2);
\draw (3,2)--(2,2)--(2,3)--(1,3)--(1,4);
\draw (3-.25,3)--(3+.25,3);
\draw (2-.25,4)--(2+.25,4);
\end{scope}

\end{tikzpicture}\]

Create a cycle for each tunnel hook by reading the entries in decreasing order, deleting any repeated values.  Order these cycles by starting row to obtain the permutation.  For the example above, we have
\[\perm(T)=(2,1)(4,3,2)(3)(4)=[2,4,1,3].\]

Second, this also works for a special rim hook tableau by placing the rim hooks in the order of their intersection with the first column. For example:





\[\begin{tikzpicture}[yscale=-1,scale=.55]

\def\L{{4,3,3,3}}
\pgfmathsetmacro{\len}{dim(\L)}

\foreach \y in {1,...,\len}{
    \pgfmathsetmacro{\j}{\L[\y-1]}
    \foreach \i in {1,...,\j}{
        \draw (\i-.5,\y-.5) rectangle (\i+.5,\y+.5);
        }
}
\setsepchar{\\/,/ }
\readlist\LS{1,1,1,1\\2,2,2\\3,3,3\\4,4,4}
\pgfmathsetmacro{\lens}{\listlen\LS[]}
\foreach \y in {1,...,\lens}{
    \pgfmathsetmacro{\j}{\L[\y-1]}
    \foreach \i in {1,...,\j}{
        \node at (\i,\y) {\LS[\y,\i]};
        }
}

\tikzset{every node/.style={inner sep=-4pt}}
\begin{scope}[on background layer]
\tikzset{every path/.style={line width = 7pt,color=black,line cap=round,opacity=.15,rounded corners}}
\draw (4,1)--(2,1)--(2,2)--(1,2);
\draw (1-.25,1)--(1+.25,1);
\draw (1,3)--(2,3);
\draw (3,2)--(3,4)--(1,4);
\end{scope}

\end{tikzpicture}\]
\[\permSRT(R)=(1)(2,1)(3)(4,3,2)=[2,4,1,3]\]

Alternatively, the permutation of a tunnel hook covering can be computed inductively by keeping track of the terminal cells of
the portions of the tunnel hooks appearing in partial diagrams.  

\begin{prop}\label{prop:THC-perm-comp-row-by-row}
Let $T$ be a tunnel hook covering of $\alpha=(\alpha_1,\alpha_2,\dots,\alpha_\ell)$.
For $1 \le k \le \ell$, set $\hat T^k$ to be the tunnel hook covering of shape $\alpha^k=(\alpha_1,\alpha_2,\dots,\alpha_k)$ that is constructed by removing the cells of the diagram of shape $\alpha$
as well as the corresponding cells of the tunnel hooks in rows $k+1,\dots, \ell$.  If $d_1<d_2<\cdots < d_j$ are the
diagonals of the terminal cells in row $k$ of $\hat T^{k}$ then
\[\perm(\hat T^{k})=(d_1,d_2,\dots,d_j) \perm(\hat T^{k-1}).\]
\end{prop}

Although we leave the proof of this proposition to Appendix~\ref{appendix:proofs-of-perm-comps}, it relies on the fact that removing rows from the bottom of a tunnel hook covering still produces a valid tunnel hook covering of the resulting shape.  Similarly, removing columns from the right of a special rim hook tableau produces a valid special rim hook tableau of the resulting shape.

\begin{ex}
Consider the tunnel hook covering $T=\hat{T}^4$ below. We draw it along with $\hat{T}^3,\hat{T}^2,$ and $\hat{T}^1$ with the diagonals containing the terminal cells lying the final row.
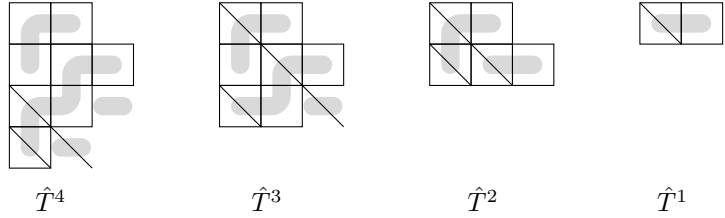
\[\begin{tikzpicture}[yscale=-1,scale=.55]
\def\L{{2,3,2,1}}
\pgfmathsetmacro{\len}{dim(\L)}


\foreach \y in {1,...,\len}{
    \pgfmathsetmacro{\j}{\L[\y-1]}
    \foreach \i in {1,...,\j}{
        \draw (\i-.5,\y-.5) rectangle (\i+.5,\y+.5);
        }
}
\setsepchar{\\/,/ }
\readlist\LS{1,1\\2,2,2\\3,3\\4}
\pgfmathsetmacro{\lens}{\listlen\LS[]}
\foreach \y in {1,...,\lens}{
    \pgfmathsetmacro{\j}{\L[\y-1]}
    \foreach \i in {1,...,\j}{
        \node at (\i,\y) {};
        }
}

\draw (1-.5,4-.5)--(1+.5,4+.5);
\draw (1-.5,3-.5)--(2+.5,4+.5);

\tikzset{every node/.style={inner sep=-4pt}}
\begin{scope}[on background layer]
\tikzset{every path/.style={line width = 7pt,color=black,line cap=round,opacity=.15,rounded corners}}
\draw (2,1)--(1,1)--(1,2);
\draw (3,2)--(2,2)--(2,3)--(1,3)--(1,4);
\draw (3-.25,3)--(3+.25,3);
\draw (2-.25,4)--(2+.25,4);
\end{scope}

\begin{scope}[shift={(2in,0)}]
\def\L{{2,3,2}}
\pgfmathsetmacro{\len}{dim(\L)}

\foreach \y in {1,...,\len}{
    \pgfmathsetmacro{\j}{\L[\y-1]}
    \foreach \i in {1,...,\j}{
        \draw (\i-.5,\y-.5) rectangle (\i+.5,\y+.5);
        }
}
\setsepchar{\\/,/ }
\readlist\LS{1,1\\2,2,2\\3,3}
\pgfmathsetmacro{\lens}{\listlen\LS[]}
\foreach \y in {1,...,\lens}{
    \pgfmathsetmacro{\j}{\L[\y-1]}
    \foreach \i in {1,...,\j}{
        \node at (\i,\y) {};
        }
}

\draw (1-.5,3-.5)--(1+.5,3+.5);
\draw (1-.5,1-.5)--(3+.5,3+.5);

\tikzset{every node/.style={inner sep=-4pt}}
\begin{scope}[on background layer]
\tikzset{every path/.style={line width = 7pt,color=black,line cap=round,opacity=.15,rounded corners}}
\draw (2,1)--(1,1)--(1,2);
\draw (3,2)--(2,2)--(2,3)--(1,3);
\draw (3-.25,3)--(3+.25,3);

\end{scope}

\end{scope}
\begin{scope}[shift={(4in,0)}]
\def\L{{2,3}}
\pgfmathsetmacro{\len}{dim(\L)}

\foreach \y in {1,...,\len}{
    \pgfmathsetmacro{\j}{\L[\y-1]}
    \foreach \i in {1,...,\j}{
        \draw (\i-.5,\y-.5) rectangle (\i+.5,\y+.5);
        }
}
\setsepchar{\\/,/ }
\readlist\LS{1,1\\2,2,2}
\pgfmathsetmacro{\lens}{\listlen\LS[]}
\foreach \y in {1,...,\lens}{
    \pgfmathsetmacro{\j}{\L[\y-1]}
    \foreach \i in {1,...,\j}{
        \node at (\i,\y) {};
        }
}

\draw (1-.5,2-.5)--(1+.5,2+.5);
\draw (1-.5,1-.5)--(2+.5,2+.5);

\tikzset{every node/.style={inner sep=-4pt}}
\begin{scope}[on background layer]
\tikzset{every path/.style={line width = 7pt,color=black,line cap=round,opacity=.15,rounded corners}}
\draw (2,1)--(1,1)--(1,2);
\draw (3,2)--(2,2);
\end{scope}

\end{scope}
\begin{scope}[shift={(6in,0)}]
\def\L{{2}}
\pgfmathsetmacro{\len}{dim(\L)}

\foreach \y in {1,...,\len}{
    \pgfmathsetmacro{\j}{\L[\y-1]}
    \foreach \i in {1,...,\j}{
        \draw (\i-.5,\y-.5) rectangle (\i+.5,\y+.5);
        }
}
\setsepchar{\\/,/ }
\readlist\LS{1,1}
\pgfmathsetmacro{\lens}{\listlen\LS[]}
\foreach \y in {1,...,\lens}{
    \pgfmathsetmacro{\j}{\L[\y-1]}
    \foreach \i in {1,...,\j}{
        \node at (\i,\y) {};
        }
}

\draw (1-.5,1-.5)--(1+.5,1+.5);

\tikzset{every node/.style={inner sep=-4pt}}
\begin{scope}[on background layer]
\tikzset{every path/.style={line width = 7pt,color=black,line cap=round,opacity=.15,rounded corners}}
\draw (2,1)--(1,1);
\end{scope}

\end{scope}
\end{tikzpicture}\]
\vskip -.25truein
\[
\begin{array}{ccccccccccccccccccccccccccccccccccccccc}
&&&&&\hat T^4&&&&&&&\hat T^3&&&&&&&\hat T^2&&&&&&\hat T^1&&&&&&&&
\end{array}
\]
\noindent

The diagonals of the terminal cells found in row $k$ of $\hat T^k$  are 
\[
\begin{array}{ccccccccccccccccccccccccccccccccccccccc}
&&&&&\{3,4\}&&&&&\{1,3\}&&&&&\{1,2\}&&&&&\{1\}&&&&&
\end{array}
\]
respectively. Multiplying the corresponding increasing cycle permutations of the above collections
gives
\[\perm(T)=
(3,4)(1,3)(1,2)(1)=[2,4,1,3].
\]
\end{ex}

\subsection{The relationship between the permutation and content of a special rim hook tableau}

We now want to describe which permutations one can obtain from special rim hook tableaux.  For a special rim hook $\rh$, let $|\rh|$ denote the number of cells in $\rh$. Define a weak composition $\Gamma(R)=(\Gamma_1(R),\Gamma_2(R),\dots,\Gamma_\ell(R))$ by setting
\[\Gamma_i(R)=\begin{cases}
    |\rh|&\text{if there is a $\rh\in R$ with initial cell on diagonal $i-\lambda_i+1$,}\\
    0&\text{otherwise,}
\end{cases}\]
for each $i\in [\ell]$.  Note that $\Gamma_i(R)$ is well-defined since diagonal $i-\lambda_i+1$ contains at most one initial cell by \Cref{lem:uniqueDiagonals}.  The content of $R$ is $\dec(\fl(\Gamma(R))).$

\begin{lem}\label{lem:SRT-sizes-perm}
    Let $R$ be a special rim hook tableau of shape $\lambda=(\lambda_1,\lambda_2,\dots,\lambda_\ell)\vdash n$.
    
    \begin{enumerate}
        \item For each  $i\in [\ell]$, we have $\Gamma_i(R)=\lambda_i-i+\permSRT(R)_i.$
    \item The image of $R\mapsto \permSRT(R)$ is the set of permutations $\sigma\in \S_\ell$ such that $\lambda_i-i+\sigma_i\geq 0$ for all $i\in [\ell].$
    \end{enumerate} 
    \end{lem}

\begin{proof}
For (1), let $\sigma=\permSRT(R)$ and suppose $\rh$ is a special rim hook of $R$ with initial cell $(x,y)$ in diagonal $\mathcal{L}_{i-\lambda_i+1}$ and terminal cell in row $r=\sigma_i$. Note that the terminal cell of $\rh$ is $(\sigma_i,1)$.  

Since the number of cells in a special rim hook equals one plus the number of south and west steps from the initial cell to the terminal cell, $|\rh|= | \sigma_i-x| + |1-y| + 1$.  Since $\sigma_i\geq x$ and $y\geq 1$, we have
    \[|\rh|= | \sigma_i-x| + |y-1| + 1= (\sigma_i-x)+(y-1)+1=y-x+\sigma_i.\]
Since $(x,y)$ is in diagonal $\mathcal{L}_{i-\lambda_i+1}$, we have $x-y+1=i-\lambda_i+1$ and hence $y-x=\lambda_i-i$.  Therefore $|\rh|=\lambda_i-i+\sigma_i$, as desired.

If there is no special rim hook of $R$ with an initial cell on diagonal $\mathcal{L}_{i-\lambda_i+1}$, then $\sigma_i=i-\lambda_i$. Hence, $\lambda_i-i+\sigma_i=0=\Gamma_i(R)$, as desired.

For (2), suppose $\sigma\in \S_\ell$ satisfies $\lambda_i-i+\sigma_i\geq 0$ for all $i \in [\ell].$ We can construct a special rim hook tableau of shape $\lambda$ with permutation $\sigma$ inductively as follows.  First observe that if $\ell=1$, then the claim is obvious. Now suppose $\ell>1$ and $\sigma_i=\ell$. If $\lambda_i-i+\sigma_i=0$, then $\lambda_i+\ell=i$. However, $i\leq \ell$ implies that $\lambda_i\leq 0$, which is a contradiction.  Thus, $\lambda_i-i+\sigma_i>0$ and therefore we may construct a special rim hook covering $\lambda_i-i+\sigma_i$ cells by starting with initial cell at $(i, \lambda_i)$ and moving down and to the left along the boundary of $\lambda$ until we reach the cell $(\ell,1)$.  

Let $\lambda'$ be the partition obtained by removing the cells of $\rh$. If $\sigma'=\sigma (i,i+1,\dots,\ell-1,\ell)$ (see \Cref{prop:permutation-SRHT-cycle-decomp}),  then $\sigma' \in S_k$  for some $k<\ell$ (because $\sigma_i=\ell$), so we can construct a special rim hook tableau $R'$ with permutation $\sigma'$.  To apply induction, we must also check that $\lambda_j'-j+\sigma_j'\geq 0$ for all $j\in [k]$.  If $j<i$, then $\sigma_j'=\sigma_j$ and $\lambda_j'=\lambda_j$, so $\lambda'_j-j+\sigma_j'=\lambda_j-j+\sigma_j\geq 0$. If $j\in \{i,i+1,\dots,k\}$, then $\sigma'_j=\sigma_{j+1}$ and $\lambda'_j\geq \lambda_{j+1}-1$. Hence, \[\lambda_j'-j+\sigma_j'\geq \lambda_{j+1}-1-j+\sigma_{j+1}=\lambda_{j+1}-(j+1)+\sigma_{j+1},\] which is at least 0 by assumption.  Then, as in the proof of \Cref{thm:SRT-and-Permutations}, attaching $\rh$ to $R'$ forms a special rim hook tableau $R$ of shape $\lambda$ with $\permSRT(R)=\sigma$.
\end{proof}

We combine \Cref{thm:SRT-and-Permutations} and \Cref{lem:SRT-sizes-perm} into a more succinct statement. For $\lambda\vdash n$, let $\SRT_{\lambda}$ be the set of special rim hook tableaux of shape $\lambda$.
\begin{cor}\label{cor:SRT-Perm-wt-preserving-bijection}
    Let $\lambda=(\lambda_1,\lambda_2,\dots,\lambda_\ell)\vdash n$. The map
    \[\permSRT:\SRT_{\lambda}\to \{\sigma\in S_\ell\mid \lambda_i-i+\sigma_i\geq 0\ \text{for all $i\in [\ell]$}\}\]
    is a bijection. Moreover, if $R\in \SRT_{\lambda}$ and $\sigma=\permSRT(R)$, then $\Gamma_i(R)=\lambda_i-i+\sigma_i$ for all $i\in [\ell].$
\end{cor}

Our development of the permutation of a special rim hook tableau leads us to a direct proof of E\u{g}ecio\u{g}lu and Remmel's result connecting special rim hook tableaux to the Jacobi--Trudi identity, which was originally proven via recurrence relations. For $R\in \SRT_\lambda$, let \[wt(R)=\prod_{i=1}^{\ell(\lambda)}h_{\Gamma_i(R)}.\]
\begin{cor}[{\cite[Theorem 1]{EgeRem90}}] 
    For $\lambda=(\lambda_1,\dots,\lambda_\ell)\vdash n$, 
    \[\det(h_{\lambda_i-i+j})_{1\leq i,j\leq \ell}=\sum_{R\in \SRT_\lambda}\sgn(R) wt(R).\]
\end{cor}
\begin{proof}
    Combine \Cref{cor:SRT-Perm-wt-preserving-bijection} with the observation that in
    \[\det(h_{\lambda_i-i+j})_{1\leq i,j\leq \ell}=\sum_{\sigma\in \S_{\ell}}\sgn(\sigma)\prod_{i=1}^{\ell}h_{\lambda_i-i+\sigma_i},\]
    the term $\prod_{i=1}^{\ell}h_{\lambda_i-i+\sigma_i}$ is nonzero if and only if $\lambda_i-i+\sigma_i\geq 0$ for all $i\in [\ell].$
\end{proof}

\subsection{Connections and comparisons}

Now we discuss connections between special rim hook tableaux and tunnel hook coverings. For $\lambda\vdash n$, let $\thc^{\geq 0}_{\lambda}$ be the set of tunnel hook coverings of shape $\lambda$ such that $\Delta_i(T)\geq 0$ for all $i\in [\ell(\lambda)].$ Define a map 
    \[\xi_{\lambda}:\SRT_{\lambda}\to \thc^{\geq 0}_{\lambda}\]
    given by setting $\xi_{\lambda}(R)=T$ if $T$ is the (unique) tunnel hook covering with shape $\lambda$ and permutation $\perm(T)=\permSRT(R)$. The map is well-defined by \Cref{Th:THC_and_Perms}. See \Cref{fig:SRT permutation}.

    \begin{thm}\label{thm:wt-preserving-bij-SRT-THC}
        For all partitions $\lambda$, the map $\xi_\lambda$ is a bijection. Moreover, for all $R\in \SRT_\lambda$, $\Gamma(R)=\Delta(\xi_\lambda(R)).$
    \end{thm}
    \begin{proof}
        Let $T$ be a tunnel hook covering of shape $\lambda$ with permutation $\perm(T)=\sigma$. By \Cref{delta-and-permutation}, $\Delta_i(T)\geq 0$ if and only if $\lambda_i+\sigma_i-i\geq 0$ for each $i\in [\ell(\lambda)].$ Hence, \Cref{cor:SRT-Perm-wt-preserving-bijection} and \Cref{Th:THC_and_Perms} show that $\xi_\lambda$ is a bijection such that if $R\in \SRT_\lambda$, then
        \[\Gamma_i(R)=\lambda_i-i+\sigma_i=\Delta_i(\xi_\lambda(R))\]
        for all $i\in [\ell(\lambda)]$.
    \end{proof}

For a partition $\lambda$, we can view an element of $\thc_\lambda^{\geq 0}$ as a set theoretic partition of the diagram of $\lambda$ by rim hooks by removing all tunnel hooks with $\Delta_i(T)=0$ (as we do in \Cref{fig:SRT permutation}). In this set up, it makes sense to ask which special rim hook tableaux are tunnel hook coverings (and vice versa). 

\begin{lem}\label{lem: intersection-of-srt-and-thc}
    Let $\lambda\vdash n$. Let $R$ be a set theoretic partition of the diagram of $\lambda$ by rim hooks $\rh_1,\rh_2,\dots,\rh_k$. Then $R\in \SRT_\lambda\cap \thc_\lambda^{\geq 0}$ if and only if for each $i\in [k]$, the initial cell of $\rh_i$ is at the end of its row and the terminal cell of $\rh_i$ is in the first column.
\end{lem}

\begin{proof}
If the initial cell of $\rh_i$ is at the end of its row for each $i$, then $R\in \thc_\lambda^{\geq 0}$. If the terminal cell of $\rh_i$ is in the first column for each $i$, then $R\in \SRT_\lambda$. The other inclusion is immediate from the definitions of special rim hook tableaux and tunnel hook coverings.
\end{proof}
We have two ways of computing the permutation of $R\in \SRT_\lambda\cap \thc_\lambda^{\geq 0}$. The following propositions shows that they agree.
\begin{prop}
    Let $\lambda\vdash n$. If $R\in \SRT_\lambda\cap \thc_\lambda^{\geq 0}$, then $\perm(R)=\permSRT(R)$ and $\xi_\lambda(R)=R.$
\end{prop}
\begin{proof}
    Let $R$ have rim hooks $\rh_1,\rh_2,\dots,\rh_k$ with initial cells in rows $x_1,x_2,\dots,x_k$ and terminal cells in rows $j_1,j_2,\dots,j_k$, respectively. Let $i\in [k]$. By \Cref{lem: intersection-of-srt-and-thc}, the initial cell of $\rh_i$ is at the end of row $x_i$  and so $\permSRT(R)_{x_i}=j_i$.  
    
    On the other hand, the terminal cell of $\rh_i$ is on diagonal $j_i$ by \Cref{lem: intersection-of-srt-and-thc}, so $\perm(R)_{x_i}=j_i$. Now suppose $s\not\in \{x_1,\dots,x_k\}$. If we view $R$ as a tunnel hook covering with tunnel hooks starting in each row, then $R$ has a tunnel hook starting in row $s$ with $\Delta_s(R)=0$ occupying the cell $(s,\lambda_s+1)$, which is on diagonal $\mathcal{L}_{s-\lambda_s}$. Hence, $\perm(R)_s=s-\lambda_s=\permSRT(R)_s$. Thus, we have proven that $\perm(R)_j=\permSRT(R)_j$ for all $j\in [\ell(\lambda)].$ Since $R$ itself is the unique tunnel hook covering of $\lambda$ with $\perm(R)=\permSRT(R)$, we have that $\xi_\lambda(R)=R$.
\end{proof}

We may now compare our algorithms to those of Loehr and Mendes \cite{LoeMen06} and Sagan and Lee \cite{SagLee03} by converting between special rim hook tableaux and tunnel hook coverings.
  Indeed, we show that \Cref{alg: sym involution FULL} is different from that of Loehr and Mendes \cite[Section 7.3]{LoeMen06}. 
  Consider the following pair (drawn using special rim hook tableaux and tunnel hook coverings, respectively):

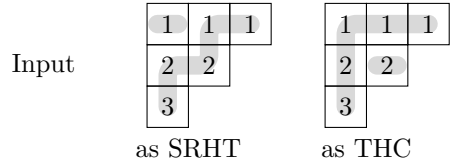
\[\begin{tikzpicture}[yscale=-1,scale=.55]

\def\L{{3,2,1}}
\pgfmathsetmacro{\len}{dim(\L)}

\foreach \y in {1,...,\len}{
    \pgfmathsetmacro{\j}{\L[\y-1]}
    \foreach \i in {1,...,\j}{
        \draw (\i-.5,\y-.5) rectangle (\i+.5,\y+.5);
        }
}
\setsepchar{\\/,/ }
\readlist\LS{1,1,1\\2,2\\3}
\pgfmathsetmacro{\lens}{\listlen\LS[]}
\foreach \y in {1,...,\lens}{
    \pgfmathsetmacro{\j}{\L[\y-1]}
    \foreach \i in {1,...,\j}{
        \node at (\i,\y) {\LS[\y,\i]};
        }
}

\tikzset{every node/.style={inner sep=-4pt}}
\begin{scope}[on background layer]
\tikzset{every path/.style={line width = 7pt,color=black,line cap=round,opacity=.15,rounded corners}}
\draw (3,1)--(2,1)--(2,2)--(1,2)--(1,3);
\draw (1+.25,1)--(1-.25,1);
\end{scope}
\node at (1.5,4) {as SRHT};
\node at (-2,2) {Input};

\end{tikzpicture}
\qquad 
\begin{tikzpicture}[yscale=-1,scale=.55]

\def\L{{3,2,1}}
\pgfmathsetmacro{\len}{dim(\L)}

\foreach \y in {1,...,\len}{
    \pgfmathsetmacro{\j}{\L[\y-1]}
    \foreach \i in {1,...,\j}{
        \draw (\i-.5,\y-.5) rectangle (\i+.5,\y+.5);
        }
}
\setsepchar{\\/,/ }
\readlist\LS{1,1,1\\2,2\\3}
\pgfmathsetmacro{\lens}{\listlen\LS[]}
\foreach \y in {1,...,\lens}{
    \pgfmathsetmacro{\j}{\L[\y-1]}
    \foreach \i in {1,...,\j}{
        \node at (\i,\y) {\LS[\y,\i]};
        }
}

\tikzset{every node/.style={inner sep=-4pt}}
\begin{scope}[on background layer]
\tikzset{every path/.style={line width = 7pt,color=black,line cap=round,opacity=.15,rounded corners}}
\draw (3,1)--(1,1)--(1,3);
\draw (2+.25,2)--(2-.25,2);
\end{scope}
\node at (1.5,4) {as THC};
\end{tikzpicture}
\]
The Jacobi--Trudi determinant corresponding to $\lambda=(3,2,1)$ is
  \[
  \det\begin{pmatrix}
h_3&h_4&h_5\\
h_1&h_2&h_3\\
h_{-1}&h_0&h_1
  \end{pmatrix}
\]
with $h_{-1}=0$ and $h_0=1.$
There is a single permutation $\sigma=[3,1,2]$ that produces the term $h_5h_1h_0$ in the determinant (even when allowing for commutativity). Hence,
there is a unique special rim hook tableau as well as a unique tunnel hook covering
of shape $\lambda$ and content $(5,1)$.  Therefore in this example there is only one possible SSYT-SRHT pair $(S,T')$ corresponding to our chosen SSYT-THC pair input $(S,T)$.  Thus, to show that the Loehr and Mendes algorithm is different from \Cref{alg: sym involution FULL}, it is sufficient to show that when we apply the Loehr--Mendes algorithm to $(S,T')$, the output is different from when we apply our algorithm to $(S,T)$.  In particular, we will show that the semistandard Young tableau produced by the Loehr--Mendes algorithm is different from that produced by \Cref{alg: sym involution FULL}.

Indeed, the Loehr and Mendes procedure \cite[p. 941]{LoeMen06} produces a pair consisting of a semistandard Young tableau and a special rim hook tableau which corresponds in our setting to the following pair:
\[
\begin{tikzpicture}[yscale=-1,scale=.55]

\def\L{{4,2}}
\pgfmathsetmacro{\len}{dim(\L)}

\foreach \y in {1,...,\len}{
    \pgfmathsetmacro{\j}{\L[\y-1]}
    \foreach \i in {1,...,\j}{
        \draw (\i-.5,\y-.5) rectangle (\i+.5,\y+.5);
        }
}
\setsepchar{\\/,/ }
\readlist\LS{1,1,1,2\\2,3}
\pgfmathsetmacro{\lens}{\listlen\LS[]}
\foreach \y in {1,...,\lens}{
    \pgfmathsetmacro{\j}{\L[\y-1]}
    \foreach \i in {1,...,\j}{
        \node at (\i,\y) {\LS[\y,\i]};
        }
}

\tikzset{every node/.style={inner sep=-4pt}}
\begin{scope}[on background layer]
\tikzset{every path/.style={line width = 7pt,color=black,line cap=round,opacity=.15,rounded corners}}
\draw (4,1)--(2,1)--(2,2)--(1,2);
\draw (1+.25,1)--(1-.25,1);
\end{scope}
\node at (2,3) {as SRHT};
\node at (-2,2) {LM Output};
\end{tikzpicture}\qquad 
\begin{tikzpicture}[yscale=-1,scale=.55]

\def\L{{4,2}}
\pgfmathsetmacro{\len}{dim(\L)}

\foreach \y in {1,...,\len}{
    \pgfmathsetmacro{\j}{\L[\y-1]}
    \foreach \i in {1,...,\j}{
        \draw (\i-.5,\y-.5) rectangle (\i+.5,\y+.5);
        }
}
\setsepchar{\\/,/ }
\readlist\LS{1,1,1,2\\2,3}
\pgfmathsetmacro{\lens}{\listlen\LS[]}
\foreach \y in {1,...,\lens}{
    \pgfmathsetmacro{\j}{\L[\y-1]}
    \foreach \i in {1,...,\j}{
        \node at (\i,\y) {\LS[\y,\i]};
        }
}

\tikzset{every node/.style={inner sep=-4pt}}
\begin{scope}[on background layer]
\tikzset{every path/.style={line width = 7pt,color=black,line cap=round,opacity=.15,rounded corners}}
\draw (4,1)--(1,1)--(1,2);
\draw (2+.25,2)--(2-.25,2);
\end{scope}
\node at (2,3) {as THC};
\end{tikzpicture}
\]
whereas \Cref{alg: sym involution FULL} produces:

\[
\begin{tikzpicture}[yscale=-1,scale=.55]

\def\L{{4,2}}
\pgfmathsetmacro{\len}{dim(\L)}

\foreach \y in {1,...,\len}{
    \pgfmathsetmacro{\j}{\L[\y-1]}
    \foreach \i in {1,...,\j}{
        \draw (\i-.5,\y-.5) rectangle (\i+.5,\y+.5);
        }
}
\setsepchar{\\/,/ }
\readlist\LS{1,1,1,3\\2,2}
\pgfmathsetmacro{\lens}{\listlen\LS[]}
\foreach \y in {1,...,\lens}{
    \pgfmathsetmacro{\j}{\L[\y-1]}
    \foreach \i in {1,...,\j}{
        \node at (\i,\y) {\LS[\y,\i]};
        }
}

\tikzset{every node/.style={inner sep=-4pt}}
\begin{scope}[on background layer]
\tikzset{every path/.style={line width = 7pt,color=black,line cap=round,opacity=.15,rounded corners}}
\draw (4,1)--(2,1)--(2,2)--(1,2);
\draw (1+.25,1)--(1-.25,1);
\end{scope}
\node at (2,3) {as SRHT};
\node at (-3.5,1.5) {\Cref{alg: sym involution FULL} Output};
\end{tikzpicture}
\qquad 
\begin{tikzpicture}[yscale=-1,scale=.55]

\def\L{{4,2}}
\pgfmathsetmacro{\len}{dim(\L)}

\foreach \y in {1,...,\len}{
    \pgfmathsetmacro{\j}{\L[\y-1]}
    \foreach \i in {1,...,\j}{
        \draw (\i-.5,\y-.5) rectangle (\i+.5,\y+.5);
        }
}
\setsepchar{\\/,/ }
\readlist\LS{1,1,1,3\\2,2}
\pgfmathsetmacro{\lens}{\listlen\LS[]}
\foreach \y in {1,...,\lens}{
    \pgfmathsetmacro{\j}{\L[\y-1]}
    \foreach \i in {1,...,\j}{
        \node at (\i,\y) {\LS[\y,\i]};
        }
}

\tikzset{every node/.style={inner sep=-4pt}}
\begin{scope}[on background layer]
\tikzset{every path/.style={line width = 7pt,color=black,line cap=round,opacity=.15,rounded corners}}
\draw (4,1)--(1,1)--(1,2);
\draw (2+.25,2)--(2-.25,2);
\end{scope}
\node at (2,3) {as THC};
\end{tikzpicture}
\]

The reader should note as well that there is a unique permutation that produces $h_{5}h_1$ in the Jacobi--Trudi determinant for the partition $(4,2)$.

We conjecture that when \Cref{alg: sym involution FULL} is restricted to standard Young tableaux, the algorithm is equivalent to the Sagan and Lee involution~\cite{SagLee03}.  We have tested this conjecture in SageMath \cite{sagemath} for all pairs of size less than or equal to $12$.\footnote{In order to fully implement the algorithm we added the following disambiguation to the TV move.  If the head of the active hook is in the first column, the root must be attached horizontally to the head even though the cell above the head is also a permissible cell.}

Consider the following example.

\[
\begin{tikzpicture}[yscale=-1,scale=.55]

\def\L{{5,3,3,3}}
\pgfmathsetmacro{\len}{dim(\L)}

\foreach \y in {1,...,\len}{
    \pgfmathsetmacro{\j}{\L[\y-1]}
    \foreach \i in {1,...,\j}{
        \draw (\i-.5,\y-.5) rectangle (\i+.5,\y+.5);
        }
}
\setsepchar{\\/,/ }
\readlist\LS{1,4,5,13,14\\2,6,8\\3,9,10\\7,11,12}
\pgfmathsetmacro{\lens}{\listlen\LS[]}
\foreach \y in {1,...,\lens}{
    \pgfmathsetmacro{\j}{\L[\y-1]}
    \foreach \i in {1,...,\j}{
        \node at (\i,\y) {\LS[\y,\i]};
        }
}

\tikzset{every node/.style={inner sep=-4pt}}
\begin{scope}[on background layer]
\tikzset{every path/.style={line width = 7pt,color=black,line cap=round,opacity=.15,rounded corners}}
\draw (5,1)--(1,1);
\draw (3,2)--(3,4)--(1,4);
\draw (2,2)--(2,3)--(1,3);
\draw (1-.25,2)--(1+.25,2);
\end{scope}
\node at (2,5) {as SRHT};
\node at (-1.5,2.5) {Input};
\end{tikzpicture}
\qquad 
\begin{tikzpicture}[yscale=-1,scale=.55]

\def\L{{5,3,3,3}}
\pgfmathsetmacro{\len}{dim(\L)}

\foreach \y in {1,...,\len}{
    \pgfmathsetmacro{\j}{\L[\y-1]}
    \foreach \i in {1,...,\j}{
        \draw (\i-.5,\y-.5) rectangle (\i+.5,\y+.5);
        }
}
\setsepchar{\\/,/ }
\readlist\LS{1,4,5,13,14\\2,6,8\\3,9,10\\7,11,12}
\pgfmathsetmacro{\lens}{\listlen\LS[]}
\foreach \y in {1,...,\lens}{
    \pgfmathsetmacro{\j}{\L[\y-1]}
    \foreach \i in {1,...,\j}{
        \node at (\i,\y) {\LS[\y,\i]};
        }
}

\tikzset{every node/.style={inner sep=-4pt}}
\begin{scope}[on background layer]
\tikzset{every path/.style={line width = 7pt,color=black,line cap=round,opacity=.15,rounded corners}}
\draw (5,1)--(1,1);
\draw (3,2)--(1,2)--(1,4);
\draw (3,3)--(2,3)--(2,4);
\draw (3-.25,4)--(3+.25,4);
\end{scope}
\node at (2,5) {as THC};
\end{tikzpicture}
\]
The Sagan and Lee procedure (rooting at $n=14$ and then applying the sequence of operations on overlapping rooted special rim hook tableaux on \cite[pg 153]{SagLee03}) and our \Cref{alg: sym involution FULL} both produce the following pair:
\[
\begin{tikzpicture}[yscale=-1,scale=.55]

\def\L{{4,4,3,3}}
\pgfmathsetmacro{\len}{dim(\L)}

\foreach \y in {1,...,\len}{
    \pgfmathsetmacro{\j}{\L[\y-1]}
    \foreach \i in {1,...,\j}{
        \draw (\i-.5,\y-.5) rectangle (\i+.5,\y+.5);
        }
}
\setsepchar{\\/,/ }
\readlist\LS{1,4,5,13\\2,6,8,14\\3,9,10\\7,11,12}
\pgfmathsetmacro{\lens}{\listlen\LS[]}
\foreach \y in {1,...,\lens}{
    \pgfmathsetmacro{\j}{\L[\y-1]}
    \foreach \i in {1,...,\j}{
        \node at (\i,\y) {\LS[\y,\i]};
        }
}

\tikzset{every node/.style={inner sep=-4pt}}
\begin{scope}[on background layer]
\tikzset{every path/.style={line width = 7pt,color=black,line cap=round,opacity=.15,rounded corners}}
\draw (4,1)--(2,1)--(2,2)--(1,2);
\draw (4,2)--(3,2)--(3,3)--(1,3);
\draw (3,4)--(1,4);
\draw (1-.25,1)--(1+.25,1);
\end{scope}
\node at (2,5) {as SRHT};
\node at (-1.5,2.5) {Output};
\end{tikzpicture}
\qquad 
\begin{tikzpicture}[yscale=-1,scale=.55]

\def\L{{4,4,3,3}}
\pgfmathsetmacro{\len}{dim(\L)}

\foreach \y in {1,...,\len}{
    \pgfmathsetmacro{\j}{\L[\y-1]}
    \foreach \i in {1,...,\j}{
        \draw (\i-.5,\y-.5) rectangle (\i+.5,\y+.5);
        }
}
\setsepchar{\\/,/ }
\readlist\LS{1,4,5,13\\2,6,8,14\\3,9,10\\7,11,12}
\pgfmathsetmacro{\lens}{\listlen\LS[]}
\foreach \y in {1,...,\lens}{
    \pgfmathsetmacro{\j}{\L[\y-1]}
    \foreach \i in {1,...,\j}{
        \node at (\i,\y) {\LS[\y,\i]};
        }
}
\node at (-1.5,1.5) {};

\tikzset{every node/.style={inner sep=-4pt}}
\begin{scope}[on background layer]
\tikzset{every path/.style={line width = 7pt,color=black,line cap=round,opacity=.15,rounded corners}}
\draw (4,1)--(1,1)--(1,2);
\draw (4,2)--(2,2)--(2,3)--(1,3);
\draw (3,4)--(1,4);
\draw (3-.25,3)--(3+.25,3);
\end{scope}
\node at (2,5) {as THC};
\end{tikzpicture}
\]

\subsection{Conclusion and discussion}

We proved both $\Sym$ identities ($KK^{-1}=I$ and $K^{-1}K=I$) by first proving an analogous identity ($\K\K^{-1}=I$ and $\K^{-1}\K=I$) in $\NSym$ and then reducing to $\Sym$.  The proof of the second identity is an instance of the Garsia--Milne involution principle, relying heavily on the fact that every semistandard Young tableau is also an immaculate tableau. In particular, we could abstract our proof as follows:
\begin{enumerate}
    \item Show that proving the symmetric function identity is equivalent to showing that a signed sum of pairs of combinatorial objects \[\sum_{(S,T)\in \mathcal{S}} \sgn(T)\] takes a particular value.
    \item Find a corresponding noncommutative symmetric function identity that is equivalent to a signed sum of pairs of other combinatorial objects \[\sum_{(S,T)\in \mathcal{N}}\sgn(T)\] taking on a particular value. Importantly, $\mathcal{S}\subseteq \mathcal{N}.$
    \item Produce a sign-reversing involution on $\mathcal{N}$ to prove the identity in NSym.
    \item Produce another involution on $\mathcal{N}\setminus \mathcal{S}$ and combine these two NSym involutions in the spirit of Garsia--Milne to produce a Sym involution.
\end{enumerate}

This paradigm of combining results from $\NSym$ together with the Garsia--Milne involution principle could also be applied to a wide array of symmetric function problems that naturally can be phrased as signed sums as in Step (1). To prove an identity in $\Sym$ (ideally with a known combinatorial expansion involving semistandard Young tableaux such as an expansion into the monomial basis), first identify the corresponding identity in $\NSym$ to complete Step (2).  Then construct a sign-reversing involution to prove the identity in $\NSym$ for Step (3). Finally, construct one more involution to solve the original problem in Step (4).

\section*{Acknowledgments}
We thank Nantel Bergeron, \:{Omer} E\u{g}ecio\u{g}lu, \'{A}lvaro Guti\'{e}rrez, Bruce Sagan, Mike Zabrocki for their helpful conversations. This work benefited from computations in SageMath \cite{sagemath}.

\bibliographystyle{amsplain}

\bibliography{Kostkabib}

\appendix
\section{Proofs of Propositions \ref{prop:permutation-THC-cycle-decomp} and \ref{prop:THC-perm-comp-row-by-row}}\label{appendix:proofs-of-perm-comps}

We include the proofs of several propositions in this appendix in order to streamline the body of the paper.

\begin{proprestatecycle}
    Suppose a tunnel hook covering $T$ has tunnel hooks $\h(1,\tau_1), \allowbreak \h(2,\tau_2), \dots, \h(\ell,\tau_\ell)$ with terminal cells in rows $j_1,j_2,\cdots,j_\ell$, respectively. Then,
\[\perm(T)=(j_1,j_1-1,\dots,2,1)(j_2,j_2-1,\dots,3,2)\cdots (j_{\ell-1},\ell-1)(\ell).\] 
\end{proprestatecycle}

\begin{proof}

Let $\sigma=\perm(T)$.  Recall that the \textit{Lehmer code} of a permutation $\sigma$ is the sequence $L(\sigma)=(L(\sigma)_1,L(\sigma)_2,\dots,L(\sigma)_{\ell})$, where $L(\sigma)_i=\#\{j>i\mid \sigma_j<\sigma_i\}.$  By \cite[Lemma 35]{AllenMason25EJC}, $L(\sigma)_i+1$ is the height of $\h(i,\tau_i)$ for each $i\in [\ell]$, so $L(\sigma)_i=j_i-i$.  We claim that the permutation \[\pi=(j_1,j_1-1,\dots,2,1)(j_2,j_2-1,\dots,3,2)\cdots (j_{\ell-1},\ell-1)(\ell)\] also has Lehmer code $L(\pi)_i=j_i-i$ for all $i \in [\ell]$.

To see this, for $k\in [\ell]$ set 
\[\pi^{(k)}=(j_k,j_k-1,\dots,k)\cdots (j_{\ell-1},\ell-1)(\ell).\]
Since $j_t\geq t$ for all $t\in [\ell]$, we have that $\pi^{(k)}(i)=i$ for $i<k$. Moreover, $\pi^{(k)}(i)\geq k$ if and only if $i\geq k$. Suppose $k\geq 2$. Since $\pi^{(k)}(i)=i$ for $i\leq k-1$ and $\pi^{(k-1)}=(j_{k-1},j_{k-1}-1,\dots,k-1) \pi^{(k)}$, we have 
\begin{equation}\label{eq:cycle-computation}\pi^{(k-1)}(i)=\begin{cases}
    \pi^{(k)}(i)&\text{if $\pi^{(k)}(i)>j_{k-1}$,}\\
    \pi^{(k)}(i)-1&\text{if $k\leq \pi^{(k)}(i)\leq j_{k-1}$,}\\
    j_{k-1}&\text{if $i=k-1$,}\\
    i&\text{if $i<k-1$}\\
\end{cases}\end{equation}
for each $i\in [\ell]$. Thus, if $j>i\geq k$, we have that $\pi^{(k-1)}(i)>\pi^{(k-1)}(j)$ if and only if $\pi^{(k)}(i)>\pi^{(k)}(j)$. Hence, $L(\pi^{(k)})_i=L(\pi^{(k-1)})_i$ for $i\geq k$. From this it follows that for all $k\in [\ell]$, \[L(\pi^{(k)})_k=L(\pi^{(k-1)})_k=\cdots=L(\pi^{(1)})_k=L(\pi)_k.\] 
By \Cref{eq:cycle-computation}, $\pi^{(k-1)}(k-1)=j_{k-1}$ and the numbers $k-1,\dots,j_{k-1}-1$ are in positions $k$ or greater of $\pi^{(k-1)}$ and the numbers $1,\dots,k-2$ are in positions $1,\dots,k-2$. Hence, $L(\pi)_{k-1}=L(\pi^{(k-1)})_{k-1}=j_{k-1}-(k-1)$.  Therefore, $L(\pi)_k=j_k-k$ for all $k\in [\ell].$

Since the Lehmer code of the permutation $\pi$ matches the Lehmer code of \(\perm(T)\), they are the same permutation and thus \[\perm(T)=(j_1,j_1-1,\dots,2,1)(j_2,j_2-1,\dots,3,2)\cdots (j_{\ell-1},\ell-1)(\ell).\qedhere\]
\end{proof}

\begin{proprestaterows}
Let $T$ be a tunnel hook covering of $\alpha=(\alpha_1,\alpha_2,\dots,\alpha_\ell)$.
For $1 \le k \le \ell$, set $\hat T^k$ to be the tunnel hook covering of shape $\alpha^k=(\alpha_1,\alpha_2,\dots,\alpha_k)$ that is constructed by removing the cells of the diagram of shape $\alpha$
as well as the corresponding cells of the tunnel hooks in rows $k+1,\dots, \ell$.  If $d_1<d_2<\cdots < d_j$ are the
diagonals of the terminal cells in row $k$ of $\hat T^{k}$ then
\[\perm(\hat T^{k})=(d_1,d_2,\dots,d_j) \perm(\hat T^{k-1}).\]
\end{proprestaterows}

\begin{proof}
    Let $k\geq 2$.  Set $\pi^{(k-1)}=\perm(\hat{T}^{k-1})$ and $\pi^{(k)}=\perm(\hat{T}^{k})$. Suppose $\h$ is a tunnel hook of $\hat{T}^{k}$ starting in row $i \leq k-1$ and ending in row $j\leq k-1$. Then, $\h$ is a tunnel hook of $\T^{k-1}$ and in particular starts and ends in the same cells. Since no two tunnel hooks end on the same diagonal, $\pi^{(k-1)}(i)\not\in \{d_1,\dots,d_j\}$ and hence \[\pi^{(k)}(i)=(d_1,\dots,d_j)\pi^{(k-1)}(i).\] 

Suppose $\h$ is a tunnel hook of $\hat{T}^{k}$ starting in row $i \leq k-1$ and ending in row $k$. Set $d=\pi^{(k-1)}(i)$. Then the terminal cell of $\h$ in $\hat{T}^{k-1}$ is $(k-1,k-d)$. Since $\h$ contains the cell $(k,k-d)$ in $\hat{T}^k$, we know that $(k,k-d+1)$ is the terminal cell of a tunnel hook in $\hat{T}^k$, and this is on diagonal $d$. Observe that $\h$ ends  in  $\hat{T}^k$ on some diagonal $\pi^{(k)}(i)=d'>d$ and there is no other tunnel hook ending on a diagonal strictly between $d$ and $d'$ in row $k$. Hence, the cycle of diagonals we construct is of the form $(d_1,\dots,d,d',\dots,d_j)$ and we can see that \[(d_1,\dots,d,d',\dots,d_j)\pi^{(k-1)}(i)=d'=\pi^{(k)}(i).\]\\  
Now suppose $\h$ is a tunnel hook of $\hat{T}^{k}$ starting in row $k$ and hence ending in row $k$. Then by construction, $\pi^{(k)}(k)=d_1$. Because $(k,1)$ is always the terminal cell of a tunnel hook of $\hat{T}^{k}$, we have that $d_j=k$. Since $\hat{T}^{k-1}$ has $k-1$ rows, $\pi^{(k-1)}(k)=k$. Thus, we have
\[(d_1,\dots,d_{j-1},k)\pi^{(k-1)}(k)=d_1=\pi^{(k)}(k).\qedhere\]
\end{proof}

\end{document}